\documentclass[psamsfonts,reqno]{amsart}

\usepackage{amssymb,amsfonts}
\usepackage{graphicx}
\usepackage{comment}
\usepackage{scalerel}[2016/12/29]
\usepackage[total={170mm,225mm},top=25mm,left=25mm]{geometry}
\usepackage[all,arc]{xy}
\usepackage{enumerate}
\usepackage{mathrsfs}
\usepackage[linktocpage=true]{hyperref}
\usepackage{bm}
\usepackage{pgfplots}
\usepackage{stmaryrd}
\usepackage{xcolor}
\usepackage[new]{old-arrows}
\usepackage{tikz-cd}
\usepackage{amsmath}
\usepackage{mathtools}
\hypersetup{
    colorlinks,
    linkcolor={blue!100!black},
    citecolor={blue!50!black},
    urlcolor={blue!80!black}
}
\DeclareFontFamily{U}{wncy}{}
\DeclareFontShape{U}{wncy}{m}{n}{<->wncyr10}{}
\DeclareSymbolFont{mcy}{U}{wncy}{m}{n}
\DeclareMathSymbol{\Sh}{\mathord}{mcy}{"58} 
\newcommand\restr[2]{{
  #1 
  |_{#2} 
  }}

\newtheorem{thm}{Theorem}[section]
\newtheorem{cor}[thm]{Corollary}
\newtheorem{prop}[thm]{Proposition}
\newtheorem{lem}[thm]{Lemma}

\newtheorem{claim}[thm]{Claim}
\newtheorem{assump}[thm]{Assumptions}
\newtheorem{theorem}{Theorem}

\theoremstyle{definition}
\newtheorem{defn}[thm]{Definition}

\theoremstyle{remark}
\newtheorem{rem}[thm]{Remark}

\DeclareMathOperator{\GL}{GL}

\DeclareMathOperator{\Gal}{Gal}
\DeclareMathOperator{\der}{der}
\DeclareMathOperator{\ad}{ad}
\DeclareMathOperator{\Ad}{Ad}

\DeclareMathOperator{\Lift}{Lift}
\DeclareMathOperator{\Def}{Def}
\DeclareMathOperator{\Hom}{Hom}
\DeclareMathOperator{\im}{im}

\DeclareMathOperator{\Spec}{Spec}
\DeclareMathOperator{\dR}{dR}

\newcommand\sqr{\scaleobj{0.65}{\square}}

\makeatletter
\let\c@equation\c@thm
\makeatother
\numberwithin{equation}{section}

\title{Forcing Ramification in the Relative Deformation Method}

\author{Stefan Nikoloski}
\address{Department of Mathematics, The Ohio State University, Columbus, OH 43210, USA}
\email{nikoloski.1@osu.edu}


\thanks{I would like to thank Najmuddin Fakhruddin and Chandrashekhar Khare for their useful feedback on an earlier draft of this paper. I would like to thank Stefan Patrikis for bringing this project to my attention and for helpful conversations about it. I also thank the referee for their helpful report. This work was supported by NSF grant DMS-2120325.}

\begin{document}

\begin{abstract}

    In this paper we generalize the forcing ramification argument of Khare-Ramakrishna to the setting of the lifting result by Fakhruddin-Khare-Patrikis. In particular, we show that in the relative deformation method the finite set of primes can be chosen so that the geometric characteristic $0$ lift will be ramified at each of those primes. Moreover, we show that the restrictions of this lift to the local absolute Galois groups will correspond to formally smooth points in the generic fibers of the local universal lifting rings at each prime. Eventually, we prove that the local formal smoothness at each of the primes implies the vanishing of the geometric Bloch-Kato Selmer group associated to the adjoint representation of that characteristic $0$ lift.
    
\end{abstract}

\maketitle
\tableofcontents

\section{Introduction}

The deformation theory of Galois representations traces its origin back to Mazur's fundamental paper \cite{Maz89}. In it under certain conditions Mazur establishes the existence of universal lifting rings that will govern the lifting properties of a residual representation $\overline{\rho}$. Ever since then, deformation theory has become an indispensable tool for proving modularity conjectures, specifically via various $R = \mathbb T$ theorems. A major step forward in this field was made by Ravi Ramakrishna in \cite{Ram99} and \cite{Ram02}. In those papers, he develops a purely Galois theoretic method which annihilates certain Selmer groups by allowing ramification at finitely many primes. This allows us to produce a geometric characteristic $0$ lift $\rho$ of a given two dimensional residual representation $\overline{\rho}$. At the time, this result has provided support to Serre's modularity conjecture by showing that the residual representations in question are indeed reductions of geometric characteristic $0$ representations. Eventually, Serre's modularity conjecture was proven by Khare and Wintenberger in \cite{Kha06}, \cite{KW09a}, \cite{KW09b} and \cite{KW09b} by a different method which is automorphic in nature, but also uses some Galois theoretic lifting results. In comparison to the method of Ramakrishna, the method of Khare and Wintenberger produces stronger results. Indeed, one of the main drawbacks of Ramakrishna's method is that we need to allow extra ramification at a finite set of prime and therefore the produced lifts will not be minimal in general. On the other hand, the ideas in Ramakrishna's method can be applied in more general settings and over time this approach has inspired many more results (\cite{KLR05}, \cite{HR08}, \cite{Pat16}, \cite{FKP22}), vastly generalizing the original one. The advancements of the method have allowed us to make conclusions about representations valued in reductive groups and not just $\GL_2$. Additionally, we can relax some of the constraints on the residual representations $\overline{\rho}$, mainly about the size of its image, and the local lifting rings, in particular by working around their singularities.

One of the peculiarities of Ramakrishna's method is that even though we may need to allow extra ramification at some primes, there is no guarantee that the lift $\rho$ will be ramified at those primes. The main difficulty lies in the fact that being ramified is not a deformation condition, so the standard methods of controlling the local properties of $\rho$ are not applicable in this case. The question of whether the characteristic $0$ lift $\rho$ will be ramified at the new primes is interesting from multiple points of view. On one hand, this question conjecturally can be seen as a Galois theoretic analogue to cases of the Ramanujan and the weight-monodromy conjectures. On the other hand, the same question is closely connected to formally smooth points in the generic fibers of the local lifting rings. The lifting methods allow us to exert great deal of control over the local behavior of the lift $\rho$ at the new set of primes, in particular over its inertial type at those primes. The restriction to the local absolute Galois group at such a prime will yield a point that lies in a prescribed irreducible component of the generic fiber of the local lifting ring. However, if the lift is not ramified at that prime the point will lie in the intersection of that irreducible component and the irreducible component corresponding to potentially unramified lifts. Hence, we can't expect that point to be formally smooth. Therefore, our initial question can be rephrased in terms of formally smooth points and we ask whether we can produce a geometric characteristic $0$ lift $\rho$, whose restrictions to the local absolute Galois groups will correspond to formally smooth points in the generic fibers of the local lifting rings. 

The question of whether we can produce a lift $\rho$ which will be ramified at all the extra primes was first considered by Khare and Ramakrishna in \cite{KR03} for representations valued in $\GL_2$ in the setting of Ramakrishna's lifting method. In that article, the authors develop a method by which every prime in the finite set at which $\rho$ is unramified will be replaced by two primes and then they produce a new characteristic $0$ lift which will be ramified at the two new primes. The original ideas of this forcing ramification argument, which were only applicable to two dimensional representations have been eventually generalized in \cite{Pat17} to representations valued in certain reductive groups. However, both of these results work only under the fundamental restrictions of the Ramakrishna's result. The goal of this paper is to generalize the forcing ramification argument to the setting of \cite{FKP21}, which provides a vast generalization of the earlier purely Galois theoretic lifting methods, relaxing the restrictions on the representations which we can lift. The main difficulty comes from the fact that the lifting argument in \cite{FKP21}, for which we give a quick overview in \S 3, relies on the vanishing of a relative Selmer group (see Definition \ref{3.11}) and not a Selmer group attached to the residual representation. One of the implications is that the na\"ive generalization of the methods in \cite{KR03} and \cite{Pat17} would not give us what we want. In particular, by applying the idea of Khare and Ramakrishna directly to the setting of \cite{FKP21} we can replace each prime at which the lift $\rho$ is unramified with a finite set of primes, but we can only force the ramification of the newly-produced lift at two of those primes and not all of them. Because of this we need to come up with a better way of choosing our new primes and hence modify the lifting method. The lifting method in \cite{FKP21} studies modulo $\varpi^M$ Selmer groups for a fixed $M$ and their reductions modulo $\varpi$. In \S $4.2$ we will consider larger and larger values for $M$, while also taking the reductions modulo higher powers of $\varpi$ and not just modulo $\varpi$. By varying these parameters and exploiting the relations between the relative Selmer groups we are able to replace each prime at which the lift $\rho$ is unramified with exactly two primes and not just a finite set of primes. We then produce a new $p$-adic lift which will be ramified at these two primes, as well as all the primes at which $\rho$ was already ramified.

We now let $F$ be a totally real number field and $S$ a finite set of finite places in $F$ containing the places above $p$. Let $\mathcal O$ be the ring of integers of $E$, a finite extension  of $\mathbb Q_p$ with residue field $k$. Let $G$ be a smooth group scheme over $\mathcal O$ such that $G^0$ is split connected reductive with $G/G^0$ finite \'etale of order prime to $p$. Let $\overline{\rho}:\Gamma_{F,S} \to G(k)$ be a continuous representation unramified away from $S$. We write $\overline{\rho}(\mathfrak g^{\der})$ for the $\mathcal O$-module $\mathfrak g^{\der} \otimes_{\mathcal O} k$ with $\Gamma_{F}$-action given by $\Ad \circ \overline{\rho}$. We now fix a lift $\mu:\Gamma_{F,S} \to G/G^{\der}(\mathcal O)$ of $\overline{\mu} \coloneqq \overline{\rho} \pmod{G^{\der}}$, which we will moreover assume is de Rham at primes above $p$. We note that as $G/G^{\der}$ has order prime to $p$ such a lift exists. We suppose that $\overline{\rho}$ satisfies the following conditions:

\vspace{5mm}

\begin{assump}
\label{1.1}

\text{ }

    \begin{itemize}
    \item $H^1(\Gal(K/F),\overline{\rho}(\mathfrak g^{\der})^*) = 0$, where $K = F(\overline{\rho},\mu_p)$.
    \item $\overline{\rho}(\mathfrak g^{\der})$ and $\overline{\rho}(\mathfrak g^{\der})^*$ are $\mathbb F_p[\Gamma_F]$-modules having no common $\mathbb F_p[\Gamma_F]$-subquotients, and neither contains the trivial representation.
    \item For all $\nu \in S$, there is some lift $\rho_\nu:\Gamma_{F_\nu} \to G(\mathcal O)$ with multiplier type $\mu$, of $\restr{\overline{\rho}}{\Gamma_{F_\nu}}$; and if $\nu \mid p$, this lift can be chosen to be de Rham and Hodge-Tate regular.
    \item $\overline{\rho}$ is odd, i.e. for every infinite place $\nu$ of $F$, $\dim H^0(\Gamma_{F_{\nu}},\overline{\rho}(\mathfrak g^{\der})) = \dim(\mathrm{Flag}_{G^{\der}})$.
\end{itemize}
    
\end{assump}

We emphasize that these are the same assumptions as in \cite{FKP21} and not in \cite{FKP22}. The conditions in the latter are slightly weaker. Our results do not directly generalize to this setting, since we are not able to establish the linear disjointness of fields that is crucial in the selection of the new primes. However, we will still use some of the technical improvements of the lifting method made in \cite{FKP22}. This will occur mostly in the minor changes we introduce during the doubling method. We prove the following result:

\begin{theorem}[See Theorem \ref{4.21}]
\label{A}

    Let $p \gg_G 0$ be a prime. Then there exists a finite set $T \supseteq S$, a finite extension $\mathcal O'$ of $\mathcal O$ and a geometric lift $\rho:\Gamma_{F,T} \to G(\mathcal O')$ of $\overline{\rho}$ with multiplier type $\mu$ such that for each prime $\nu \in T$ the local restriction $\restr{\rho}{\Gamma_{F_\nu}}$ corresponds to a formally smooth point of
    
    \begin{itemize}
        \item the generic fiber of the local lifting ring $R_{\restr{\overline{\rho}}{\Gamma_{F_\nu}}}^{\sqr,\mu}[1/\varpi]$ for $\nu \nmid p$.
        \item the inertial type $\tau$ and Hodge type $\textup{\textbf{v}}$ lifting ring $R^{\sqr,\mu,\tau,\textup{\textbf{v}}}_{\restr{\overline{\rho}}{\Gamma_{F_\nu}}}[1/\varpi]$ (see \cite[Proposition $3.0.12$]{Bal12} for the construction of this ring) for $\nu \mid p$.
    \end{itemize}

\end{theorem}

Using known modularity results, in certain situations we can show that the lift $\rho$ will be modular. In those cases we can give the statement of the theorem a different interpretation and also relate it to some other known results or conjectures. For instance, when dealing with modular lifts we can view the result of Theorem \ref{A} as some sort of a level raising result, since the set of primes where the lift ramifies divides the level of the newform associated to the lift $\rho$. This is especially true in light of Proposition \ref{4.22} which tells us that the set $T$ in Theorem \ref{A} can be chosen to properly include $S$. This means that we have indeed increased the set of primes where a lift of $\overline{\rho}$ ramifies, something which is not immediately guaranteed by the lifting results of \cite{FKP21}, nor by Theorem \ref{A}.

For a lift $\rho$ of $\overline{\rho}$ produced by Theorem \ref{A} the formal smoothness at each prime in $T$ will imply the finiteness of a certain Selmer group attached to the adjoint representation $\rho(\mathfrak g^{\der})$. Equivalently, the tangent space at $\rho$, when viewed as a point in a quotient of the global deformation ring will be finite. We spend the majority of \S 5 proving this implication. This can be seen as a version of a result implied by applying the Bloch-Kato conjecture (\cite[Conjecture $5.15$]{BK90}) to the adjoint representation $\rho(\mathfrak g^{\der})$. In our setting the Bloch-Kato conjecture (see \linebreak \cite[II, \S $3.4.5$]{FPR94} for the restatement in terms of L-functions) will predict the equality of the dimension of the geometric Bloch-Kato Selmer group $H_{g}^1(\Gamma_{F,T},\rho(\mathfrak g^{\der}))$ and the order of vanishing at $s=1$ of the L-function attached to the representation $\rho(\mathfrak g^{\der})$. If $\rho$ is pure of weight $w$, then $\rho(\mathfrak g^{\der})$ is pure of weight $0$. It is then conjectured that $L_T(\rho(\mathfrak g^{\der}),s)$ will not vanish at $s=1$ and so the Bloch-Kato conjecture predicts that $H_{g}^1(\Gamma_{F,T},\rho(\mathfrak g^{\der})) = 0$. We prove this prediction for $p$-adic lifts produced by the forcing ramification argument:

\begin{theorem}[See Theorem \ref{5.13}]
\label{B}

    Let $\rho$ be a lift of $\overline{\rho}$ produced by Theorem \ref{A}. Then, the geometric Bloch-Kato Selmer group $H^1_{g}(\Gamma_{F,T},\rho(\mathfrak g^{\der}))$ is trivial.

\end{theorem}

To prove this we first use the vanishing of the relative Selmer group which was used to produce the lift $\rho$ in order to reduce the problem to computing the sizes of the local lifting conditions. Then using the formal smoothness of $\rho$ at primes $\nu \in T$ we show that replacing the local conditions at primes not dividing $p$ with the whole local cohomology group $H^1(\Gamma_{F_\nu},\rho(\mathfrak g^{\der}))$ will not contribute to any extra global cohomology classes. On the other side, at primes $\nu \mid p$ we show the equality of the local Bloch-Kato groups $H^1_g(\Gamma_{F_\nu},\rho(\mathfrak g^{\der}))$ with the local lifting conditions we are using throughout the lifting and the forcing ramification methods.

\section{Notation}

We let $F$ be a totally real number field. We fix an algebraic closure $\overline{F}$ of $F$ and we write $\Gamma_F$ for the absolute Galois group $\Gal(\overline{F}/F)$. For each place $\nu \in S$ we fix an embedding $\overline{F} \hookrightarrow \overline{F_\nu}$ which gives us an embedding $\Gamma_{F_\nu} \hookrightarrow \Gamma_F$ of local absolute Galois groups into global absolute Galois groups. For a finite set of primes $S$ of $F$ we let $F(S)$ be the maximal extension of $F$ inside $\overline{F}$ that is unramified at primes outside of $S$. We then denote the Galois group $\Gal(F(S)/F)$ by $\Gamma_{F,S}$. In what follows we will often consider tamely ramified lifts, so for a prime $p$ we will work with $\Gal(F^{\mathrm{tame},p}_\nu/F_\nu)$, which is the maximal tamely ramified with \linebreak $p$-power ramification quotient of $\Gamma_{F_\nu}$. It's a well-known fact that this group is generated by a lift of the Frobenius $\sigma_\nu$ and a generator of the $p$-part of the tame inertia group $\tau_\nu$, whose interaction is given by the fundamental relation $\sigma_\nu\tau_\nu\sigma_\nu^{-1} = \tau_\nu^{N(\nu)}$.

Let $\mathcal O$ be the ring of integers of a finite extension $E$ of $\mathbb Q_p$, with uniformizer $\varpi$ and a maximal ideal $\mathfrak m = (\varpi)$. We write $k = \mathcal O/\mathfrak m$ for its residue field and we let $e$ be the ramification index of $\mathcal O/\mathbb Z_p$. We let $G$ be a smooth group scheme over $\mathcal O$ such that $G^0$ is split connected reductive and $G/G^0$ is finite \'etale of order prime to $p$. We write $\mathfrak g$ for its Lie algebra. We let $G^{\der}$ be the derived group of $G^0$ and we write $\mathfrak g^{\der}$ for its Lie algebra. We also let $Z_{G^0}$ and $\mathfrak z_G$ be the center of $G^0$ and its Lie algebra, respectively. We will also work with maximal tori $T$ inside of $G^0$ and we label their Lie algebras with $\mathfrak t$. We will often abuse the notation and write $\mathfrak g^{\der}$ for $\mathfrak g^{\der} \otimes_\mathcal O k$, but that will be obvious from the context. For a complete local $\mathcal O$-algebra $O'$ with residue field $k'$ we write $\widehat{G}(\mathcal O')$ for the kernel of the reduction map $G(\mathcal O') \to G(k')$.

For a split maximal torus $T$ of $G^0$ (over $\mathcal O$) and a root $\alpha \in \Phi(G^0,T)$, we let $U_\alpha \subset G^0$ denote the root subgroup that is the image of the root homomorphism (``exponential mapping") $u_\alpha: \mathfrak g_\alpha \to G$. The homomorphism $u_\alpha$ is a $T$-equivariant isomorphism $\mathfrak g_\alpha \to U_\alpha$. See \cite[Theorem $4.1.4$]{Con14}.

Given a group homomorphism $\rho:\Gamma \to H$ and an $H$-module $V$, we write $\rho(V)$ for the associated $\Gamma$-module. More specifically, we will use this whenever we have a continuous representation $\rho_s:\Gamma \to G(\mathcal O/\varpi^s)$ of some topological group $\Gamma$. We will then write $\rho_s(\mathfrak g^{\der})$ for the $\mathcal O$-module $\mathfrak g^{\der} \otimes_\mathcal O \mathcal O/\varpi^s$ with the action of $\Gamma$ given by $\Ad \circ \rho_s$. For $1 \le r \le s$, our choice of a uniformizer of $\mathcal O$ allows us to view $\mathcal O/\varpi^r$ as a submodule of $\mathcal O/\varpi^s$ by the multiplication by $\varpi^{s-r}$ map. More generally, for $\rho_r \coloneqq \rho_s \pmod{\varpi^r}$ we get an inclusion of $\rho_r(\mathfrak g^{\der})$ into $\rho_s(\mathfrak g^{\der})$.

We fix an isomorphism between $\mathbb Q_p/\mathbb Z_p(1)$ and $\mu_{p^{\infty}}(\overline{F})$, which amounts to choosing a compatible collection of $p$-power roots of unity. Using this, for a $\Gamma_F$-module $V$ we can identify the Tate dual $V^* = \Hom(V,\mu_{p^{\infty}}(\overline{F}))$ with $\Hom(V,\mathbb Q_p/\mathbb Z_p(1))$. Moreover, for a number field $K$ we denote by $K_\infty$ the extension $K(\mu_{p^\infty})$.

We assume that $p > 3$ is very good for $G^{\der}$ (\cite[\S $1.14$]{Car85}). We also assume that the canonical central isogeny $G^{\der} \times Z^0_G \to G^0$ has kernel of order coprime to $p$. Then, in particular we have a $G$-equivariant direct sum decomposition $\mathfrak g = \mathfrak g^{\der} \oplus \mathfrak z_G$. Also, by \cite[\S $1.16$]{Car85} we have a non-degenerate $G$-invariant trace form $\mathfrak g^{\der} \times \mathfrak g^{\der} \to k$. For a continuous homomorhism $\overline{\rho}:\Gamma_F \to G(k)$ we set $\overline{\mu} \coloneqq \overline{\rho} \pmod{G^{\der}}:\Gamma_F \to G/G^{\der}(k)$. We will fix a lift $\mu:\Gamma_F \to G/G^{\der}(\mathcal O)$ of $\overline{\mu}$, which is easily seen to exists since $G/G^{\der}(k)$ has an order coprime to $p$. We will be interested in lifts $\rho$ of $\overline{\rho}$ to $G(\mathcal O)$ of type $\mu$, i.e. lifts whose reduction modulo $G^{\der}$ is equal to $\mu$.

\section{The lifting method}

We start with a residual representation $\overline{\rho}:\Gamma_{F,S} \to G(k)$, which satisfies Assumptions \ref{1.1}. We denote by $D$ the largest integer such that $\mu_{p^D}$ is contained in $K = F(\overline{\rho},\mu_p)$. We may assume that $\overline{\rho}$ surjects onto $\pi_0(G)$. If not, we can replace $G$ by the preimage in $G$ of the image of $\overline{\rho}$ in $\pi_0(G)$, which will not affect the deformation theory of $\overline{\rho}$. Let $H \unlhd \Gamma_F$ be the subgroup consisting of elements $\sigma$ such that $\overline{\rho}(\sigma)$ lands inside $G^0(k)$. Let $\widetilde{F}$ be the fixed field of $H$. Then $\overline{\rho}$ induces an isomorphism $\Gal(\widetilde{F}/F) \simeq \pi_0(G)$.

As shown in \cite[Theorem $5.2$]{FKP22} we can produce a characteristic $0$ lift $\rho:\Gamma_{F,T} \to G(\mathcal O)$ of $\overline{\rho}$, which is unramified outside of a finite set $T$, containing $S$. Although we will go in more detail later on, we give a brief outline of the proof here. A key technical ingredient in the proof is a result from Lazard in \cite{Laz65}, which tells us that for all large enough $M$, the reduction map $\mathcal O/\varpi^M \to \mathcal O/\varpi$ induces a zero map between certain cohomology groups (see \cite[Corollary B$.2$]{FKP21}). Then, as in the first part of the paper, by using the so called doubling method, for any $n \in \mathbb N$ we can produce a lift $\rho_n:\Gamma_{F,S_n} \to G(\mathcal O/\varpi^n)$ of $\overline{\rho}$. We stop this process at a large enough $N$. The integer $M$ plays a role in the lower bound of $N$, among many other factors. As explained in the relative deformation section of the paper we produce a finite set of primes $Q$ that will annihilate an $M$-relative Selmer group (see Definition \ref{3.11}). Using this we produce the desired lift $\rho$. 

This method guarantees that $\rho$ will be unramified outside $S_N \cup Q$. On the other side, even though we have to allow ramification at primes in these sets for the method to work, it is possible that $\rho$ will not be ramified at some primes in $S_N \cup Q$. For most of these primes, especially the primes in $Q$, the question of whether $\rho$ is ramified at them is closely related to whether its restrictions to the local Galois group $\restr{\rho}{\Gamma_{F_\nu}}:\Gamma_{F_\nu} \to G(\mathcal O)$ correspond to formally smooth points of the generic fiber of the local lifting rings $R_{\restr{\overline{\rho}}{\Gamma_{F_\nu}}}^{\sqr,\mu}[1/\varpi]$. The goal of this paper is to answer this more general question and produce $\rho$ whose local restrictions will be formally smooth points in the fibers of the corresponding local lifting rings. A key ingredient in the proof will be the following formal smoothness criterion by Bellovin and Gee:

\begin{prop}[{\cite[Remark $3.3.7$]{BG19}}]

\label{3.1}

Suppose that $\nu \nmid p$. A point $x$ is a formally smooth point in $R_{\restr{\overline{\rho}}{\Gamma_{F_\nu}}}^{\sqr,\tau}[1/p]$ if and only if for the corresponding representation $\rho_x$ we have $H^0(\Gamma_{F_\nu},\rho_x(\mathfrak g^{\der})^*) = 0$.

\end{prop}

To achieve our goal we will take a closer look at the lifting method in \cite{FKP22} and see how the primes in $S_N \cup Q$ and the local conditions at them are chosen. Motivated by Proposition \ref{3.1} we will also make some technical adjustments along the way. As the primes in $S_N$ and $Q$ are obtained by using completely distinct methods, we will have to establish the smoothness at primes in those sets separately. This new lift will not necessarily satisfy the wanted properties at all the primes, yet. However, it will serve the purpose of an auxiliary lift and using it we will produce new lifts, possibly changing the initial ramifying set. Eventually, we will end up with the desired lift.

\subsection{Doubling method}

We will now look at the doubling method (\cite[\S$3$ and \S $4$]{FKP22}) into more details. As explained above, we will need to run the doubling method $N$ times and produce a lift $\rho_N: \Gamma_{F,S_N} \to G(\mathcal O/\varpi^N)$ of $\overline{\rho}$. After that we can continue with the relative deformation argument and produce a characteristic $0$ lift. For the time being we leave the integer $N$ unspecified. We will fix the integer $N$ in $\S 3.3$ ensuring that it is large enough, which is possible as the doubling method will allow us to produce a lift of $\overline{\rho}$ module any power of $\varpi$. The set $S_N$ will be a finite enlargement of $S$ by primes which are split in $K$. We now define a special type of local lifts and we will ask for the local lifts at each new prime we add in the ramifying set to be of that type.

\begin{defn}

    \label{3.2}

    Let $\nu$ be a prime split in $K$. Fix a split maximal torus $T$ of $G^0$ (over $\mathcal O$) and a root $\alpha \in \Phi(G^0,T)$. Define $\Lift_{\restr{\overline{\rho}}{\Gamma_{F_\nu}}}^{\mu,\alpha}(R)$ to be the set of lifts with multiplier $\mu$ which are $\widehat{G}(R)$-conjugate to a lift $\rho:\Gamma_{F_\nu} \to G(R)$ satisfying

    \begin{itemize}
        \item $\rho(\sigma_\nu) \in T(R)$.
        \item $\alpha(\rho(\sigma_\nu)) = q$ in $R^\times$.
        \item $\rho(\tau_\nu) \in U_\alpha(R)$.
    \end{itemize}
    
\end{defn}

\begin{defn}

    \label{LiftM}

    Let $\nu$ be a prime split in $K$. Fix a split maximal torus $T$ of $G^0$ (over $\mathcal O$) and a root $\alpha \in \Phi(G^0,T)$. For a fixed integer $M \ge 1$ let $\rho_M$ be a lift of $\overline{\rho}$ such that $\restr{\rho_M}{\Gamma_{F_\nu}} \in \Lift_{\restr{\overline{\rho}}{\Gamma_{F_\nu}}}^{\mu,\alpha}(\mathcal O/\varpi^M)$. Define $\Lift_{\restr{\rho_M}{\Gamma_{F_\nu}}}^{\mu,\alpha}(R)$ to be the set of lifts with multiplier $\mu$ which are $\ker(G(R) \to G(\mathcal O/\varpi^M))$-conjugate to a lift $\rho:\Gamma_{F_\nu} \to G(R)$ satisfying

    \begin{itemize}
        \item $\rho(\sigma_\nu) \in T(R)$.
        \item $\alpha(\rho(\sigma_\nu)) = q$ in $R^\times$.
        \item $\rho(\tau_\nu) \in U_\alpha(R)$.
    \end{itemize}
    
\end{defn}

\begin{rem}

    \label{3.3}

    These definitions differ slightly from \cite[Definition $3.1$]{FKP21} and \cite[Definition $3.4$]{FKP21}. However, the modification simplifies some of the computations in \cite[\S $3$]{FKP21} and the relevant results from that section still hold. Most of the details are worked out in \cite{KP25}.
    
\end{rem}

\begin{rem}

\label{3.6}

In the results where we prove the existence of specific group elements that satisfy certain properties we will make use of the exponential isomorphism of groups (see \cite[\S $3.5$]{Til96})

$$\exp: \mathfrak g \otimes_{\mathcal O} I \longrightarrow \ker(G(R) \to G(R/I))$$

\vspace{2 mm}

\noindent where $R$ is a Noetherian local $\mathcal O$-algebra with residue field $k$, $I$ is a square-zero ideal in $R$ and we view $\mathfrak g \otimes_{\mathcal O} I$ lying inside $\mathfrak g \otimes_{\mathcal O} R$. Moreover, this isomorphism is $G(R/I)$-equivariant, where the action on the source is the adjoint action after identifying $\mathfrak g \otimes_{\mathcal O} I \simeq \mathfrak g_{R/I} \otimes_{R/I} I$, while the action on the target is given by lifting to $G(R)$ and conjugating. The most common application will be in inductive proofs where $R = \mathcal O/\varpi^{n+1}$ and $I$ is the ideal consisting of $\varpi^n$ multiples inside $R$. We will use the generic notation $\exp$, leaving the integer $n$ implicit in it.
    
\end{rem}

We now state the main result of the doubling method together with the necessary modifications needed in the forcing ramification argument:

\begin{thm}[cf. {\cite[Theorem $4.4$]{FKP22}}]

    \label{3.4}

    Suppose $p \gg_G 0$. There exists a sequence of finite sets of primes of $F$, $S \subseteq S_1 \subseteq S_2 \subseteq \dots \subseteq S_n \subseteq \dots$, and for each $n \ge 1$ a lift $\rho_n:\Gamma_{F,S_n} \to G(\mathcal O/\varpi^n)$ of $\overline{\rho}$ with multiplier $\mu$, such that $\rho_n = \rho_{n+1} \pmod{\varpi^n}$ for all $n$. This system of lifts $(\rho_n)_{n \ge 1}$ satisfies the following properties:

    \begin{enumerate}
        \item If $\nu \in S_n \setminus S$ is ramified in $\rho_n$, then there is a split maximal torus and a root $(T_\nu,\alpha_\nu)$ such that $\rho_n(\sigma_\nu) \in T_\nu(\mathcal O/\varpi^n)$, $\alpha_\nu(\rho_n(\sigma_\nu)) \equiv N(\nu) \pmod {\varpi^n}$ and $\restr{\rho_n}{\Gamma_{F_\nu}} \in \Lift_{\overline{\rho}}^{\mu,\alpha_\nu}(\mathcal O/\varpi^n);$ in addition one of the following holds:

        \begin{enumerate}[(a)]
            \item For some $s \le eD$ (recall $e$ is the ramification index of $\mathcal O/\mathbb Z_p$ and $D$ is the largest integer such that $\mu_{p^D} \subset K$), $\rho_s(\tau_\nu)$ is a non-trivial element of $U_{\alpha_\nu}(\mathcal O/\varpi^s)$ and for all $n' \ge s$, $\restr{\rho_{n'}}{\Gamma_{F_\nu}}$ is a $\widehat{G}(\mathcal O)$-conjugate to the reduction modulo $\varpi^{n'}$ of a fixed lift $\rho_\nu:\Gamma_{F_\nu} \to G(\mathcal O)$ of $\restr{\rho_s}{\Gamma_{F_\nu}}$. We may choose this $\rho_\nu$ to be a formally smooth point of the generic fiber of the local lifting ring of $\restr{\overline{\rho}}{\Gamma_{F_\nu}}$.
            \item For $s = eD$, $\restr{\rho_s}{\Gamma_{F_\nu}}$ is trivial modulo the center, while $\alpha_\nu(\rho_{s+1}(\sigma_\nu)) \equiv N(\nu) \not \equiv 1 \pmod{\varpi^{s+1}}$, and $\beta(\rho_{s+1}(\sigma_\nu)) \not \equiv 1, N(\nu) \pmod{\varpi^{s+1}}$ for all roots $\beta \neq \alpha_\nu \in \Phi(G^0,T_\nu)$.
        \end{enumerate}

        \item For all $\nu \in S$, $\restr{\rho_n}{\Gamma_{F_\nu}}$ is strictly equivalent to $\rho_\nu \pmod{\varpi^n}$.
        \item If $\nu \in S_n \setminus S$ is not ramified in $\rho_n$ then $N(\nu) \not \equiv 1 \pmod{\varpi^{eD+1}}$, $\rho_n(\sigma_\nu)$ is a reduction of an element $t_\nu$ contained in a maximal torus $T_\nu(\mathcal O)$ that splits over an \'etale extension $\mathcal O^{(\nu)}/\mathcal O$, and $\beta(t_\nu) \not \equiv N(\nu) \pmod{\varpi^{n+1}}$ for all $\beta \in \Phi(G^0,T_\nu)$. Moreover, $\rho_n(\sigma_\nu)$ for all such primes $\nu$ generate $\ker(G^{\der}(\mathcal O/\varpi^n) \to G^{\der}(\mathcal O/\varpi^{n-1}))$. In particular, the image $\rho_n(\Gamma_F)$ contains $\widehat{G^{\der}}(\mathcal O/\varpi^n)$.
        
    \end{enumerate}
    
\end{thm}

\begin{proof}

    We first note that the only changes to {\cite[Theorem $4.4$]{FKP22}} are in part $(b)$ and part $(3)$. At such primes we will require stronger conditions on the image of the Frobenius. As seen from the proof of the theorem in \cite{FKP22} we are able to prescribe the image of $\sigma_\nu$ under $\rho_n$ in each of the inductive steps. Therefore, to ensure that $\rho_n(\sigma_\nu)$ satisfies the desired properties it is enough to find an element inside a maximal torus with those properties. 

    Primes $\nu$ that will satisfy the properties in part $(b)$ are exactly the ones at which $\rho_n$ will ramify and are added after the $eD$-th step of the doubling method, in particular $n \ge s+1$. These primes are chosen so that they split in $K(\rho_{n-1}(\mathfrak g^{\der}))$. Since $K(\rho_{n-1}(\mathfrak g^{\der}))$ and $K_\infty = K(\mu_{p^\infty})$ are linearly disjoint over $K$, as shown in Lemma \ref{3.19}, we can arrange for $q \coloneqq N(\nu)$ to be non-trivial modulo $\varpi^{eD+1}$, as $D$ is the highest power of $p$-th roots of unity contained in $K$. We will show the existence of the desired element inductively, by first finding such element for $n = s+1$. We note that $q$ is a square modulo $\varpi^{s+1}$. Indeed, working modulo $\varpi^{s+1}$ we have

    $$\left(1 + \frac{q-1}{2}\right)^2 = 1 + (q-1) + \frac{(q-1)^2}{4} = 1 +(q-1) = q$$

    \vspace{2 mm}

    \noindent where we used that $2$ is invertible modulo $\varpi^{s+1}$, as $p$ is odd, and also $q-1$ is divisible by $\varpi^s$, so the last term vanishes. Set $q^{1/2} \coloneqq 1 + \frac{q-1}2$. We now consider elements of the form $t_b = \exp(\varpi^sb)\alpha_\nu^\vee(q^{1/2})$, where $\alpha_\nu^\vee$ is the coroot associated to $\alpha_\nu$ and $b \in \ker(\restr{\alpha_\nu}{\mathfrak t_\nu})$. As $b \in \mathfrak t_\nu$, we get that $t_b \in T_\nu(\mathcal O/\varpi^{s+1})$. On the other hand,

    $$\alpha_\nu(t_b) = \alpha_\nu(\exp(\varpi^sb))\alpha_\nu(\alpha_\nu^\vee(q^{1/2})) = \exp(\varpi^s\alpha_\nu(b)) \cdot q^{\frac 12 \langle \alpha_\nu,\alpha_\nu^\vee\rangle} = q$$

    \vspace{2 mm}

    \noindent where we interchangeably view $\alpha_\nu$ as a root of the Lie group and the associated Lie algebra. Here, we used that $\alpha_\nu(b) = 0$ and $\langle \alpha_\nu,\alpha_\nu^\vee \rangle = 2$. Using the decomposition $\mathfrak g = \mathfrak g^{\der} \oplus \mathfrak z_G$ we can write $b = b' + z$ and so $\exp(\varpi^sb) = \exp(\varpi^sb')\exp(\varpi^sz)$. The second factor will determine the image of $t_b$ in $T_\nu(\mathcal O/\varpi^{s+1})/T_\nu^{\der}(\mathcal O/\varpi^{s+1})$. In particular, by choosing a suitable value for $z$ we can ensure this image will be equal to $\mu(\sigma_\nu) \pmod{\varpi^{s+1}}$ and therefore $\restr{\rho_{s+1}}{\Gamma_{F_\nu}}$ will have multiplier $\mu$. On the other hand, that factor will lie in the center, so it will not affect the value of $t_b$ under each of the roots, which will be given by $\exp(\varpi^sb')$.

    Now, it remains to choose $b' \in \ker(\restr{\alpha_\nu}{\mathfrak t^{\der}_\nu})$ in such a manner that the conditions for $\beta \neq \alpha_\nu$ are satisfied. We first check what are the conditions on $b'$ that will yield $\beta(t_b) = 1$ or $q$. Suppose that $\beta(t_b) = 1$. Then we have:

    $$1 = \beta(t_b) = \beta(\exp(\varpi^sb'))\beta(\alpha_\nu^\vee(q^{1/2}))= \exp(\varpi^s\beta(b'))q^{\frac 12 \langle \beta,\alpha_\nu^\vee\rangle}$$

    \vspace{2 mm}
    
    Working modulo $\varpi^{s+1}$ and taking logarithms on both sides we get:

    $$0 = \varpi^s\beta(b') + \frac 12 \langle \beta,\alpha_\nu^\vee\rangle(q-1) \implies \beta(b') = \frac{1-q}{2\varpi^s} \langle \beta,\alpha_\nu^\vee \rangle$$

    \vspace{2 mm}

    \noindent where we used $\log(q) = q-1$, as $(q-1)^m = 0$ modulo $\varpi^{s+1}$ for $m\ge 2$. A similar computation shows us that if $\beta(t_b) = q$, then modulo $\varpi^{s+1}$ we have:
    
    $$q-1 = \varpi^s\beta(b') + \frac 12 \langle \beta,\alpha_\nu^\vee\rangle(q-1) \implies \beta(b') = \frac{1-q}{2\varpi^s} \left(\langle \beta,\alpha_\nu^\vee \rangle - 2 \right)$$ 

    \vspace{2 mm}

    Now, let $\Phi_1$ be the subset of $\Phi(G^0,T_\nu) \setminus \{\alpha_\nu\}$ consisting of roots $\beta$ for which there exists $b_\beta' \in \ker(\restr{\alpha_\nu}{\mathfrak t^{\der}_\nu})$ such that $\beta(b_\beta') = \frac{1-q}{2\varpi^s}\langle \beta,\alpha_\nu^\vee \rangle$. Analogously, let $\Phi_2$ be the subset of $\Phi(G^0,T_\nu)  \setminus \{\alpha_\nu\}$ consisting of roots $\beta$ for which there exists $c_\beta' \in \ker(\restr{\alpha_\nu}{\mathfrak t^{\der}_\nu})$ such that $\beta(c_\beta') = \frac{1-q}{2\varpi^s}(\langle \beta,\alpha_\nu^\vee \rangle - 2)$. Now, we want to choose $b' \in \ker(\restr{\alpha_\nu}{\mathfrak t^{\der}_\nu})$ that will lie in the complement of the union of hyperplanes:

    $$\bigcup_{\beta \in \Phi_1} (b_\beta' + \ker(\restr{\beta}{\ker(\restr{\alpha_\nu}{\mathfrak t^{\der}_\nu})})) \cup \bigcup_{\beta \in \Phi_2} (c_\beta' + \ker(\restr{\beta}{\ker(\restr{\alpha_\nu}{\mathfrak t^{\der}_\nu})}))$$

    \vspace{2 mm}

    We note that $\ker(\restr{\beta}{\ker(\restr{\alpha_\nu}{\mathfrak t_\nu^{\der}})})$ will be a hyperplane in $\ker(\restr{\alpha_\nu}{\mathfrak t_\nu^{\der}})$, as long as $\beta \neq \pm \alpha_\nu$. We have already discarded $\alpha_\nu$ from both sets, while also $-\alpha_\nu$ is not an element of either $\Phi_1$ or $\Phi_2$. Indeed, if such $b_{- \alpha_\nu}'$ or $c_{- \alpha_\nu}'$ existed, as they would be chosen from $\ker(\restr{\alpha_\nu}{\mathfrak t_\nu^{\der}})$, the conditions above boil down to $q-1$ being divisible by $\varpi^{s+1}$, which is not true by our choice of $\nu$. Hence, the union above is indeed a union of hyperplanes inside $\ker(\restr{\alpha_\nu}{\mathfrak t_\nu^{\der}})$. But the number of such hyperplanes is bounded in terms of the size of $\Phi(G^0,T_\nu)$, which in turn depends only on the Dynkin type of $G^{\der}$. Hence, for $p \gg_G 0$ the union can't be equal to the whole of $\ker(\restr{\alpha_\nu}{\mathfrak t_\nu^{\der}})$. We choose $b'$ in the complement and for this choice $t_b$ is the desired element. 

    Having found such element for $n=s+1$, we will construct element $t_n$ modulo higher powers of $\varpi$ inductively. Suppose that for the given torus $T_\nu$ and some $n \ge s+1$ there is an element $t_n \in T_\nu(\mathcal O/\varpi^n)$ that satisfies the conditions. As the torus is formally smooth we can find a lift $t_{n+1}' \in T_\nu(\mathcal O/\varpi^{n+1})$ of $t_n$. Clearly, the conditions for roots $\beta \neq \alpha_\nu$ will be satisfied for any lift of $t_n$, as they are checked modulo $\varpi^{s+1}$. We now show that we can choose $t_{n+1}'$ such that $\alpha_\nu(t_{n+1}') \equiv q \pmod{\varpi^n}$. As $t_{n+1}'$ is a lift of $t_n$ from the inductive hypothesis we have $\alpha_\nu(t_{n+1}') \equiv q(1 + a\varpi^n) \pmod {\varpi^{n+1}}$. As $p \neq 2$ we can set $t_{n+1} \coloneqq t_{n+1}'\alpha_\nu^\vee\left(1 - \frac a2 \varpi^n\right)$. We then have:

    $$\alpha_\nu(t_{n+1}) = \alpha_\nu(t_{n+1}') \cdot \alpha_\nu\left(\alpha_\nu^\vee\left(1 - \frac a2 \varpi^n\right)\right) \equiv q(1+a\varpi^n)\left(1 - \frac a2 \varpi^n\right)^2 \equiv q(1+a\varpi^n)(1-a\varpi^n) \equiv q \pmod{\varpi^{n+1}}$$

    \vspace{2 mm}

    Hence, $t_{n+1}$ will be an element of $T_\nu(\mathcal O/\varpi^{n+1})$ as we wanted. Additionally, multiplying by an element $\exp(\varpi^nz)$ for $z \in \mathfrak z_G$ we can arrange for the image of $t_{n+1}$ in $T_\nu(\mathcal O/\varpi^n)/T_\nu^{\der}(\mathcal O/\varpi^n)$ to be equal to $\mu(\sigma_\nu) \pmod{\varpi^n}$, since we have already done that modulo $\varpi^n$.

    The primes $\nu \in S_n$ at which $\rho_n$ is not ramified are added for technical reason. Initially, we enlarge $S$ to $S_1$ before we even run the doubling method, by adding primes in order to annihilate a certain second Tate-Shafarevich group. After this no more primes are added for this reason. In the original doubling method at such primes we produce unramified local lifts (with multiplier $\mu$). On the other hand, throughout the doubling method we add additional primes to $S_n$ in order to guarantee that $\rho_n(\Gamma_F)$ contains $\widehat{G^{\der}}(\mathcal O/\varpi^n)$. These primes are chosen so that they split over $K(\rho_{n-1}(\mathfrak g^{\der}))$. By Lemma \ref{3.19} we know that $K(\rho_{n-1}(\mathfrak g^{\der}))$ and $K_\infty$ are disjoint over $K$, so we can arrange $q = N(v) \not \equiv 1 \pmod{\varpi^{eD+1}}$ for all such primes $\nu$. At these primes we choose unramified local lifts $\rho_\nu$ (with multiplier $\mu$) such that $\rho_\nu(\sigma_\nu) \pmod{\varpi^n}$ generate $\ker(G^{\der}(\mathcal O/\varpi^n) \to G^{\der}(\mathcal O/\varpi^{n-1}))$ for $n \ge 2$. For our purposes we need to impose some further conditions on $\rho_\nu(\sigma_\nu)$. As the Galois group of the maximal unramified extension at these primes is a free profinite group generated by $\sigma_\nu$ we can produce unramified lifts $\rho_\nu$ with a prescribed image of $\sigma_\nu$.

    We will first show that $\ker(G^{\der}(\mathcal O/\varpi^n) \to G^{\der}(\mathcal O/\varpi^{n-1}))$ is generated by regular semisimple elements. As in Remark \ref{3.6} the exponential map gives us an isomorphism $\mathfrak g^{\der} \simeq \ker(G^{\der}(\mathcal O/\varpi^n) \to G^{\der}(\mathcal O/\varpi^{n-1}))$, where we identify $\mathfrak g^{\der} \otimes_{\mathcal O} \varpi^{n-1}\mathcal O/\varpi^n$ with $\mathfrak g^{\der} \otimes_{\mathcal O} k$. Thus, we can work with the Lie algebra instead. The claim translates to showing that the $\mathbb F_p$-span of regular semisimple elements spans $\mathfrak g^{\der}$. Clearly, regularity and semisimplicity are preserved under the adjoint action of $G(k)$, so this $\mathbb F_p$-span is in fact a $\mathbb F_p[G(k)]$-submodule of $\mathfrak g^{\der}$. For $p \gg_G 0$, $\mathfrak g^{\der}$ is an irreducible $\mathbb F_p[G(k)]$-module, which gives us what we want. 

    Let $t_{1,n},\dots, t_{r,n} \in G(\mathcal O/\varpi^n)$ be elements whose images are regular semisimple generators of $\ker(G^{\der}(\mathcal O/\varpi^n) \to G^{\der}(\mathcal O/\varpi^{n-1}))$. The centralizer of each $t_{i,n}$ is a maximal torus in $G(\mathcal O/\varpi^n)$. This torus can be lifted to an $\mathcal O$-torus $T_i$, which will split over an \'etale extension $\mathcal O_i/\mathcal O$ that corresponds to a finite extension of residue fields $k_i/k$. We now add $r$ new primes $\nu_i$ to $S_n$, one for each element $t_{i,n}$. These are chosen so that $q_i \coloneqq N(\nu_i) \equiv 1 \pmod{\varpi^{n-1}}$ for $i \le r$. We want to lift each $t_{i,n}$ modulo $\varpi^{n+1}$ such that the conditions on the roots are satisfied. We do this in the same way as in the first part of the proof. We lift each $t_{i,n}$ to $t_{i,n+1}' \in G(\mathcal O/\varpi^{n+1})$. We are interested in elements of the form $t_{b_i} = \exp(\varpi^nb_i)t_{i,n+1}'$ for $b_i \in \mathfrak t_i$. We want to check what conditions $\beta(t_{b_i}) = q_i$ implies on $b_i$. If that is the case we have:
    
    $$q_i = \beta(t_{b_i}) = \beta(\exp(\varpi^nb_i))\beta(t_{i,n+1}') = \exp(\varpi^n\beta(b_i))\beta(t_{i,n+1}')$$
    
    \vspace{2 mm}
    
    Working modulo $\varpi^{n+1}$, by taking logarithms, which is possible, as $q \equiv 1 \pmod p$ and $t_{i,n+1}$ is trivial modulo $\varpi^{n-1}$ we get:

    $$q_i-1 = \varpi^n\beta(b_i) + \log(\beta(t_{i,n+1}')) \implies \beta(b_i) = \frac{1}{\varpi^n}(\log(\beta(t_{i,n+1}'))+1-q_i)$$

    \vspace{2 mm}

    Now, let $\Phi'$ be the subset of $\Phi(G^0,T_i)$ consisting of roots $\beta$ for which there exists $b_{i,\beta} \in \mathfrak t_i$ that satisfies the above equation. We want to choose $b_i$ that will lie in the complement of the union of hyperplanes $\bigcup_{\beta \in \Phi'} \left(b_{i,\beta} + \ker(\restr{\beta}{\mathfrak t_i})\right)$. For $p \gg_G 0$ this can't equal the whole of $\mathfrak t_i$, so we can find such $b_i$ for each $i$. We set $t_{i,n+1} = t_{b_i}$. Using the smoothness of the torus we lift these to elements $t_i \in T_i(\mathcal O)$. By multiplying with an element of $Z_{G^0}(\mathcal O)$, if needed, we can assume that each $t_i$ has image equal to $\mu(\sigma_{\nu_i})$ in $G(\mathcal O)/G^{\der}(\mathcal O)$. As explained above at each prime $\nu$ we define local unramified multiplier $\mu$ lifts $\rho_{\nu_i}:\Gamma_{F_{\nu_i}} \to G(\mathcal O)$ such that $\rho_{\nu_i}(\sigma_{\nu_i}) \pmod{\varpi^n}$ will generate $\ker(G^{\der}(\mathcal O/\varpi^n) \to G^{\der}(\mathcal O/\varpi^{n-1}))$. We feed these lifts in the proof of {\cite[Theorem $4.4$]{FKP22}} to construct a lift $\rho_n$ which will satisfy the modified part $(3)$ of the theorem.
     
\end{proof}

\subsection{Relative deformation theory}

In this subsection we will take a closer look at the relative deformation argument (\cite[\S $6$]{FKP21}) and make some necessary technical modifications. Because of this and to make things easier to follow, we will restate here most of the results and proofs of the original argument. Let $n \ge 1$ be an arbitrary, but fixed integer and $\rho_n:\Gamma_{F,T} \to G(\mathcal O/\varpi^n)$ a lift of $\overline{\rho}$, where $T$ is a finite set of primes of $F$.

\begin{assump}

\label{3.10} Assume that for each $\nu \in T$ and $1 \le r \le n$ we have subgroups $Z_{r,\nu} \subseteq Z^1(\Gamma_{F_\nu},\rho_r(\mathfrak g^{\der}))$ which satisfy:

\begin{itemize}
    \item Each $Z_{r,\nu}$ contains the boundaries $B^1(\Gamma_{F_\nu},\rho_r(\mathfrak g^{\der}))$.
    \item For $a,b \ge 1$ and $a+b = r$ we have a short exact sequence

    $$0 \longrightarrow Z_{a,\nu} \longrightarrow Z_{a+b,\nu} \longrightarrow Z_{b,\nu} \longrightarrow 0$$

    \vspace{2 mm}

    \noindent where the first map is multiplication by $\varpi^b$, while the second one is reduction modulo $\varpi^b$. 
    
\end{itemize}
    
\end{assump}

We now write $L_{r,\nu} \subseteq H^1(\Gamma_{F_\nu},\rho_r(\mathfrak g^{\der}))$ for the image of $Z_{r,\nu}$. We let $L_{r,\nu}^\perp \subseteq H^1(\Gamma_{F_\nu},\rho_r(\mathfrak g^{\der})^*)$ be the annihilator of $L_{r,v}$ under the local duality pairing. We label the collection $\{L_{r,\nu}\}_{\nu \in T}$ by $\mathcal L_{r,T}$. Analogously, we write $\mathcal L_{r,T}^\perp$ for $\{L_{r,\nu}^\perp\}_{v \in T}$. As mentioned above, the lifting process relies heavily on our ability to annihilate an $M$-relative Selmer group. We now give the definition of a relative Selmer group, as well as its generalization which we will use later.

\begin{defn}

    \text{ }

    \label{3.11}

    \begin{enumerate}[a)]
        \item For $1 \le r \le n$ we define the $r$-th Selmer group $H^1_{\mathcal L_{r,T}}(\Gamma_{F,T},\rho_r(\mathfrak g^{\der}))$ to be

        $$\ker \left( H^1(\Gamma_{F,T},\rho_r(\mathfrak g^{\der})) \longrightarrow \bigoplus_{v \in T} \frac{H^1(\Gamma_{F_\nu},\rho_r(\mathfrak g^{\der}))}{L_{r,\nu}}\right)$$

        \vspace{2 mm}

        \noindent with the $r$-th dual Selmer group $H^1_{\mathcal L_{r,T}^\perp}(\Gamma_{F,T},\rho_r(\mathfrak g^{\der})^*)$ defined analogously.

        \item For $1 \le s \le r \le n$ we define the $(r,s)$-relative Selmer group  $H_{(r,s)}^1(\Gamma_{F,T},\rho_r(\mathfrak g^{\der}))$ to be

        $$\im \left(H^1_{\mathcal L_{r,T}}(\Gamma_{F,T},\rho_r(\mathfrak g^{\der})) \longrightarrow H^1_{\mathcal L_{s,T}}(\Gamma_{F,T},\rho_s(\mathfrak g^{\der}))\right)$$

        \vspace{2 mm}

        \noindent where the map is induced by the reduction modulo $\varpi^s$ map. We analogously define the $(r,s)$-relative dual Selmer group $H_{(r,s)}^1(\Gamma_{F,T},\rho_r(\mathfrak g^{\der})^*)$.

    \end{enumerate}
    
\end{defn}

\begin{rem}

    \label{3.12}

    When $s=1$ in the definition of the $(r,s)$-relative Selmer group we use the alternative notation $\overline{H^1_{\mathcal L_{r,T}}(\Gamma_{F,T},\rho_r(\mathfrak g^{\der}))}$ and we call it the $r$-th relative Selmer group. We use the analogous convention for the $r$-th relative dual Selmer group, too.
    
\end{rem}

\begin{rem}

    \label{3.13}

    The local conditions are left implicit in the notation for the $(r,s)$-relative Selmer group. We do this in order to simplify the notation. This will not cause any confusion as we will be mostly using the same local conditions throughout the paper. We will make sure to point it out whenever there is a possibility of confusion or the local conditions are altered. 
    
\end{rem}

We have the following important exact sequences of Selmer groups

\begin{lem}[{\cite[Lemma $6.1$]{FKP21}}]

    \label{3.14}
    
    For $a,b \ge 1$ such that $a + b \le n$ there are exact sequences:

    $$H^1_{\mathcal L_{a,T}}(\Gamma_{F,T},\rho_a(\mathfrak g^{\der})) \longrightarrow H^1_{\mathcal L_{a+b,T}}(\Gamma_{F,T},\rho_{a+b}(\mathfrak g^{\der})) \longrightarrow H^1_{\mathcal L_{b,T}}(\Gamma_{F,T},\rho_b(\mathfrak g^{\der}))$$

    \vspace{2 mm}

    \noindent and

    $$H^1_{\mathcal L_{a,T}^\perp}(\Gamma_{F,T},\rho_a(\mathfrak g^{\der})^*) \longrightarrow H^1_{\mathcal L_{a+b,T}^\perp}(\Gamma_{F,T},\rho_{a+b}(\mathfrak g^{\der})^*) \longrightarrow H^1_{\mathcal L_{b,T}^\perp}(\Gamma_{F,T},\rho_b(\mathfrak g^{\der})^*)$$

    \vspace{2 mm}
    
\end{lem}

Working with Selmer groups which have the same size as the corresponding dual Selmer groups plays a major technical role in the relative deformation theory. We say that the local conditions $\mathcal L_{r,T}$ for $1 \le r \le n$ are \textit{balanced} if

$$\left|H^1_{\mathcal L_{r,T}}(\Gamma_{F,T},\rho_r(\mathfrak g^{\der}))\right| = \left|H^1_{\mathcal L_{r,T}^\perp}(\Gamma_{F,T},\rho_r(\mathfrak g^{\der})^*)\right|$$

\vspace{2 mm}

\noindent for all $1 \le r \le n$. We now have the following result which tells us that under some mild condition the local conditions being balanced implies that the $(r,s)$-relative Selmer groups are also balanced.

\begin{lem}

    \label{3.15}

    Suppose that the local conditions $\mathcal L_{r,T}$ for $1 \le r \le n$ are balanced and the spaces of invariants $\overline{\rho}(\mathfrak g^{\der})^{\Gamma_F}$ and $(\overline{\rho}(\mathfrak g^{\der})^*)^{\Gamma_F}$ are both zero. Then for $1 \le s \le r \le n$:

    $$\left|H^1_{(r,s)}(\Gamma_{F,T},\rho_r(\mathfrak g^{\der}))\right| = \left|H^1_{(r,s)}(\Gamma_{F,T},\rho_r(\mathfrak g^{\der})^*)\right|$$

    \vspace{2 mm}
    
\end{lem}
\begin{proof}

    We first show that $\overline{\rho}(\mathfrak g^{\der})^{\Gamma_F} = 0$ implies $\rho_r(\mathfrak g^{\der})^{\Gamma_F} = 0$ for all $1 \le r \le n$. Suppose that we can find $A \neq 0 \in \mathfrak g^{\der} \otimes_{\mathcal O} \mathcal O/\varpi^r$ such that $\Ad(\rho_r(\sigma)) \cdot A = A$ for all $\sigma \in \Gamma_F$. We can then write $A = \varpi^{r_1}A'$ for some $r_1 < r$. such that $\overline{A'} \coloneqq A' \pmod{\varpi} \neq 0$. Hence, we get that $\Ad(\rho_r(\sigma)) \cdot A' \equiv A' \pmod{\varpi^{r-r_1}}$ for all $\sigma \in \Gamma_F$. In particular, working modulo $\varpi$ we get that $\Ad(\overline{\rho}(\sigma)) \cdot \overline{A'} = \overline{A'}$. But, as $\overline{\rho}(\mathfrak g^{\der})^{\Gamma_F} = 0$ this yields $\overline{A'} = 0$, which is a contradiction. Analogously, $(\overline{\rho}(\mathfrak g^{\der})^*)^{\Gamma_F} = 0$ implies that $(\rho_r(\mathfrak g^{\der})^*)^{\Gamma_F} = 0$ for all $1 \le r \le n$.

    Now, using the exact sequences in Lemma $\ref{3.14}$ we get:

    \begin{align*}
        \left|H^1_{(r,s)}(\Gamma_{F,T},\rho_r(\mathfrak g^{\der}))\right| &= \left| \im \left(H^1_{\mathcal L_{r,T}}(\Gamma_{F,T},\rho_r(\mathfrak g^{\der})) \longrightarrow H^1_{\mathcal L_{s,T}}(\Gamma_{F,T},\rho_s(\mathfrak g^{\der}))\right) \right| \\
        &= \left|H^1_{\mathcal L_{r,T}}(\Gamma_{F,T},\rho_r(\mathfrak g^{\der}))\right| - \left| \ker \left(H^1_{\mathcal L_{r,T}}(\Gamma_{F,T},\rho_r(\mathfrak g^{\der})) \longrightarrow H^1_{\mathcal L_{s,T}}(\Gamma_{F,T},\rho_s(\mathfrak g^{\der}))\right) \right| \\
        &= \left|H^1_{\mathcal L_{r,T}}(\Gamma_{F,T},\rho_r(\mathfrak g^{\der}))\right| - \left| \im \left(H^1_{\mathcal L_{r-s,T}}(\Gamma_{F,S},\rho_{r-s}(\mathfrak g^{\der})) \longrightarrow H^1_{\mathcal L_{r,T}}(\Gamma_{F,S},\rho_r(\mathfrak g^{\der}))\right) \right| 
    \end{align*}

    \vspace{2 mm}

    On the other side, from the long exact sequence on cohomology we get the exact sequence:

    $$H^0(\Gamma_{F,T},\rho_s(\mathfrak g^{\der})) \longrightarrow H^1(\Gamma_{F,T},\rho_{r-s}(\mathfrak g^{\der})) \longrightarrow H^1(\Gamma_{F,T},\rho_r(\mathfrak g^{\der}))$$

    \vspace{2 mm}

    However, by what we have proven at the beginning the first term is equal to $0$ and therefore the second map is injective. This means that the size of the image is equal to the size of the domain. Combining this with the analogous computations for the dual Selmer group and the assumption that the local conditions are balanced we get:

    \begin{align*}
    \left|H^1_{(r,s)}(\Gamma_{F,T},\rho_r(\mathfrak g^{\der}))\right| &= \left|H^1_{\mathcal L_{r,T}}(\Gamma_{F,T},\rho_r(\mathfrak g^{\der}))\right| - \left|H^1_{\mathcal L_{r-s,T}}(\Gamma_{F,T},\rho_{r-s}(\mathfrak g^{\der}))\right| \\
    &= \left|H^1_{\mathcal L_{r,T}^\perp}(\Gamma_{F,T},\rho_r(\mathfrak g^{\der})^*)\right| - \left|H^1_{\mathcal L_{r-s,T}^\perp}(\Gamma_{F,T},\rho_{r-s}(\mathfrak g^{\der})^*)\right| \\
    &= \left|H^1_{(r,s)}(\Gamma_{F,T},\rho_r(\mathfrak g^{\der})^*)\right|
    \end{align*}
    
\end{proof}

We remark that all the results up to this point in this subsection hold for any local lift $\rho_n$ that satisfies Assumptions \ref{3.10}. However, from now on we will focus on lifts produced by the doubling method. We remark that in that case the existence of such cocycles for primes in $S_n$ is guaranteed by either \cite[Lemma $3.5$]{FKP21} or \cite[Proposition $4.7$]{FKP21}. Since the cocycles produced by \cite[Lemma $3.5$]{FKP21} will appear in some of the computations later we take our time to restate its result, as well as to describe explicitly those cocycles.

\begin{lem}[{\cite[Lemma $3.5$]{FKP21}}]

    \label{3.16}

    Let $F_\nu/\mathbb Q_\ell$ be a finite extension with residue field of order $q$, let $M \ge 1$ be a fixed integer, and let $1 \le s \le M$ be another fixed integer. Suppose $\rho_{M+s}:\Gamma_{F_\nu} \to G(\mathcal O/\varpi^{M+s})$ is a homomorphism with multipler $\mu$ satisfying:

    \begin{itemize}
        \item The reduction $\rho_s \coloneqq \rho_{M+s} \pmod{\varpi^s}$ is trivial (modulo the center) and $q \equiv 1 \pmod{\varpi^s}$, but $q \not \equiv 1 \pmod{\varpi^{s+1}}$.
        \item There is a suitable choice of split maximal torus $T$ and root $\alpha \in \Phi(G^0,T)$ such that $\rho_{M+s}(\sigma_\nu) \in T(\mathcal O/\varpi^{M+s})$, $\alpha(\rho_{M+s}(\sigma_\nu)) = q$ and $\rho_{M+s}(\tau_\nu) \in U_\alpha(\mathcal O/\varpi^{M+s})$. In particular $\rho_{M+s} \in \Lift_{\rho_M}^{\mu,\alpha}(\mathcal O/\varpi^{M+s})$.
        \item For any root $\beta \in \Phi(G^0,T)$, $\beta(\rho_{M+s}(\sigma_\nu)) \not \equiv 1 \pmod{\varpi^{s+1}}$.

    \end{itemize}

        Then, for all $1 \le r \le M$ there are spaces of cocycles $Z_{r,\nu}^\alpha \subseteq Z^1(\Gamma_{F_\nu},\rho_r(\mathfrak g^{\der}))$ with images $L_{r,\nu}^\alpha \subseteq H^1(\Gamma_{F_\nu},\rho_r(\mathfrak g^{\der}))$ such that

    \begin{itemize}
         \item $Z_{r,\nu}^\alpha$ contains all coboundaries and is free over $\mathcal O/\varpi^r$ of rank $\dim(\mathfrak g^{\der})$.
        \item For any integers $a,b \ge 1$ such that $a+b = r$, the nautral maps induce short exact sequences:

        $$0 \longrightarrow Z_{a,\nu}^\alpha \longrightarrow Z_{r,\nu}^\alpha \longrightarrow Z_{b,\nu}^\alpha \longrightarrow 0.$$

        \vspace{2 mm}
        
        \item For all $m \ge 2s + r$ (in particular, for all $m \ge 3M$), and any lift $\rho_m \in \Lift_{\rho_r}^{\mu,\alpha}(\mathcal O/\varpi^m)$ of $\rho_{r+s}$, the fiber of $\Lift_{\rho_r}^{\mu,\alpha}(\mathcal O/\varpi^{m+r}) \to \Lift_{\rho_r}^{\mu,\alpha}(\mathcal O/\varpi^m)$ over $\rho_m$ is non-empty and $Z_{r,\nu}^\alpha$-stable.
    \end{itemize}
    
\end{lem}
\begin{rem}

\label{3.17}

    The spaces $Z_{r,\nu}^\alpha$ consist of $3$ types of cocycles:

    \begin{itemize}
        \item $\phi_{X_\beta}^r(\gamma) = \frac{X_\beta - \Ad(\rho_{m+r}(\gamma))X_\beta}{\varpi^s} \pmod{\varpi^r}$, depending only on $\rho_{r+s}$, where $X_\beta$ is a generator of $\mathfrak g_\beta$ over $\mathcal O$.
        \item For $X$ in an $\mathcal O$-basis of $\ker(\restr{\alpha}{\mathfrak t^{\der}})$, the unramified cocycles given by $\phi_X^{\mathrm{un},r}(\sigma_\nu) = X$.
        \item The ramified cocycle $\phi_\alpha^r$ given by $\phi_\alpha^r(\sigma_\nu) = 0$ and $\phi_\alpha^r(\tau_\nu) = X_\alpha$.
    \end{itemize}
    
\end{rem}

\begin{rem}

    \label{3.18}

    We consider the map $X \to \phi_X$ that assigns to every element in $\mathfrak g^{\der} \otimes_\mathcal O \mathcal O/\varpi^r$ the coboundary $\phi_X(\gamma) \coloneqq X - \gamma \cdot X$. The kernel of this map consists of all the $\Gamma_{F_\nu}$-invariants, while the image is exactly the subspace of all coboundaries $B^1(\Gamma_{F_\nu},\rho_r(\mathfrak g^{\der}))$. Therefore, we get that:

    $$|\,H^0(\Gamma_{F_\nu},\rho_r(\mathfrak g^{\der})\,| \cdot |\,B^1(\Gamma_{F_\nu},\rho_r(\mathfrak g^{\der}))\,| = |\,\mathfrak g^{\der} \otimes_\mathcal O \mathcal O/\varpi^r\,| = |\,Z_{r,\nu}^\alpha\,|$$

    \vspace{2 mm}

    Using the fact that $Z_{r,\nu}^\alpha$ contains all the coboundaries we get that $|\,H^0(\Gamma_{F_\nu},\rho_r(\mathfrak g^{\der}))\,| = |\,L_{r,\nu}^\alpha\,|$. This equality will be important when it comes down to computing the balancedness of the local conditions. Additionally, using the explicit description of $\rho_{M+s}$ we can compute the size of $L_{r,\nu}^\alpha$ and we get that:

    $$|\,L_{r,\nu}^\alpha\,| = \begin{cases} |\,\mathcal O/\varpi\,|^{r\dim(\mathfrak g^{\der})} & \text{if } r \le s \\ |\,\mathcal O/\varpi\,|^{s\dim(\mathfrak g^{\der}) + (r-s)\dim(\mathfrak t^{\der}) + (t-r)} & \text{if } r>s \end{cases}$$

    \vspace{2 mm}

    \noindent where $t$ is the largest integer less than or equal to $r$ such that $\rho_t$ is not ramified. This is done by computing the sizes of either $B^1(\Gamma_{F_\nu},\rho_r(\mathfrak g^{\der}))$ or $H^0(\Gamma_{F_\nu},\rho_r(\mathfrak g^{\der}))$. We quickly explain how to do it for the latter. If $r \le s$, then $\rho_r$ is trivial, hence $|\,H^0(\Gamma_{F_\nu},\rho_r(\mathfrak g^{\der}))\,| = |\,\mathcal O/\varpi\,|^{r\dim(\mathfrak g^{\der})}$. For $r > s$ we use the root decomposition of $\mathfrak g^{\der}$ with respect to $T$. Suppose that $X = Z + \sum_\beta X_\beta$ is invariant under the adjoint action of $\rho_r$. Acting by $\sigma_\nu$ and comparing terms we get that $(\beta(\rho_r(\sigma_\nu))-1))X_\beta = 0$ for all $\beta \in \Phi(G^0,T)$. By the third bullet point of the assumptions in the lemma we get that $\beta(\rho_r(\sigma_\nu))-1$ is divisible by $\varpi^s$, but not by $\varpi^{s+1}$. Therefore, $X_\beta$ is necessarily a multiple of $\varpi^{r-s}$. As $\rho_s$ is trivial this is also a sufficient condition. From this we get that each $X_\beta$ is fixed by the $\Gamma_{F_\nu}$-action and hence the same must be true for $Z$. Moreover, we have $|\,\mathcal O/\varpi\,|^{s}$ options for each $X_\beta$. On the other side, $Z$ is always fixed by $\sigma_\nu$, so we consider the action of $\tau_\nu$. By our assumption $\rho_r(\tau_\nu) = u_\alpha(yY_\alpha) \pmod{\varpi^r}$ for some $y \in \mathcal O$ and $Y_\alpha$ a generator of $\mathfrak g_\alpha$ over $\mathcal O$. Hence, $Z$ being fixed by $\tau_\nu$ is equivalent to $y\alpha(Z) \equiv 0 \pmod{\varpi^r}$. If $Z \in \ker(\restr{\alpha}{\mathfrak t^{\der}})$ then this is always true, so we have $|\,\mathcal O/\varpi\,|^{r(\dim(\mathfrak t^{\der})-1)}$ options for $Z$ coming from this kernel. If $Z$ lies in a complement of this kernel it must be a multiple of $\varpi^{r-t}$, where $t$ is the largest integer less than or equal to $r$ such that $\rho_t$ is not ramified, hence $|\,\mathcal O/\varpi \,|^t$ options in this case. Combining everything we get the desired size of $H^0(\Gamma_{F_\nu},\rho_r(\mathfrak g^{\der}))$.
    
\end{rem}

    Let $\Gamma$ be the inverse image in $G^{\mathrm{ad}}(\mathcal O)$ of $\Ad \circ \overline{\rho}(\Gamma_{\widetilde{F}}) \subseteq G^{\mathrm{ad}}(k)$. Applying \cite[Corollary B.$2$]{FKP21} we get that the reduction map 

    $$H^1(\Gamma,\mathfrak g^{\der} \otimes_{\mathcal O} \mathcal O/\varpi^M) \longrightarrow H^1(\Gamma,\mathfrak g^{\der} \otimes_{\mathcal O} k)$$

    \vspace{2 mm}

    \noindent is zero for all $M$ greater than some integer $M_1$, depending only on $\im(\overline{\rho})$. We suppose we have integers $N \ge M \ge M_1$, which are additionally divisible by $e$, the ramification index of $\mathcal O$. We let $\rho_N: \Gamma_{F,S_N} \to G(\mathcal O/\varpi^N)$ be a lift of $\overline{\rho}$ obtained by running the doubling method. Then, for $1 \le r \le N$ we set $F_r = F(\overline{\rho},\rho_r(\mathfrak g^{\der}))$ and $F_r^* = F_N(\mu_{p^{\lceil r/e \rceil}})$. We now state a simple lemma that we will use to establish the linear disjointness of fields:

    \begin{lem}[{\cite[Lemma $4.3$]{FKP22}}]

        \label{3.19}

       For any $1 \le r \le N$, the fields $K_\infty = K(\mu_{p^\infty})$ and $K(\rho_r(\mathfrak g^{\der}))$ are linearly disjoint over $K$. 
        
    \end{lem}
    
    We have the following result:

    \begin{lem}[{\cite[Lemma $6.4$]{FKP21}}]

        \label{3.20}

        For any sufficiently large $M$, in a manner depending only on $\im(\overline{\rho})$ we have:

        \begin{itemize}
            \item $H^1(\Gal(F_N^*/F),\rho_r(\mathfrak g^{\der})^*) = 0$ for all $r \le M$.
            \item The map $H^1(\Gal(F_N^*/F),\rho_M(\mathfrak g^{\der})) \to H^1(\Gal(F_N^*/F),\overline{\rho}(\mathfrak g^{\der}))$ is zero.
        \end{itemize}
        
    \end{lem}

    We now fix an integer $M \ge \max\{eD,M_1\}$, asking for it to be divisible by $e$, as before. On the other hand, we will not fix the integer $N$ for the time being, but we will continue requiring it to be divisible by $e$. We now define the set of auxiliary primes that will be used to annihilate a relative dual Selmer group. 

    \begin{defn}

        \label{3.21}

        For $N \ge M$ we let $Q_{N,M}$ be the set of trivial primes $\nu$ of $F$ such that:

        \begin{itemize}
            \item $N(\nu) \equiv 1 \pmod{\varpi^M}$, but $N(\nu) \not \equiv 1 \pmod{\varpi^{M+1}}$.
            \item $\restr{\rho_N}{\Gamma_{F_\nu}}$ is unramified (with multiplier $\mu$) and $\restr{\rho_{M}}{\Gamma_{F_\nu}}$ is trivial modulo the center. 
            \item There exists a split maximal torus $T$ of $G$ and a root $\alpha \in \Phi(G^0,T)$ such that $\rho_N(\sigma_\nu) \in T(\mathcal O/\varpi^N)$ and $\alpha(\rho_N(\sigma_\nu)) = \kappa(\sigma_\nu) = N(\nu)$.
            \item For all roots $\beta \neq \alpha \in \Phi(G^0,T)$, $\beta(\rho_{M+1}(\sigma_\nu)) \not \equiv 1, N(\nu) \pmod{\varpi^{M+1}}$.
        \end{itemize}
        
    \end{defn}

    We first remark that this definition of $Q_{N,M}$ differs slightly than the one in \cite[Definition $6.6$]{FKP21}, as we ask for additional condition in the fourth bullet point. In any case, the set $Q_{N,M}$ corresponds to a Chebotarev condition in $F_N^*/F$, hence it has a positive density. Indeed, the first condition can be prescribed by choosing a Chebotarev condition in $K(\mu_p^{N/e})/K$. On the other hand, as $\im(\rho_N)$ contains $\widehat{G^{\der}}(\mathcal O/\varpi^N)$ we can find an element of it that satisfies the last three conditions as in the proof of Theorem \ref{3.4}. This will give us a Chebotarev condition in $K(\rho_N(\mathfrak g^{\der}))/K$, which is moreover trivial on $K(\rho_M(\mathfrak g^{\der}))$. Finally, we need to make sure that these two conditions are compatible. This is guaranteed by the linear disjointess of the extensions over $K$, which comes from Lemma \ref{3.19}. Now, for primes in $Q_{N,M}$ we consider the lifting functors $\Lift_{\overline{\rho}}^{\mu,\alpha}$ and $\Lift_{\rho_M}^{\mu,\alpha}$, which we have defined in Definition \ref{3.2} and Definition \ref{LiftM}, respectively.

    \begin{lem}[cf. {\cite[Lemma $6.7$]{FKP21}}]

        \label{3.22}

        Assume $N \ge 4M$, Then for any $\nu \in Q_{N,M}$ with the corresponding $(T,\alpha)$, the following properties hold:

        \begin{enumerate}
            \item For any $m \ge N-M$ and $1 \le r \le M$, the fibers of $\Lift_{\restr{\rho_M}{\Gamma_{F_\nu}}}^{\mu,\alpha}(\mathcal O/\varpi^{m+r}) \longrightarrow \Lift_{\restr{\rho_M}{\Gamma_{F_\nu}}}^{\mu,\alpha}(\mathcal O/\varpi^m)$ are nonempty and stable under $Z_{r,\nu}^\alpha$, where $Z_{r,\nu}^\alpha \subseteq Z^1(\Gamma_{F_\nu},\rho_r(\mathfrak g^{\der}))$ is the submodule produced in Lemma \textup{\ref{3.16}}.
    
            \item For $1 \le r \le N-M$ let $L_{r,\nu}^\alpha$ be the image of $Z_{r,\nu}^\alpha$ in $H^1(\Gamma_{F_\nu}.\rho_r(\mathfrak g^{\der}))$, and let $L_{r,\nu}^{\alpha,\perp}$ be its annihilator in $H^1(\Gamma_{F_\nu},\rho_r(\mathfrak g^{\der})^*)$ under the local duality pairing. Then: 

            \begin{itemize}
                \item $|L_{r,\nu}^\alpha| = |H^0(\Gamma_{F_\nu},\rho_r(\mathfrak g^{\der}))|= |H^1_{\mathrm{unr}}(\Gamma_{F_\nu},\rho_r(\mathfrak g^{\der}))|$.
                \item The inclusion

                $$L_{r,\nu}^\alpha \cap H^1_{\mathrm{unr}}(\Gamma_{F_\nu},\rho_r(\mathfrak g^{\der})) \hookrightarrow H^1_{\mathrm{unr}}(\Gamma_{F_\nu},\rho_r(\mathfrak g^{\der}))$$

                \vspace{2 mm}

                has cokernel isomorphic to $\mathcal O/\varpi^r$, and this cokernel is generated by the image of the unramified cocycle that maps $\sigma_\nu$ to $H_\alpha = d(\alpha^\vee)(1)$, the usual coroot element in $\mathfrak t^{\der}$.

                \item The inclusion

                $$L_{r,\nu}^{\alpha,\perp} \cap H^1_{\mathrm{unr}}(\Gamma_{F_\nu},\rho_r(\mathfrak g^{\der})^*) \hookrightarrow H^1_{\mathrm{unr}}(\Gamma_{F_\nu},\rho_r(\mathfrak g^{\der})^*)$$

                \vspace{2 mm}

                has cokernel isomorphic to $\mathcal O/\varpi^r$, and this cokernel is generated by the image of the unramified cocycle that maps $\sigma_\nu$ to any element of $(\mathfrak g^{\der})^*$ whose restriction to $\mathfrak g_\alpha$ spans the free rank one $\mathcal O/\varpi^r$-module $\Hom_\mathcal O(\mathfrak g_\alpha,\mathcal O/\varpi^r)$.

                \vspace{2 mm}
                
            \end{itemize}
        \end{enumerate}
        
    \end{lem}
    \begin{proof}

        Part $(2)$ of the lemma is stronger than its corresponding part in {\cite[Lemma $6.7$]{FKP21}}, so we take our time to prove it. More precisely, the proof in \cite{FKP21} relies on the fact that $\restr{\rho_M}{\Gamma_{F_\nu}}$ is trivial and hence the statement is valid for $1 \le r \le M$. On the other hand, the proof below uses only that $\restr{\rho_N}{\Gamma_{F_\nu}}$ is unramified and therefore the statement is valid for $1 \le r \le N-M$, as well.
        
        For the first bullet point of part $(2)$ we first note that $H_{\mathrm{unr}}^1(\Gamma_{F_\nu},\rho_r(\mathfrak g^{\der})) = H^1(\Gamma_{F_{\nu}}/I_{F_\nu},\rho_r(\mathfrak g^{\der})^{I_{F_\nu}})$, where $I_{F_\nu}$ is the inertia group of $F_\nu$. Then, it is a standard exercise in cohomology of free abelian profinite groups that this group has the same size as $H^0(\Gamma_{F_\nu},\rho_r(\mathfrak g^{\der}))$. By Remark \ref{3.18} the latter has the same size as $L_{r,\nu}^\alpha$, which gives us what we want.

        For the first cokernel we note that as $\rho_{r+M}$ is unramified the cocycles $\phi_{X_{\beta}}^r$ are in fact unramified. This means that $Z_{r,\nu}^\alpha$ contains the unramified cocycles $\phi_X^{\mathrm{un},r}$ for $X \in \ker(\restr{\alpha}{\mathfrak t^{\der}}) \oplus \bigoplus_{\beta \in \Phi(G^0,T)} \mathfrak g_\beta$. On the other hand, the unramified cocycles $Z_{r,\nu}^{\mathrm{un}}$ are determined by the image of $\sigma_\nu$ and we are free to send it to any element of $\rho_r(\mathfrak g^{\der})^{I_{F_\nu}}$. As $\rho_r$ is unramified at $\nu$, this is the whole of $\rho_r(\mathfrak g^{\der})$, so $Z_{r,\nu}^\mathrm{un}$ consists exactly of the cocycles $\phi_X^{\mathrm{un},r}$ for $X \in \mathfrak g^{\der}$. Hence, by Remark \ref{3.17} we conclude that the difference between $Z_{r,\nu}^\mathrm{un}$ and $Z_{r,\nu}^\alpha$ is that the first contains the cocycle $\phi_{H_\alpha}^{\mathrm{un},r}$, while the second one contains the ramified cocycle $\phi_{\alpha}^r$. As $\rho_r$ is unramified the coboundaries are unramified, as well. Hence, the ramified cocycle can't be cohomologous to any unramified cocycle. As both $Z_{r,\nu}^\mathrm{un}$ and $Z_{r,\nu}^\alpha$ contain the coboundaries $B^1(\Gamma_{F_\nu}.\rho_r(\mathfrak g^{\der}))$ we deduce that

        $$L_{r,\nu}^\alpha \cap H^1_{\mathrm{unr}}(\Gamma_{F_\nu},\rho_r(\mathfrak g^{\der})) \hookrightarrow L_{r,\nu}^\alpha$$

        \vspace{2 mm}

        \noindent has a cokernel generated by the class of $\phi_\alpha^r$ and is isomorphic to $\mathcal O/\varpi^r$. We also get that the inclusion

        $$L_{r,\nu}^\alpha \cap H^1_{\mathrm{unr}}(\Gamma_{F_\nu},\rho_r(\mathfrak g^{\der})) \hookrightarrow H^1_{\mathrm{unr}}(\Gamma_{F_\nu},\rho_r(\mathfrak g^{\der}))$$

        \vspace{2 mm}

        \noindent has a cokernel isomorphic to $\mathcal O/\varpi^r$, which is necessarily generated by the class of $\phi_{H_\alpha}^{\mathrm{un},r}$. The statement about the cokernel of the dual cohomology groups follows similarly by using the fact that the annihilator of the unramified cocycles under the local duality pairing consists exactly of the dual unramified cocycles, the inertia group is trivial under the cyclotomic character and the computations in \cite[Lemma $3.9$]{FKP21}.

        \end{proof}

        \begin{rem}

            \label{3.23}

            If the $\Gamma_{F_\nu}$-action on $\rho_r(\mathfrak g^{\der})$ is trivial, then there will be no non-trivial coboundaries, so we can work with cocycles instead of cohomology classes. Therefore, we can make sense of the images of $\sigma_\nu$ and $\tau_\nu$ under those cocycles. Using the explicit definition of the cocycles given in Remark \ref{3.17} we can express the difference between $L_{r,\nu}^\alpha$ and $H_{\mathrm{unr}}^1(\Gamma_{F_\nu},\rho_r(\mathfrak g^{\der}))$ in terms of the images of $\sigma_\nu$ and $\tau_\nu$. In particular, from the lemma $\phi \in H_{\mathrm{unr}}^1(\Gamma_{F_\nu},\rho_r(\mathfrak g^{\der}))$ will not be an element of $L_{r,\nu}^\alpha$ if and only if $\phi(\sigma_\nu) \notin \ker(\restr{\alpha}{\mathfrak t}) \oplus \bigoplus_\beta \mathfrak g_{\beta}$. Similarly, $\psi \in H^1_{\mathrm{unr}}(\Gamma_{F_\nu},\rho_r(\mathfrak g^{\der})^*)$ will not be an element of $L_{r,\nu}^{\alpha,\perp}$ if and only if $\psi(\sigma_\nu) \notin \mathfrak g_\alpha^\perp$, where $\mathfrak g_\alpha^\perp$ is the annihilator of $\mathfrak g_\alpha$ under the local duality.

            The results from the lemma will be needed in their full generality, i.e. when $\rho_r$ is unramified and not just trivial in some computations in \S 4. However, during the relative deformation method for the new chosen prime $\nu$ it is important that a given cocycle lies in the difference of the unramified cocycles and $Z_{r,\nu}^\alpha$. To control this more easily, we will choose $\nu$ so that the $\Gamma_{F_\nu}$-action is trivial on $\rho_r(\mathfrak g^{\der})$, which can be achieved by choosing it from the set $Q_{n,r}$ for suitably large $n$.
            
        \end{rem}

        We assume that the local conditions $\mathcal L_{r,S_N}$ for $1 \le r \le M$ are balanced. As we will need to apply some of the later results recursively, it is important to verify that enlarging $S_N$ by some finite subset of $Q_{N,M}$ will keep the local conditions balanced, where the local conditions at the new primes are $L_{r,\nu}^\alpha$. For notational convenience, we will drop the root notation when it does not play a significant role in the proof.

    \begin{lem}

        \label{3.24}

        Let $Q$ be a finite subset of $Q_{N,M}$. The local conditions $\mathcal L_{r,S_N \cup Q}$ for $1 \le r \le M$ are balanced. 
        
    \end{lem}
    \begin{proof}

        Applying the Greenberg-Wiles formula (\cite[Theorem $2.18$]{DDT97}) to $\mathcal L_{S_N}$ we get:

        $$\frac{|\, H_{\mathcal L_{r,S_N}}^1(\Gamma_{F,S_N},\rho_{r}(\mathfrak g^{\der}))\,|}{|\, H_{\mathcal L_{r,S_N}^\perp}^1(\Gamma_{F,S_N},\rho_{r}(\mathfrak g^{\der})^*)\,|} = \frac{|\,H^0(\Gamma_{F},\rho_r(\mathfrak g^{\der}))\,|}{|\,H^0(\Gamma_F,\rho_r(\mathfrak g^{\der})^*)\,|} \prod_{\nu \in S_N} \frac{|\,L_{r,\nu}\,|}{|\,H^0(\Gamma_{F_\nu},\rho_r(\mathfrak g^{\der}))\,|}$$

        \vspace{2 mm}

        As we assume that $\mathcal L_{r,S_N}$ are balanced, the left-hand side is equal to $1$. Applying, the formula to $\mathcal L_{r,S_N \cup Q}$ we get:

        $$\frac{|\, H_{\mathcal L_{r,S_N \cup Q}}^1(\Gamma_{F,S_N \cup Q},\rho_{r}(\mathfrak g^{\der}))\,|}{|\, H_{\mathcal L_{r,S_N \cup Q}^\perp}^1(\Gamma_{F,S_N \cup Q},\rho_{r}(\mathfrak g^{\der})^*)\,|} = \frac{|\,H^0(\Gamma_{F},\rho_r(\mathfrak g^{\der}))\,|}{|\,H^0(\Gamma_F,\rho_r(\mathfrak g^{\der})^*)\,|} \prod_{\nu \in S_N \cup Q} \frac{|\,L_{r,\nu}\,|}{|\,H^0(\Gamma_{F_\nu},\rho_r(\mathfrak g^{\der}))\,|}$$

        \vspace{2 mm}

        By part $(2)$ of Lemma $\ref{3.22}$ all the factors in the product that correspond to primes in $Q$ will be equal to $1$. Therefore, the right-hand side of the second equation reduces to the right-hand side of the first equation. This means that $|\,H_{\mathcal L_{r,S_N \cup Q}}^1(\Gamma_{F,S_N \cup Q},\rho_{r}(\mathfrak g^{\der}))\,| = |\,H_{\mathcal L_{r,S_N \cup Q}^\perp}^1(\Gamma_{F,S_N \cup Q},\rho_{r}(\mathfrak g^{\der})^*)\,|$, i.e. the local conditions $\mathcal L_{r,S_N \cup Q}$ are balanced.
        
    \end{proof}
    
    From now on we assume that $G^0$ is adjoint and $\mathfrak g^{\der} = \mathfrak g$ consists of a single $\pi_0(G)$-orbit of simple factors.

    \begin{prop}[{\cite[Proposition $6.8$]{FKP21}}]

        \label{3.26}

        Let $Q$ be a finite subset of $Q_{N,M}$. Suppose we can find $\phi \in H^1_{\mathcal L_{M,S_N \cup Q}}(\Gamma_{F,S_N \cup Q},\rho_M(\mathfrak g^{\der}))$ and $\psi \in H^1_{\mathcal L_{M,S_N \cup Q}^\perp}(\Gamma_{F,S_N \cup Q},\rho_M(\mathfrak g^{\der})^*)$ such that their reductions $\overline{\phi}$ and $\overline{\psi}$ are non-zero. Then, there exists a prime $\nu \in Q_{N,M}$ with an associated torus and root $(T,\alpha)$ such that

        \begin{itemize}
            \item $\restr{\overline{\psi}}{\Gamma_{F_\nu}} \notin L_{1,\nu}^{\alpha,\perp}$; and
            \item $\restr{\phi}{\Gamma_{F_\nu}} \notin L_{M,\nu}^{\alpha}$.
        \end{itemize}
        
    \end{prop}
    \begin{proof}

        We will prove a more general statement in Proposition \ref{4.16}, so we will omit the proof for the sake of not repeating ourselves.
        
    \end{proof}

    \begin{thm}[{\cite[Theorem $6.9$]{FKP21}}]

        \label{3.27}

        There exists a finite set $Q \subset Q_{N,M}$ such that

        $$\overline{H^1_{\mathcal L_{M,S_N \cup Q}}(\Gamma_{F,S_N \cup Q},\rho_M(\mathfrak g^{\der}))} = 0$$

        \vspace{2 mm}

        \noindent and

        $$\overline{H^1_{\mathcal L^\perp_{M,S_N \cup Q}}(\Gamma_{F,S_N \cup Q},\rho_M(\mathfrak g^{\der})^*)} = 0$$

        \vspace{2 mm}

        \noindent where the local conditions at primes in $Q$ are $L_{M,\nu}^\alpha$, defined in Lemma \textup{\ref{3.16}}.

    \end{thm}
    
    \begin{rem}

        \label{3.28}

        The main idea behind the proof is to recursively lower the size of $H^1_{\mathcal L_{M,S_N}}(\Gamma_{F,S_N},\rho_M(\mathfrak g^{\der}))$ by adding primes produced by Proposition \ref{3.26}. The condition on $\overline{\psi}$ will guarantee that changing the local conditions from the unramified ones to $L_{M,\nu}^\alpha$ will not introduce any new cocycles, while the condition on $\phi$ guarantees that we have also eliminated some cocycles. The computations in Proposition \ref{4.18} and Remark \ref{4.19} go into more details on why the conditions we imposed on the cocycles give us this.
        
    \end{rem}

\subsection{Lifting to characteristic \texorpdfstring{$0$}{0}}

Having constructed a high enough lift of $\overline{\rho}$ using the doubling method and having annihilated a relative Selmer group we can produce a $G(\mathcal O)$-lift of $\overline{\rho}$ using the main result in \cite{FKP22}:

\begin{thm}[{\cite[Theorem $5.2$]{FKP22}}]

    \label{3.29}

    Let $p \gg_G 0$. Then for some finite set of primes $T$ containing $S$, there is a geometric lift

        $$\rho:\Gamma_{F,T} \to G(\overline{\mathbb Z_p})$$

    \vspace{2 mm}

    More precisely, we fix an integer $t$ and for each $\nu \in S$ and irreducible component containing $\rho_\nu$ of

    \begin{itemize}
        \item for $\nu \in S \setminus \{\nu \mid p\}$, the generic fiber of the local lifting ring $R^{\sqr,\mu}_{\restr{\overline{\rho}}{\Gamma_{F_\nu}}}[1/\varpi]$; and
        \item for $v \mid p$, the lifting ring $R^{\sqr,\mu,\tau,\textup{\textbf{v}}}_{\restr{\overline{\rho}}{\Gamma_{F_\nu}}}[1/\varpi]$, whose $\overline{E}$-points parametrize lifts of $\restr{\overline{\rho}}{\Gamma_{F_\nu}}$ with specified inertial type $\tau$ and $p$-adic Hodge type $\textup{\textbf{v}}$.
    \end{itemize}

    Then, there exists a finite extension $E'$ of $E = \mathrm{Frac}(\mathcal O)$, whose ring of integers and residue field we denote by $\mathcal O'$ and $k'$, respectively, which depends only on the set $\{\rho_\nu\}_{\nu \in S}$; a finite set of places $T$ containing $S$ and a geometric lift:

    \begin{center}

        \begin{tikzcd}
                                                                   &  & G(\mathcal O') \arrow[dd] \\
                                                                   &  &                           \\
        {\Gamma_{F,T}} \arrow[rruu, "\rho"] \arrow[rr, "\overline{\rho}"'] &  & G(k')                    
        \end{tikzcd}
        
    \end{center}

    \vspace{2 mm}

    of $\overline{\rho}$ such that

    \begin{itemize}
        \item $\rho$ has multiplier $\mu$.
        \item $\rho(\Gamma_F)$ contains $\widehat{G^{\der}}(\mathcal O')$.
        \item For all $\nu \in S$, $\restr{\rho}{\Gamma_{F_\nu}}$ is congruent modulo $\varpi^t$ to some $\widehat{G}(\mathcal O')$-conjugate of $\rho_v$ and $\restr{\rho}{\Gamma_{F_\nu}}$ belongs to the specified irreducible component for every $\nu \in S$.
    \end{itemize}
    
\end{thm}

\begin{rem}
\label{G^0Adjoint}
    
    By \cite[Claim $6.13$]{FKP21} it suffices to prove this theorem in the case when $G^0$ is an adjoint group and $\mathfrak g$, now equal to $\mathfrak g^{\der}$ is equal to a single $\pi_0(G)$-orbit of simple factors. Therefore, we have lost nothing by assuming this properties, as we did earlier in \S 3.2.
    
\end{rem}

We will give a short summary of the proof, illustrating the main technical details. We select an integer $M$ in a manner explained in the previous subsection. In particular, we ask for $M \ge \max\{eD,M_1\}$, where $M_1$ is the integer coming from Lemma \ref{3.20}, and for $M$ to be divisible by $e$. 

For each $\nu \in S$ let $\overline{R}_\nu[1/\varpi]$ be the chosen irreducible component of $R^{\sqr,\mu}_{\restr{\overline{\rho}}{\Gamma_{F_\nu}}}[1/\varpi]$ (for $\nu \nmid p$) or $R^{\sqr,\mu,\tau,\textbf{v}}_{\restr{\overline{\rho}}{\Gamma_{F_\nu}}}$ (for $\nu \mid p$). Let $\overline{R}_\nu$ be its scheme-theoretic closure in the local lifting ring $R^{\sqr,\mu}_{\restr{\overline{\rho}}{\Gamma_{F_\nu}}}$. By \cite[Theorem $3.3.3$]{BG19} we know that $\overline{R}_\nu[1/\varpi]$ has a dense subset of formally smooth closed points. Then, for the fixed integer $t$ using \linebreak \cite[Lemma $4.9$]{FKP21} we can produce a finite extension $\mathcal O'$ of $\mathcal O$, independent of $t$, and lifts $\rho_\nu'$ of $\restr{\overline{\rho}}{\Gamma_{F_\nu}}$ corresponding to $\mathcal O'$-points of $\overline{R}_\nu$ such that $\rho_\nu' \equiv \rho_\nu \pmod{\varpi^t}$ and $\rho_\nu'$ defines a formally smooth point on $\overline{R}_\nu[1/\varpi]$. We replace $\mathcal O$ by $\mathcal O'$ and $\rho_\nu$ by $\rho_\nu'$, but we will keep using the original notation. 

We now run the first $eD$ steps of the doubling method, producing a lift $\rho_{eD}:\Gamma_{F,S_{eD}} \to G(\mathcal O/\varpi^{eD})$ that satisfies the conclusion of Theorem \ref{3.4}. We recall that for all primes  $\nu \in S_{eD} \setminus S$ at which we have ramification by part $(a)$ of the theorem we get local $G(\mathcal O)$-valued lifts $\rho_\nu$ that correspond to formally smooth points on a suitable irreducible component $\overline{R}_\nu[1/\varpi]$ of the generic fiber of the local lifting ring $R^{\sqr,\mu}_{\restr{\overline{\rho}}{\Gamma_{F_\nu}}}$. 

Now, for each prime $\nu$ in $S_{eD}$ at which $\rho_{eD}$ is ramified we apply \cite[Proposition $4.7$]{FKP21} with $r_0 = M$, the formally smooth local lift $\rho_\nu$ and irreducible component $\overline{R}_\nu[1/\varpi]$. This will provide us with cocycles $Z_{r,\nu} \subseteq Z^1(\Gamma_{F_\nu},\rho_r(\mathfrak g^{\der}))$ that in particular satisfy the conditions mention in Assumption \ref{3.10}. The proposition also produces a class of local lifts $D_{n,\nu} \subseteq \Spec(\overline{R}_\nu)(\mathcal O/\varpi^n)$ for each $n \ge 1$. Moreover, there exists an integer $n_0$ such that for $n \ge n_0$ and $r \le M$ the fibers of the reduction maps $D_{n+r,\nu} \to D_{n,\nu}$ are nonempty principal homogeneous spaces over $Z_{r,\nu}$. Let $N_0$ be the maximum of all the $n_0$ produced by applying the proposition for each such prime. On the other hand, by \cite[Lemma $6.15$]{FKP21} we get that for $N_1$ large enough if $\rho_{N_1}$ contains $\widehat{G^{\der}}(\mathcal O/\varpi^{N_1})$, then the same is true for every $n \ge N_1$. We then choose an integer $N$ such that:

\begin{equation}
    \label{Nbound} N \ge \max\{N_0+M,(N_1+1)+M,t+M,4M\}  \tag{$1$}
\end{equation}

\vspace{2 mm}

\noindent In addition, we also ask for it to be divisible by $e$. We then start running the doubling method, continuing from $\rho_{eD}$ until we produce a lift $\rho_N:\Gamma_{F,S_N} \to G(\mathcal O/\varpi^N)$. We quickly explain why we impose each of these bounds. We ask for $N \ge N_0+M$ so that the local lifting conditions at primes in $S_{eD}$ at which $\rho_{eD}$ is ramified satisfy the desired properties, i.e. the conclusions in \cite[Proposition $4.7$]{FKP21}. We want $N \ge (N_1+1) + M$, since the doubling method guarantees that any lift produced by it will have a maximal image, and as noted above this will be true for every lift modulo $\varpi^n$ for $n \ge N_1$, as well as for the characteristic $0$ lift $\rho$. We remark that because of the modifications we made to part $(3)$ of Theorem \ref{3.4} we need to prescribe the image of the unramified lifts one level higher, so in particular we need to run the doubling method at least $N_1+1$ times to achieve what we want. Asking for $N \ge t+M$ allows us to control the local behavior of $\rho$ at the primes in $S$ up to the desired level of precision. Finally, $N \ge 4M$ allows us to apply Lemma \ref{3.22}, so the local lifting conditions at the primes added during the relative deformation argument will be as desired. This will also make sure that $N$ is relatively large compared to $M$, so we can start building upwards from a lift $\rho_N$ modulo high enough power of $\varpi$.

Eventually, for all primes $\nu \in S_N$ we end up with subspaces $L_{r,\nu} \subseteq H^1(\Gamma_{F_\nu},\rho_r(\mathfrak g^{\der}))$ for $1 \le r \le M$ and classes of lifts $D_{n,\nu} \subseteq \Lift_{\restr{\overline{\rho}}{\Gamma_{F_\nu}}}(\mathcal O/\varpi^n)$ for $n \ge N-M$, such that for all $1 \le r \le M$ the fibers of $D_{n+r,\nu} \to D_{n,\nu}$ are non-empty principal homogenous spaces under $L_{r,\nu}$. Before using the relative deformation argument in this case, it is important to check that these local conditions are balanced.

\begin{claim}

    \label{3.30}

    For $1 \le r \le M$, the local conditions $\mathcal L_{r,S_N} \coloneqq \{L_{r,\nu}\}_{\nu \in S_N}$ are balanced. 
    
\end{claim}

\begin{proof}

For primes $\nu$ above $p$, the lifts $\rho_\nu$ given in Assumptions \ref{1.1} are Hodge-Tate regular. Thus, the parabolic subgroup associated to the Hodge type \textbf{v} will be Borel, and so the local conditions $L_{r,\nu}$ have size $|\,H^0(\Gamma_{F_\nu},\rho_r(\mathfrak g^{\der}))\,| \cdot |\mathcal O/\varpi^r|^{[F_\nu:\mathbb Q_p]\cdot \dim_k(\mathfrak n)}$, where $\mathfrak n$ is the Lie algebra of the unipotent radical of the Borel subgroup of $G$ associated to \textbf{v}. On the other hand, for primes $\nu$ not lying over $p$ the local conditions have size $|\,H^0(\Gamma_{F_\nu},\rho_r(\mathfrak g^{\der}))\,|$. This follows either from \cite[Proposition $4.7$]{FKP21} or the discussion in Remark \ref{3.18}. Finally, for infinite places $\nu$, the local conditions are trivial, while the assumption on $\overline{\rho}$ being odd gives us that $|\,H^0(\Gamma_{F_\nu},\rho_{r}(\mathfrak g^{\der}))\,| = |\mathcal O/\varpi^r|^{\dim_k(\mathfrak n)}$. From the Greenberg-Wiles formula we get:

$$\frac{|\,H^1_{\mathcal L_{r,S_N}}(\Gamma_{F,S_N},\rho_{r}(\mathfrak g^{\der}))\,|}{|\,H^1_{\mathcal L_{r,S_N}^\perp}(\Gamma_{F,S_N},\rho_{r}(\mathfrak g^{\der})^*)\,|} = \frac{|\,H^0(\Gamma_F,\rho_{r}(\mathfrak g^{\der}))\,|}{|\,H^0(\Gamma_F,\rho_{r}(\mathfrak g^{\der})^*)\,|} \prod_{\nu \in S_N} \frac{|\,L_{r,\nu}\,|}{|\,H^0(\Gamma_{F_\nu},\rho_{r}(\mathfrak g^{\der}))\,|}$$

\vspace{2 mm}

As neither $\overline{\rho}(\mathfrak g^{\der})$ nor $\overline{\rho}(\mathfrak g^{\der})^*$ contain the trivial representation the same is is true for $\rho_r(\mathfrak g^{\der})$ and $\rho_r(\mathfrak g^{\der})^*$. Thus, the first factor on the right is equal to $1$. Using the remarks on the sizes of the local conditions the product on the right reduces to:

$$\prod_{\nu \mid p} |\mathcal O/\varpi^r|^{[F_\nu:\mathbb Q_p]\cdot \dim_k(\mathfrak n)} \cdot \prod_{\nu \mid \infty} |\mathcal O/\varpi^r|^{-\dim_k(\mathfrak n)} = |\mathcal O/\varpi^r|^{\left(-[F:\mathbb Q] + \sum_{\nu \mid p} [F_\nu:\mathbb Q_p]\right)\dim_k(\mathfrak n)} = 1$$

\vspace{2 mm}

\noindent where we used that $F$ is a totally real field, so the number of non-Archimedean places is $[F:\mathbb Q]$. This gives us that $|\,H^1_{\mathcal L_{r,S_N}}(\Gamma_{F,S_N},\rho_{r}(\mathfrak g^{\der}))\,| = |\,H^1_{\mathcal L_{r,S_N}^\perp}(\Gamma_{F,S_N},\rho_{r}(\mathfrak g^{\der})^*)\,|$, i.e. the local conditions $\mathcal L_{r,S_N}$ are balanced. 

\end{proof}

Now, from Theorem \ref{3.27} we find a subset $Q$ of $Q_{N,M}$, such that $\overline{H^1_{\mathcal L_{M,S_N \cup Q}}(\Gamma_{F,S_N \cup Q},\rho_{M}(\mathfrak g^{\der})^*)} = 0$. \linebreak \cite[Claim $6.14$]{FKP21} tells us that this is enough to get a lift $\rho$ as in the statement of Theorem \ref{3.29}. Since we will constantly invoke this result, as well as some ideas from its proof during the forcing ramification argument, we take our time to reproduce the proof.

\begin{claim}[{\cite[Claim $6.14$]{FKP21}}]

    \label{3.31}
    
    For $n \ge N - M$, we have pairs of (multiplier $\mu$) lifts $(\tau_n,\rho_{n+M})$, where $\tau_n:\Gamma_{F,S_N \cup Q} \to G(\mathcal O/\varpi^n)$, $\rho_{n+M}:\Gamma_{F,S_N \cup Q} \to G(\mathcal O/\varpi^{n+M})$, and $\tau_n = \rho_{n+M} \pmod{\varpi^n}$ such that:

    \begin{enumerate}
        \item For each $\nu \in S_N \cup Q$, $\restr{\tau_n}{\Gamma_{F_\nu}} \in D_{n,\nu}$ and $\restr{\rho_{n+M}}{\Gamma_{F_\nu}} \in D_{n+M,\nu}$.
        \item $\tau_{n+1} = \rho_{n+M} \pmod{\varpi^{n+1}}$.
        \item $\tau_n = \tau_{n+1} \pmod{\varpi^n}$.
    \end{enumerate}
    
\end{claim}
\begin{proof}

    We will prove this claim by induction on $n$. For $n = N-M$ we take $\tau_{N-M} = \rho_{N-M}$. Clearly, properties $(2)$ and $(3)$ are satisfied, while property $(1)$ holds for primes in $S_N$ as this was arranged in the doubling method. On the other hand, property $(1)$ holds for primes in $Q$, as they satisfy the third property in Definition \ref{3.21}.

    Assume that we have such a pair $(\tau_n,\rho_{n+M})$ for some $n \ge N-M$. We set $\tau_{n+1} = \rho_{n+M} \pmod{\varpi^{n+1}}$. As $\restr{\rho_{n+M}}{\Gamma_{F_\nu}} \in D_{n+M,\nu}$ for all $\nu \in S_N \cup Q$ and as the fibers are non-empty we can find find local lifts of them modulo $\varpi^{n+M+1}$. We arranged $\Sh^1_{S_N}(\Gamma_{F,S_N},\overline{\rho}(\mathfrak g^{\der})^*) = 0$, which implies that $\Sh^2_{S_N \cup Q}(\Gamma_{F,S_N \cup Q},\overline{\rho}(\mathfrak g^{\der})) = 0$ by using Poitou-Tate's duality. This means that we can lift $\rho_{n+M}$ to $\rho'_{n+M+1}:\Gamma_{F,S_N \cup Q} \to G(\mathcal O/\varpi^{n+M+1})$. For each $\nu \in S_N \cup Q$, $\restr{\rho'_{n+M+1}}{\Gamma_{F_\nu}}$ is a lift of $\restr{\tau_{n+1}}{\Gamma_{F_\nu}}$, so comparing it with an element of the fiber of $D_{n+M+1,\nu} \to D_{n+1,\nu}$ over $\restr{\tau_{n+1}}{\Gamma_{F_\nu}}$ we end up with an element:

    $$(f_\nu)_{\nu \in S_N \cup Q} \in \bigoplus_{\nu \in S_N \cup Q} \frac{H^1(\Gamma_{F_\nu},\rho_{M}(\mathfrak g^{\der}))}{L_{M,\nu}}$$

    \vspace{2 mm}

    As $\rho_{n+M+1}'$ lifts $\rho_{n+M}$ and by the inductive hypothesis for all $\nu \in S_N \cup Q$, $\restr{\rho_{n+M}}{\Gamma_{F_\nu}}$ lies in the fiber of $D_{n+M,\nu} \to D_{n+1,\nu}$ over $\restr{\tau_{n+1}}{\Gamma_{F_\nu}}$, we get that the reduction of $(f_\nu)$ vanishes in $\bigoplus_{\nu \in S_N \cup Q} \frac{H^1(\Gamma_{F_\nu},\rho_{M-1}(\mathfrak g^{\der}))}{L_{M-1,\nu}}$. We now have the following commutative diagram:

    \begin{center}

        \begin{tikzcd}
            {H^1(\Gamma_{F,S_N \cup Q},\rho_{M}(\mathfrak g^{\der}))} \arrow[r] \arrow[d] &  {\bigoplus_{\nu \in S_N \cup Q} \frac{H^1(\Gamma_{F_\nu},\rho_{M}(\mathfrak g^{\der}))}{L_{M,\nu}}} \arrow[r] \arrow[d] &  {H^1_{\mathcal L^\perp_{M,S_N \cup Q}}(\Gamma_{F,S_N \cup Q},\rho_{M}(\mathfrak g^{\der})^*)^\vee} \arrow[d] \\
            {H^1(\Gamma_{F,S_N \cup Q},\rho_{M-1}(\mathfrak g^{\der}))} \arrow[r]         &  {\bigoplus_{\nu \in S_N \cup Q} \frac{H^1(\Gamma_{F_\nu},\rho_{M-1}(\mathfrak g^{\der}))}{L_{M-1,\nu}}} \arrow[r]       &  {H^1_{\mathcal L^\perp_{M-1,S_N \cup Q}}(\Gamma_{F,S_N \cup Q},\rho_{M-1}(\mathfrak g^{\der})^*)^\vee}      
        \end{tikzcd}
        
    \end{center}

    \vspace{2 mm}

    \noindent where the rows come from the Poitou-Tate exact sequence and the vertical maps are induced by reduction modulo $\varpi^{M-1}$. Now, the exact sequence in Lemma \ref{3.14}, coupled with the vanishing of the $M$-relative dual Selmer group gives us that the map:

    $$H^1_{\mathcal L^\perp_{M-1,S_N \cup Q}}(\Gamma_{F,S_N \cup Q},\rho_{M-1}(\mathfrak g^{\der})^*) \to H^1_{\mathcal L^\perp_{M,S_N \cup Q}}(\Gamma_{F,S_N \cup Q},\rho_{M}(\mathfrak g^{\der})^*)$$

    \vspace{2 mm}

    \noindent is surjective. Dualizing we get that the third vertical map in the commutative diagram above is injective. This implies that $(f_\nu)$ maps to zero in $H^1_{\mathcal L^\perp_{M,S_N \cup Q}}(\Gamma_{F,S_N \cup Q},\rho_{M}(\mathfrak g^{\der})^*)^\vee$, which using the exactness of the top row means that $(f_\nu)$ comes from a global cocycle $f \in H^1(\Gamma_{F,S_N \cup Q},\rho_{M}(\mathfrak g^{\der}))$. We modify $\rho'_{n+M+1}$ by this cocycle, by setting $\rho_{n+M+1} \coloneqq \exp(\varpi^{n+1}f)\rho_{n+M+1}'$. Using the fact that the fibers $D_{n+M+1,\nu} \to D_{n+1,\nu}$ are stable under the action of elements in $L_{M,\nu}$ for all $n \in S_N \cup Q$ this modification ensures that $\rho_{n+M+1}$ satisfies property $(1)$ of the claim. Moreover, this is a lift of $\tau_{n+1}$ and clearly the rest of the properties are satisfied. 
    
\end{proof}

Finally, we take $\rho \coloneqq \varprojlim \tau_n: \Gamma_{F,S_N \cup Q} \to G(\mathcal O)$. As remarked in the choice of the bound for $N$, this lift will have a maximal image and will approximate the initial lifts at primes in $S$ modulo $\varpi^t$. Hence, this is a lift that satisfies the statement of the theorem.

\section{Producing a locally smooth lift}

In this section we will first show that the modification we did in \S $3.1$ to the doubling method will guarantee the formal smoothness at primes in $S_N$ for every characteristic $0$ lift obtained by the method explained in \S $3.3$. After that we will prove that the question of formal smoothness at primes in $Q$ reduces to the question of whether the lift $\rho$ is ramified at those primes. In general, the lifting method on its own doesn't guarantee this. However, borrowing ideas from \cite{KR03} and \cite{Pat17} we will show that this can be achieved by changing the set $Q$ and subsequently the characteristic $0$ lift $\rho$. In order to produce this new lift we will use the relative deformation method, which as seen from Theorem \ref{3.29} necessitates that $p \gg_G 0$. This will be a running assumption throughout this section and we will often leave it implicit in the statement of the results.

The criterion of Proposition \ref{3.1} will play a key role in what follows, so we recall its statement which says that for primes $\nu \nmid p$ a point $x$ is formally smooth point in $R_{\restr{\overline{\rho}}{\Gamma_{F_\nu}}}^{\sqr,\mu}[1/p]$ if and only if $H^0(\Gamma_{F_\nu},\rho_x(\mathfrak g^{\der})^*) = 0$.

\subsection{Smoothness at primes in \texorpdfstring{$S_N$}{S\textunderscore N}}

The set $S_N$ consists of the primes in the original ramifying set $S$ together with all the primes added during the doubling method. As seen from Theorem \ref{3.4} primes in $S_N$ will have different local behaviors depending on which parts of the theorem they satisfy.
Based on this we divide the primes in $S_N$ into $4$ different types:

\begin{enumerate}
    \item First, the primes in the original set $S$, the ramifying set of $\overline{\rho}$. By part $(2)$ of the theorem for $\nu \in S$ at each step of the doubling method we make sure that $\restr{\rho_n}{\Gamma_{F_\nu}}$ is strictly equivalent to $\rho_{\nu,n} \coloneqq \rho_\nu \pmod{\varpi^n}$, where $\rho_\nu$ are the lifts given in Assumptions \ref{1.1}.
    \item The second type are the primes added in the first $eD$ steps of the doubling method and at which $\rho_{eD}$ is ramified. The local behaviour at these primes is controled by part $(a)$ of the theorem. More precisely, at those primes $\nu$ using \cite[Lemma $3.7$]{FKP21} we produce characteristic $0$ local lifts $\rho_\nu$ and as above we make sure that at each step $\restr{\rho_n}{\Gamma_{F_\nu}}$ is strictly equivalent to $\rho_{\nu,n}$.
    \item The next type of primes consists of the primes added after the $eD$-th step of the doubling method and at which $\rho_N$ is ramified. At such primes $\nu$, as seen from the statement of part $(b)$ of Theorem \ref{3.4} we impose some conditions on the images of $\sigma_\nu$ and $\tau_\nu$ under $\rho_n$.
    \item Finally, we have the primes in $S_N$ at which $\rho_N$ is not ramified. These primes ensure that the doubling method can run. In particular, some of them are added to $S_N$ in order to make sure that $\Sh_{S_N}^2(\Gamma_{F,S_N},\overline{\rho}(\mathfrak g^{\der})) = 0$, while others are added to guarantee that at each step the lift $\rho_n$ has maximal image, i.e. part $(3)$ of the theorem is satisfied. At such primes $\nu$ we produced characteristic $0$ unramified local lifts $\rho_\nu$ with a prescribed image of $\sigma_\nu$ and we ask for $\restr{\rho_n}{\Gamma_{F_\nu}}$ to be strongly equivalent to $\rho_{\nu,n}$ during each step of the doubling method.
    
\end{enumerate}

We recall that for the first two types of primes we didn't make any modifications compared to the original doubling method. This is because the local lifting conditions at those primes already guarantee the formal smoothness at them. At each such prime $\nu$ we have smooth local lifts $\rho_\nu$, whose existence comes from either Assumptions \ref{1.1} (recall we potentially changed these lifts in \S $3.3$ without altering the notation) or \cite[Lemma $3.7$]{FKP21}. Then, as elaborated in \cite[\S $4$]{FKP21}, using these lifts we find a class of smooth lifts $D_{\nu} \subseteq \Lift_{\restr{\overline{\rho}}{\Gamma_{F_\nu}}}^{\mu}(\mathcal O)$. We let $D_{n,\nu}\subseteq \Lift_{\restr{\overline{\rho}}{\Gamma_{F_\nu}}}^{\mu}(\mathcal O/\varpi^n)$ be the set of their reductions modulo $\varpi^n$. By \cite[Proposition $4.7$]{FKP21} for $r \ge 1$ and $n$ large enough the fibers of the maps $D_{n+r,\nu} \to D_{n,\nu}$ will be non-empty principal homogeneous under the action of certain $1$-cocycles. We use these cocycles to define the local lifting conditions at $\nu$. These conditions will have the right size and also will satisfy the necessary conditions mentioned at the beginning of the relative deformation argument. From the way the relative deformation argument works $\restr{\rho}{\Gamma_{F_\nu}}$ will lie in $D_\nu$, which gives us the formal smoothness at these primes.

Next, we focus on the third type of primes in $S_N$. These primes will satisfy part $(b)$ of Theorem \ref{3.4} and $\rho_N$ will be ramified at them. The ramification will play a major role in establishing the smoothness at these primes. However, looking closer at Claim \ref{3.31} we can only guarantee that $\rho$ will be ramified at primes in $S_{N-M}$. This is because even though we use $\rho_N$ as an auxiliary lift and we have to allow for ramification at primes in $S_N$, the relative deformation arguments start building upwards from $\rho_{N-M}$. Therefore, for the time being we will only prove the formal smoothness at primes of the third type found in $S_{N-M}$. On the other hand, we move the primes of the third type in $S_N \setminus S_{N-M}$ to $Q$ and we will treat them together in \S $4.2$. We note that the local conditions at primes in $S_N \setminus S_{N-M}$ and the primes that are already in Q are defined in the same manner, using Lemma \ref{3.16}. Therefore, from now on we will not distinguish between these primes.

Using the modifications to part $(b)$ of Theorem \ref{3.4} we can now prove the following result:

 \begin{lem}
     \label{4.1}

    Let $\nu$ be a prime of the third type in $S_{N-M}$. Then, for $n \ge N-M$ we have that 
    
    $$H^0(\Gamma_{F_\nu},\rho_n(\mathfrak g^{\der})^*) \subseteq \varpi^{n-(N-M-1)}\rho_{N-M-1}(\mathfrak g^{\der})^*$$

    \vspace{2 mm}

    In other words, the $\Gamma_{F_\nu}$-invariant elements of $\rho_n(\mathfrak g^{\der})^*$ are all multiples of $\varpi^{n-(N-M-1)}$. 
     
 \end{lem}
 \begin{proof}

    Firstly, using the perfect trace pairing we can identify $\mathfrak g^{\der}$ with its dual and thus $(\mathfrak g^{\der})^* \simeq \mathfrak g^{\der}(1)$. Now, let $X \in H^0(\Gamma_{F_\nu},\rho_n(\mathfrak g^{\der})^*)$. We know that $\rho_n(\sigma_\nu)$ lies in a split maximal torus $T$. So using the root decomposition of $\mathfrak g^{\der}$ we can write $X = Z + \sum_{\beta \in \Phi(G^0,T)} X_\beta$ with $Z \in \mathfrak t^{\der}(1)$ and $X_\beta \in \mathfrak g_\beta(1)$. Acting by $\sigma_\nu$ we get:

    $$X = \sigma_\nu \cdot X = qZ + \sum_{\beta \in \Phi(G^0,T)} q\beta(\rho_n(\sigma_\nu))X_\beta$$

    \vspace{2 mm}

    \noindent where $q = N(\nu)$. We now compare elements on both sides of the equation. First, we must have that $(q-1)Z = 0$. Since $q \equiv 1 \pmod{\varpi^{eD}}$, but $q \not \equiv 1 \pmod{\varpi^{eD+1}}$ we must have that $Z \pmod {\varpi^{n-eD}} = 0$, i.e. $Z$ is a multiple of $\varpi^{n-eD}$. Similarly, $(q\beta(\rho_n(\sigma_\nu))-1)X_\beta = 0$ for all roots $\beta$. As $\rho_n(\sigma_\nu)$ and $q$ are trivial modulo $\varpi^{eD}$ we get that $q\beta(\rho_n(\sigma_\nu)) \equiv 1 \pmod{\varpi^{eD}}$. On the other hand, these $\nu$ satisfies part $(b)$ of Theorem \ref{3.4}, so from the modifications we made to that part we have that $\beta(\rho_n(\sigma_\nu)) \not \equiv q^{-1} \pmod{\varpi^{eD+1}}$ for $\beta \neq -\alpha$. In particular, this means that $X_\beta \pmod{\varpi^{n-eD}} = 0$ for $\beta \neq -\alpha$. Hence, all the terms, except $X_{-\alpha}$ are multiples of $\varpi^{n-eD}$ and therefore of $\varpi^{n-(N-M-1)}$.

    To show that $X_{-\alpha}$ is a multiple of $\varpi^{n-(N-M-1)}$ we use the invariance under the action of $\tau_\nu$. Since, all other terms are multiples of $\varpi^{n-eD}$ and $\rho_{eD}$ isn't ramified at $\nu$ each of them individually will be invariant under $\tau_\nu$. Hence, we have that $X_{-\alpha}$ is invariant under the action of $\tau_\nu$, as well. Now, $\rho_{N-M}(\tau_\nu)$ is a non-trivial element of $U_\alpha(\mathcal O/\varpi^{N-M})$ and from the definition of the local conditions at $\nu$ in Lemma \ref{3.16} this remains true for $\rho_n$. Hence, we can write $\rho_n(\tau_\nu) = u_\alpha(yY_\alpha)$, where $Y_\alpha$ is a generator of $\mathfrak g_\alpha$ over $\mathcal O$ and $y \not \equiv 0 \pmod{\varpi^{N-M}}$ Then:

    $$X_{-\alpha} = \tau_\nu \cdot X_{-\alpha} = \Ad(u_\alpha(yY_\alpha))(X_{-\alpha}) = X_{-\alpha} + y[Y_\alpha,X_{-\alpha}] + \frac{y^2}2[Y_\alpha,[Y_\alpha,X_{-\alpha}]]$$

    \vspace{2 mm}

    We note that the expansion above terminates after the third term, since $[Y_\alpha,X_{-\alpha}] \in \mathfrak t$, $[Y_\alpha,[Y_\alpha,X_{-\alpha}]] \in \mathfrak g_\alpha$, and $[Y_\alpha,[Y_\alpha,[Y_\alpha,X_{-\alpha}]]] \in \mathfrak g_{2\alpha} = 0$. Comparing both sides, we get that the last two terms on the right vanish. Since $Y_\alpha \pmod{\varpi} \neq 0$ this means that $yX_{-\alpha} \pmod{\varpi^n} = 0$. For the last conclusion we use the fact that \linebreak $X_{-\alpha} \in \mathfrak g_{-\alpha}$ and $[Y_\alpha,Y_{-\alpha}] \pmod{\varpi} \neq 0$. Since $y \not \equiv 0 \pmod{\varpi^{N-M}}$ we have $X_{-\alpha} \pmod{\varpi^{n-(N-M-1)}} = 0$, \linebreak i.e. $X_{-\alpha}$ is a multiple of $\varpi^{n-(N-M-1)}$. This shows that $H^0(\Gamma_{F_\nu},\rho_n(\mathfrak g^{\der})^*)$ consists of multiples of $\varpi^{n-(N-M-1)}$, which in turn are element of the image of $\rho_{N-M-1}(\mathfrak g^{\der})^*$ under the multiplication by $\varpi^{n-(N-M-1)}$ map.
     
 \end{proof}

Using this lemma we conclude that for primes $\nu$ of the third type in $S_{N-M}$, as we move up the tower of $\Gamma_{F_\nu}$-invariants they will start being multiples of higher and higher powers of $\varpi$. In particular, for large $n$, the reduction modulo $\varpi^{n-(N-M-1)}$ will induce a map between $\Gamma_{F_\nu}$-invariants which is the zero map. Therefore, $H^0(\Gamma_{F_\nu},\rho(\mathfrak g^{\der})^*) = \varprojlim_n H^0(\Gamma_{F_\nu},\rho_n(\mathfrak g^{\der})^*) = 0$ and by Proposition \ref{3.1} $\restr{\rho}{\Gamma_{F_\nu}}$ is a formally smooth point in $R_{\restr{\overline{\rho}}{\Gamma_{F_\nu}}}^{\sqr,\mu}[1/\varpi]$.

It remains to prove that $\rho$ is formally smooth at the primes in $S_N$ of the fourth type, i.e. the primes at which $\rho_{N}$ isn't ramified. We'll show an analogous result as above, using the modifications to part $(3)$ of Theorem \ref{3.4}.

\begin{lem}

    \label{4.2}

    Let $\nu \in S_N$, such that $\rho_N$ isn't ramified at $\nu$. Let $s = \max\{eD,N_1\}$. Then for $n \ge s+1$ we have:

    $$H^0(\Gamma_{F_\nu},\rho_n(\mathfrak g^{\der})^*) \subseteq \varpi^{n-s}\rho_{s}(\mathfrak g^{\der})^*$$

    \vspace{2 mm}

    \noindent where $N_1$ comes from \cite[Lemma $6.15$]{FKP21} and has already appeared in our bound for $N$ in \textup{(\ref{Nbound})}.
    
\end{lem}
\begin{proof}

    The prime $\nu$ will satisfy the properties listed in part $(3)$ of Theorem \ref{3.4}. In particular, $\rho_n(\sigma_\nu)$ is an element of $T(\mathcal O/\varpi^n)$, for some maximal $\mathcal O$-torus $T$. The torus $T$ might not split over $\mathcal O$, but it will over a finite \'etale extension $\mathcal O'$ of $\mathcal O$. Thus, to make use of root decompositions of the Lie algebra we might be required to work with $\mathfrak g^{\der} \otimes_\mathcal O \mathcal O'$ instead of $\mathfrak g^{\der}$. However, we are interested in showing that the $\Gamma_{F_\nu}$-invariants are multiples of $\varpi^{n-s}$, and $\varpi$ is a uniformizer of $\mathcal O'$, too. Therefore, for the purposes of this lemma, in order to simplify the notation, we assume that $T$ splits over $\mathcal O$.

    Let $X \in H^0(\Gamma_{F_\nu},\rho_n(\mathfrak g^{\der})^*)$. Again, using the identification $(\mathfrak g^{\der})^* \simeq \mathfrak g^{\der}(1)$ and the root decomposition of $\mathfrak g^{\der}$ we write $X = Z + \sum_{\beta \in \Phi(G^0,T)} X_\beta$, where $Z \in \mathfrak t^{\der}(1)$ and $X_\beta \in \mathfrak g_\beta(1)$. We now act by $\sigma_\nu$:

    $$X = \sigma_\nu \cdot X = qZ + \sum_{\beta \in \Phi(G^0,T)} q\beta(\rho_n(\sigma_\nu))X_\beta$$

    \vspace{2 mm}

    \noindent where $q = N(\nu)$. Comparing terms on both sides we get $(q-1)Z = 0$. 
     Since $q$ was chosen so that $q \not \equiv 1 \pmod{\varpi^{eD+1}}$ we have that $Z \pmod{\varpi^{n-eD}} = 0$. Thus, $Z$ is a multiple of $\varpi^{n-eD}$. On the other hand, we get $(q\beta(\rho_n(\sigma_\nu))-1)X_\beta = 0$. By the modifications to part $(3)$ of Theorem \ref{3.4} we have that if $\nu$ was added in the $r$-th step of the doubling method, then $\beta(\rho_{r+1}(\sigma_\nu)) \not \equiv q^{-1} \pmod{\varpi^{r+1}}$ for all $\beta \in \Phi(G^0,T)$. Since after the $N_1$-th step we are not adding any new primes at which the lift is unramified we get that $q\beta(\rho_{N_1+1}(\sigma_\nu)) \not \equiv 1 \pmod{\varpi^{N_1+1}}$. This implies that $X_\beta \pmod{\varpi^{n-N_1}} = 0$, i.e. $X_\beta$ is a multiple of $\varpi^{n-N_1}$. Combining these two results, we get that $X$ is a multiple of $\varpi^{n-s}$.
    
\end{proof}

Using this result as above $H^0(\Gamma_{F_\nu},\rho(\mathfrak g^{\der})^*) = \varprojlim_n H^0(\Gamma_{F_\nu},\rho_n(\mathfrak g^{\der})^*) = 0$. Thus, again by Proposition \ref{3.1} $\restr{\rho}{\Gamma_{F_\nu}}$ is a formally smooth point in $R_{\restr{\overline{\rho}}{\Gamma_{F_\nu}}}^{\sqr,\mu}[1/\varpi]$. 

\begin{rem}

    \label{4.3}

    All these results rely on the changes we did during the doubling method and the relative deformation argument doesn't play any role in them. More precisely, as long as we run the relative deformation argument starting from the pair $(\rho_{N-M},\rho_N)$ or any other pair that appears higher in the tower of lifts, and we keep the same local deformation conditions at primes in $S_{N-M}$ the characteristic $0$ lift we eventually obtain will be locally smooth at all primes in $S_{N-M}$. This will be of particular importance in the next section as in it we might need to change the set $Q$ in order to get the formal smoothness at primes in $Q$ as well. However, we will not make any changes to the set $S_{N-M}$, so any lift that we later get will be locally smooth at primes in $S_{N-M}$ and we don't have to worry about formal smoothness at those primes anymore.
    
\end{rem}

\subsection{Forcing ramification at primes in \texorpdfstring{$Q$}{Q}}

In this section we focus on the set of primes $Q$. We show that even though the lift $\rho$ might not be smooth at each of the primes in $Q$, we can choose a different such set $Q$, and hence a different lift $\rho$ which will be formally smooth at all primes in $Q$. This section is the gist of the paper and represents the most difficult step in achieving the formal smoothness at all the primes in question. In it we first prove that the question of formal smoothness at primes in $Q$ reduces to a question about ramification at those primes. After that, mimicking the ideas from \cite{KR03} we will force the ramification at primes in $Q$. This will require us to modify the initial set $Q$ and we will need to choose primes in it more carefully in comparison to the original lifting method. One of the novelties will be considering $(r,s)$-relative Selmer group for varying values of $r$ and $s$, instead of just working with $M$-th relative Selmer group for a fixed $M$. 

We first recall that in the previous subsection we have enlarged the set $Q$, adding the primes in $S_N \setminus S_{N-M}$ of the third type. As mentioned at the time we will not differentiate between these primes and the ones in the original set $Q$. However, we will make sure to point out some of the subtle differences when it comes down to the application of Lemma \ref{3.16}, used in the production of the local lifting conditions at them. Since the integer $N$ will not play any further role we write $S'$ for the set of primes $S_N$, remembering that we moved some of its prime to $Q$. We initially prove that in order to show that $\rho$ is formally smooth at primes in $Q$ it is enough to show that $\rho$ is ramified at each of the primes in $Q$.

\begin{lem}

    \label{4.4}

    Suppose that $\rho$ ramifies at $\nu \in Q$, then $\restr{\rho}{\Gamma_{F_\nu}}$ is a formally smooth point in $R_{\restr{\overline{\rho}}{\Gamma_{F_\nu}}}^{\sqr,\mu}[1/\varpi]$.
    
\end{lem}
\begin{proof}

    As $\rho$ is ramified at $\nu$ and $\rho$ is the inverse limit of the lifts $\rho_n$ we can find a positive integer $s$ such that $\rho_s$ is ramified at $\nu$. The local conditions at $\nu$ are chosen using Lemma \ref{3.16}. Therefore, $\restr{\rho_s}{\Gamma_{F_\nu}} \in \Lift_{\restr{\rho_M}{\Gamma_{F_\nu}}}^{\mu,\alpha}(\mathcal O/\varpi^s)$, and in particular $\rho_s(\tau_\nu)$ is a non-trivial element of $U_\alpha(\mathcal O/\varpi^s)$ for some root $\alpha$ associated to some split maximal torus $T$.

    From this it easily follows that $H^0(\Gamma_{F_\nu},\rho_n(\mathfrak g^{\der})^*) \subseteq \varpi^{n-s-1}\rho_{s-1}(\mathfrak g^{\der})^*$ for every $n \ge s$. Indeed, as in Lemma \ref{4.1} based on the first and the last bullet point of Definition \ref{3.21} (or part $(b)$ of Theorem \ref{3.4}) from the action of $\sigma_\nu$ we deduce that for $X \in H^0(\Gamma_{F_\nu},\rho_n(\mathfrak g^{\der}))$, every term in the root decomposition is a multiple of $\varpi^{n-M}$ (or $\varpi^{n-eD}$ for primes that we moved from $S_N$ to $Q$), except the one corresponding to the root $-\alpha$. As any ramification at $\nu$ must occur during the relative deformation part, looking back at Claim \ref{3.31} we get that $s > N-M > M \ge eD$, hence all these terms are multiples of $\varpi^{n-s-1}$. On the other hand, according to the same computation the action of $\tau_\nu$ gives us that the term corresponding to the root $-\alpha$ is also a multiple of $\varpi^{n-s-1}$.

    Finally, we get $H^0(\Gamma_{F_\nu},\rho(\mathfrak g^{\der})^*) = \varprojlim_n H^0(\Gamma_{F_\nu},\rho_n(\mathfrak g^{\der})^*) = 0$, which by Proposition \ref{3.1} yields the formal smoothness at $\nu$.
    
\end{proof}

Although in the relative deformation part we must allow ramification at primes in $Q$ for the argument to work, there is no guarantee that $\rho$ will be ramified at any of the primes in $Q$. In what follows we will develop a method in which we will force the ramification at primes in $Q$. In it we might need to change the initial set $Q$ and subsequently the initial lift $\rho$. We give a quick overview of the main idea in the forcing ramification argument. We first fix a prime $\nu_0 \in Q$ at which $\rho$ is not ramified. After removing $\nu_0$ from $Q$ we carefully choose two other primes $\nu_1$ and $\nu_2$, and we add them to $Q$. The prime $\nu_1$ is chosen by Proposition \ref{4.10}, while the prime $\nu_2$ by Proposition \ref{4.16}. The purpose of $\nu_2$ is to annihilate a relative dual Selmer group and use this result to produce a new $p$-adic lift $\rho'$ which will be unramified away from the modified set of primes $S' \cup Q$. On the other hand, $\nu_1$ is chosen so that it will force the ramification of $\rho'$ at $\nu_1$ and $\nu_2$, as well as all the other primes where $\rho$ was already ramified. We then run this argument recursively exhausting the subset of $Q$ at which $\rho$ is unramified. This eventually leaves us with a characteristic $0$ lift which is ramified at all primes in $Q$.

We will still rely on the same lifting method and hence on the vanishing of relative dual Selmer group in order to produce the characteristic $0$ lift. However, if beforehand we knew that the $r$-th relative dual Selmer group vanished, after replacing a prime $\nu_0$ in $Q$ at which $\rho$ doesn't ramify with two other primes $\nu_1$ and $\nu_2$, as explained above, we might no longer be able to guarantee the vanishing of the $r$-th relative dual Selmer group. Instead, we will show that the $2r$-th relative dual Selmer group is $0$. Therefore, running this process recursively, starting with $r=M$ we end up with a vanishing $2^dM$-th relative dual Selmer group, where $d$ is the number of primes in $Q$ at which $\rho$ isn't ramified. At the end, we use this to produce the characteristic $0$ lift that will be ramified at all the primes in the new set $Q$. 

Now, suppose that the $r$-th relative dual Selmer group $\overline{H_{\mathcal L^\perp_{r,S' \cup Q}}^1(\Gamma_{F,S' \cup Q},\rho_r(\mathfrak g^{\der})^*)}$ vanishes for some $r \ge M$, divisible by $e$, and we used this to produce a characteristic $0$ lift $\rho$. Working with relative Selmer groups modulo $\varpi^{2r}$ will require us to first define the local conditions modulo higher powers of $\varpi$. The next lemma shows us that we can do this in a suitable manner.

\begin{lem}

    \label{4.5}

    For $r < m \le 2r$ there exist local conditions $\mathcal L_{m,S' \cup Q} = \{ L_{m,\nu}\}_{\nu \in S' \cup Q}$ defined by cocycles $Z_{m,\nu}$ satisfying Assumptions \textup{\ref{3.10}} such that
    \begin{enumerate}
        \item The image of $\mathcal L_{m,S' \cup Q}$ under the map induced by the reduction $\mathcal O/\varpi^{m} \to \mathcal O/\varpi^r$ is exactly $\mathcal L_{r,S' \cup Q}$.
        \item For each $\nu \in S' \cup Q$ and $n$ large enough there exists a class of lifts $D_{m,\nu}(\mathcal O/\varpi^n) \subseteq \Lift^\mu_{\restr{\overline{\rho}}{\Gamma_{F_\nu}}}(\mathcal O/\varpi^n)$ such that the fibers of $D_{m,\nu}(\mathcal O/\varpi^{n+m}) \to D_{m,\nu}(\mathcal O/\varpi^n)$ consisting of lifts over $\rho_{m}$ are non-empty principal homogeneous spaces over the submodules $Z_{m,\nu}$.
        \item $$|L_{m,\nu}| = \begin{cases} |H^0(\Gamma_{F_\nu},\rho_m(\mathfrak g^{\der}))| & \text{if } \nu \nmid p \\ |H^0(\Gamma_{F_\nu},\rho_m(\mathfrak g^{\der}))| \cdot |\mathcal O/\varpi^m|^{[F_\nu:\mathbb Q_p]\dim_k(\mathfrak n)} & \text{if } \nu \mid p\end{cases}$$

        \vspace{2 mm}

        \noindent where $\mathfrak n$ is the Lie algebra of a Borel subgroup of $G$.
    \end{enumerate}
    
\end{lem}
\begin{proof}

    The new local conditions can be produced similarly to the local conditions $\mathcal L_{r,S' \cup Q}$. At the primes in $S'$ of the first two types we obtain them by applying \cite[Proposition $4.7$]{FKP21} with $r_0 = m$ and the local lift $\restr{\rho}{\Gamma_{F_\nu}}$. For primes in $S'$ of the third type and primes in $Q$ we apply Lemma \ref{3.16} using the local lift $\restr{\rho_{M+s}}{\Gamma_{F_\nu}}$ with $M=2r$ and $s$ being the largest integer such that $\restr{\rho_s}{\Gamma_{F_\nu}}$ is trivial modulo the center. We note that the value of $s$ will depend on $\nu$, but it will always be smaller than or equal to $2r$. Finally, at primes in $S'$ of the fourth type we set the unramified local conditions which can be easily defined modulo any power of $\varpi$. 
    
    It is straightforward to verify that the unramified local conditions satisfy the properties state in the lemma, so we focus on the cocycles produces by \cite[Proposition $4.7$]{FKP21} and Lemma \ref{3.16}. From these results we immediately see that the cocycles used to define the local conditions will satisfy Assumptions \ref{3.10}. The same is true for the last two properties. The conditions $\mathcal L_{r,S' \cup Q}$ were produced by reductions of the characteristic $0$ lift appearing in the previous recursive step (if this is the first step we use the given local lifts and lifts coming from the doubling method), which we label by $\rho'$. The new conditions $\mathcal L_{m,S' \cup Q}$ are defined using the same methods, so to prove the first part it suffices to show that the local lifts used to produce $\mathcal L_{r,S' \cup Q}$ and the ones used to produce $\mathcal L_{m,S' \cup Q}$ are equal modulo a certain power of $\varpi$. Indeed, for primes where we used Lemma \ref{3.16} to produce the local conditions it is enough to show that $\restr{\rho_{r+s}}{\Gamma_{F_\nu}}$ and $\restr{\rho'_{r+s}}{\Gamma_{F_\nu}}$ are the same. At primes in $S'$ of the first two types we used \cite[Proposition $4.7$]{FKP21} to define the local conditions, which requires us to use the characteristic $0$ lift. However, as discussed in more details in Remark \ref{5.2} the modulo $\varpi^r$ local conditions will depend only on the reduction modulo $\varpi^r$ of this characteristic $0$ lift. Therefore, for these cocycles it suffices to show that $\restr{\rho_{r}}{\Gamma_{F_\nu}}$ and $\restr{\rho'_{r}}{\Gamma_{F_\nu}}$ are the same.
    
    As $\overline{H_{\mathcal L^\perp_{r,S' \cup Q}}^1(\Gamma_{F,S' \cup Q},\rho'_r(\mathfrak g^{\der})^*)} = 0$, the lifting method in the forcing ramification argument that follows will start building upward from at least $\rho'_{3r}$. This means that $\rho'_{3r}$ and $\rho_{3r}$ are equal. In particular, as $s \le 2r$ this means that $\restr{\rho_{r+s}}{\Gamma_{F_\nu}}$ and $\restr{\rho'_{r+s}}{\Gamma_{F_\nu}}$ are the same, as well as $\restr{\rho_{r}}{\Gamma_{F_\nu}}$ and $\restr{\rho'_{r}}{\Gamma_{F_\nu}}$. Therefore, as explained above the first property is satisfied, too.
    
\end{proof}

\begin{rem}

    \label{4.6}

    For some of the primes in $S' \cup Q$ \cite[Proposition $4.7$]{FKP21} and Lemma \ref{3.16} allow us to produce local conditions modulo higher powers of $\varpi$ using the previous lift $\rho'$. However, these will not be good enough for us and we want to point out the small subtlety of why that is the case. These old local conditions will parametrize the fibers which consist of lifts over $\rho'_m$. There is no guarantee that the local restrictions of $\rho$ will lie in those fibers. Indeed, $\rho$ is a lift of $\rho'_r$, but not necessarily of $\rho'_m$. Therefore, we might not be able to run the relative deformation argument starting from $\rho_n$ for $n\ge r$. This becomes more evident as we work modulo higher powers of $\varpi$, since the relative deformation argument doesn't allow us to prescribe the image of the lifts. 

    Alternatively, one can spot the issue in the fact that the old local conditions were living in $H^1(\Gamma_{F_\nu},\rho'_{m}(\mathfrak g^{\der}))$, while we need ones which are valued in $H^1(\Gamma_{F_\nu},\rho_m(\mathfrak g^{\der}))$. For bigger values of $m$, $\rho'_m$ and $\rho_m$ might differ, hence these cohomology groups may not be the same. 
    
\end{rem}

These extensions of the local lifting conditions will require us to update the bounds in (\ref{Nbound}) for any subsequent application of the relative deformation argument. Let $N_0'$ be the maximum of all the $n_0$ produced by applying \cite[Proposition $4.7$]{FKP21} with $r_0 = 2r$, as in Lemma \ref{4.5}. From now on we will always choose $n$ that satisfies:

\begin{equation}
    \label{Nbound2} n \ge \max\{N_0'+2r,(N_1+1)+2r,t+2r,12r\}  \tag{$2$}
\end{equation}

\vspace{2 mm}

\noindent as well as being divisible by $e$. These bounds will guarantee that $n$ is large enough so that the fibers $D_{2r,\nu}(\mathcal O/\varpi^{n+2r}) \to D_{2r,\nu}(\mathcal O/\varpi^n)$ over $\rho_{2r}$ are preserved under the action of the corresponding local conditions. 

We now partition the set $Q$ as $Q_{\mathrm{ram}} \sqcup Q_{\mathrm{unr}}$, where a prime $\nu$ belongs to $Q_{\mathrm{ram}}$ if and only if $\rho$ is ramified at it. We suppose that $Q_{\mathrm{unr}} \neq \emptyset$, as otherwise we are done. The next lemma allows us to assume that removing any prime $\nu$ from $Q_{\mathrm{unr}}$ would result into a non-zero $r$-the relative dual Selmer group.

\begin{lem}

    \label{4.7}

    Suppose that $\overline{H_{\mathcal L^\perp_{2r,S' \cup (Q \setminus \nu)}}^1(\Gamma_{F,S' \cup (Q \setminus \nu)},\rho_{2r}(\mathfrak g^{\der})^*)} = 0$ for some $\nu \in Q_{\mathrm{unr}}$. Then, there exists a lift $\rho':\Gamma_{F,S\cup (Q \setminus \nu)} \to G(\mathcal O)$ of $\overline{\rho}$ that is ramified at all primes in $Q_{\mathrm{ram}}$.
    
\end{lem}
\begin{proof}

As $\rho$ is an inverse limit of modulo $\varpi^s$ lifts, for each prime in $Q_{\mathrm{ram}}$ there exists an integer $s_\nu$ such that $\rho_{s_\nu}$ is ramified at that prime, too. So, in addition to the bounds coming from (\ref{Nbound2}) we let $n$ be larger than all the integers $s_\nu + 2r$. Then, running the argument in Claim \ref{3.31} starting with the pair $(\rho_{n-2r},\rho_n)$, viewing them as maps from  $\Gamma_{F,S' \cup (Q \setminus \nu)}$, we can produce a characteristic $0$ lift $\rho'$. By the choice of the integer $n$ this lift $\rho'$ will be formally smooth at all the primes at which $\rho$ was formally smooth. In particular, $\rho'$ will be ramified at all primes in $Q_{\mathrm{ram}}$.

\end{proof}

Using this lemma we can assume that $\overline{H_{\mathcal L^\perp_{2r,S' \cup (Q \setminus \nu)}}^1(\Gamma_{F,S' \cup (Q \setminus \nu)},\rho_{2r}(\mathfrak g^{\der})^*)} \neq 0$ for all $\nu \in Q_{\mathrm{unr}}$, as otherwise we can run the forcing ramification argument with $\rho'$ in place of $\rho$ and $S' \cup (Q \setminus \nu)$ instead of $S' \cup Q$. This will reduce the size of the set $Q_{\mathrm{unr}}$ which is our ultimate goal. This assumption will play a crucial role in the dimension computations that follow.

\begin{lem}

    \label{4.8}
    
    Let $\nu \in Q_{\mathrm{unr}}$. Set $Q_0 \coloneqq Q \setminus \{\nu\}$. Then
    
    $$\overline{H_{\mathcal L_{2r,S' \cup Q_0}}^1(\Gamma_{F,S' \cup Q_0},\rho_{2r}(\mathfrak g^{\der}))} \quad \quad \text{and} \quad \quad \overline{H_{\mathcal L^\perp_{2r,S' \cup Q_0}}^1(\Gamma_{F,S' \cup Q_0},\rho_{2r}(\mathfrak g^{\der})^*)}$$ 

    \vspace{2 mm}
    
    \noindent have dimension $1$.
    
\end{lem}
\begin{proof}

    We set $L_{2r,\nu}'= L_{2r,\nu} + L_{2r,\nu}^{\mathrm{unr}}$. We then have a commutative diagram of exact sequences:
    
    \begin{center}
        \begin{tikzcd}
            0 \arrow[r] & {H^1_{\mathcal L_{2r,S' \cup Q}}(\Gamma_{F,S' \cup Q},\rho_{2r}(\mathfrak g^{\der}))} \arrow[r] \arrow[d, "0 \,\,"'] & {H^1_{\mathcal L_{2r,S' \cup Q_0} \cup L_{2r,\nu}'}(\Gamma_{F,S' \cup Q},\rho_{2r}(\mathfrak g^{\der}))} \arrow[d] \arrow[r] & {L_{2r,\nu}'/L_{2r,\nu}} \arrow[d] \\
            0 \arrow[r] & {H^1_{\mathcal L_{1,S' \cup Q}}(\Gamma_{F,S' \cup Q},\overline{\rho}(\mathfrak g^{\der}))} \arrow[r]                       & {H^1_{\mathcal L_{1,S' \cup Q_0} \cup L_{1,\nu}'}(\Gamma_{F,S' \cup Q},\overline{\rho}(\mathfrak g^{\der}))} \arrow[r]           & {L_{1,\nu}'/L_{1,\nu}}          
        \end{tikzcd}
    \end{center}
    
    \vspace{2 mm}
    
    Here, the vertical maps are just reductions modulo $\varpi$, while the last horizontal maps are restriction to $\Gamma_{F_\nu}$. By the overarching assumption of the forcing ramification method we have $\overline{H_{\mathcal L^\perp_{r,S' \cup Q}}^1(\Gamma_{F,S' \cup Q},\rho_r(\mathfrak g^{\der})^*)} = 0$. Hence, $\overline{H_{\mathcal L^\perp_{2r,S' \cup Q}}^1(\Gamma_{F,S' \cup Q},\rho_{2r}(\mathfrak g^{\der})^*)} = 0$, as well. By Lemma \ref{3.24} the local conditions $\mathcal L_{2r,S' \cup Q}$ are balanced and therefore $\overline{H_{\mathcal L_{2r,S' \cup Q}}^1(\Gamma_{F,S' \cup Q},\rho_{2r}(\mathfrak g^{\der}))} = 0$, which gives us that the first vertical map is $0$.

    The local conditions $L_{2r,\nu}$ we defined in Lemma \ref{4.5} and were produced by applying Lemma \ref{3.16} using a reduction of $\rho$ modulo $\varpi^{4r}$(or possibly a lower power of $\varpi$). As $\rho$ isn't ramified at $\nu$ the same is true for $\rho_{4r}$. Therefore the conclusions from Lemma \ref{3.22} are valid for the local lifting conditions $L_{2r,\nu}$ as they only relied on using a lift that is not ramified at $\nu$ to produce them. Therefore, we conclude that $L_{2r,\nu}'/L_{2r,\nu}$ is generated by the class of the unramified cocycle $\phi_{H_\alpha}^{\mathrm{un},2r}$ and is isomorphic to $\mathcal O/\varpi^{2r}$. This means that for any two cocycles $\phi_1,\phi_2 \in H^1_{\mathcal L_{2r,S' \cup Q_0} \cup L_{2r,\nu}'}(\Gamma_{F,S' \cup Q},\rho_{2r}(\mathfrak g^{\der}))$ we can find $a \in \mathcal O/\varpi^{2r}$ such that $\phi_1 = a\phi_2 + \phi$ (or possibly $\phi_2 = a\phi_1 + \phi$) for some $\phi \in H^1_{\mathcal L_{2r,S' \cup Q}}(\Gamma_{F,S' \cup Q},\rho_{2r}(\mathfrak g^{\der}))$. The first vertical map being $0$ yields $\overline{\phi} = 0$, and so the reductions modulo $\varpi$ of any two elements in $H^1_{\mathcal L_{2r,S' \cup Q_0} \cup L_{2r,\nu}'}(\Gamma_{F,S' \cup Q},\rho_{2r}(\mathfrak g^{\der}))$ are scalar multiples of each other. Thus, $\overline{H^1_{\mathcal L_{2r,S' \cup Q_0} \cup L_{2r,\nu}'}(\Gamma_{F,S' \cup Q},\rho_{2r}(\mathfrak g^{\der}))}$ has dimension at most $1$. On the other hand, we have assumed that $\overline{H_{\mathcal L_{2r,S' \cup Q_0}}^1(\Gamma_{F,S' \cup Q_0},\rho_{2r}(\mathfrak g^{\der}))}$ is a non-zero subspace of it. Combining these two results we get that $\overline{H_{\mathcal L_{2r,S' \cup Q_0}}^1(\Gamma_{F,S' \cup Q_0},\rho_{2r}(\mathfrak g^{\der}))}$ has dimension $1$. Using the balancedness of the local conditions the same can be said about $\overline{H_{\mathcal L^\perp_{2r,S' \cup Q_0}}^1(\Gamma_{F,S' \cup Q_0},\rho_{2r}(\mathfrak g^{\der})^*)}$.
    
\end{proof}

\begin{rem}

    \label{4.9}

    A similar computation shows that $\overline{H_{\mathcal L_{r,S' \cup Q_0}}^1(\Gamma_{F,S' \cup Q_0},\rho_{r}(\mathfrak g^{\der}))}$ and $\overline{H_{\mathcal L^\perp_{r,S' \cup Q_0}}^1(\Gamma_{F,S' \cup Q_0},\rho_{r}(\mathfrak g^{\der})^*)}$ are also one dimensional and therefore coincide with the corresponding $2r$-relative (dual) Selmer groups.

\end{rem}

We now fix $\nu_0 \in Q_{\mathrm{unr}}$ and as above we label $Q \setminus \{\nu_0\}$ by $Q_0$. We can write $\overline{H_{\mathcal L_{2r,S' \cup Q_0}}^1(\Gamma_{F,S' \cup Q_0},\rho_{2r}(\mathfrak g^{\der}))} = \langle \overline{\phi} \rangle$ and $\overline{H_{\mathcal L^\perp_{2r,S' \cup Q_0}}^1(\Gamma_{F,S' \cup Q_0},\rho_{2r}(\mathfrak g^{\der})^*)} = \langle \overline{\psi} \rangle$ as a result of the lemma. As in Lemma \ref{4.7}, for every prime $\nu \in Q_{\mathrm{ram}}$ we can choose an integer $s_\nu$ such that $\rho_{s_\nu}$ is ramified at $\nu$, too. We now fix an integer $n$ that is larger than all the integers $s_\nu + 2r$, while also satisfying the inequality in (\ref{Nbound2}). The next proposition will allow us to choose the prime $\nu_1$.

\begin{prop}

    \label{4.10}

    There exists a prime $\nu \in Q_{n,4r}$, disjoint from $S' \cup Q$ such that

    \begin{itemize}
        \item $\restr{\rho_n}{\Gamma_{F_\nu}} \in \Lift_{\restr{\overline{\rho}}{\Gamma_{F_\nu}}}^{\alpha,\mu}(\mathcal O/\varpi^n)$, but $\restr{\rho_{n+1}}{\Gamma_{F_\nu}} \notin \Lift_{\restr{\overline{\rho}}{\Gamma_{F_\nu}}}^{\alpha,\mu}(\mathcal O/\varpi^{n+1})$,
        \item $\restr{f}{\Gamma_{F_\nu}} \in L_{2r,\nu}^\alpha$ for each $f \in H_{\mathcal L_{2r,S' \cup Q_0}}^1(\Gamma_{F,S' \cup Q_0},\rho_{2r}(\mathfrak g^{\der}))$; and
        \item $\restr{\overline{\psi}}{\Gamma_{F_\nu}} \notin L_{1,\nu}^{\alpha,\perp}$.
    \end{itemize}

    \noindent where $\alpha$ is the root associated to $\nu$ in the definition of the set $Q_{n,4r}$ and the local conditions at $\nu$ are defined by applying Lemma \textup{\ref{3.16}} with $M=n-4r$ and $s=4r$.
    
\end{prop}
\begin{proof}

    As explained in the discussion following Definition \ref{3.21} primes in $Q_{n,4r}$ come from a Chebotarev condition in $F_n^*/F$, which is trivial on $F_{4r}^*$. By the third bullet point of the definition each prime $\nu \in Q_{n,4r}$ will satisfy $\restr{\rho_n}{\Gamma_{F_\nu}} \in \Lift_{\restr{\overline{\rho}}{\Gamma_{F_\nu}}}^{\alpha,\mu}(\mathcal O/\varpi^n)$. On the other side, in order to make sure that $\restr{\rho_{n+1}}{\Gamma_{F_\nu}} \notin \Lift_{\restr{\overline{\rho}}{\Gamma_{F_\nu}}}^{\alpha,\mu}(\mathcal O/\varpi^{n+1})$ it is enough to look for primes $\nu$ in $Q_{n,4r}$ such that for their Frobenius $\sigma_\nu$ we have $\alpha(\rho_{n+1}(\sigma_\nu)) \not \equiv N(\nu) \pmod{\varpi^{n+1}}$. As $K(\rho_{n+1}(\mathfrak g^{\der}))$ and $K_\infty$ are linearly disjoint over $K$ by Lemma \ref{3.19}, we can fix $q = N(\nu)$. Then, using the fact that $\rho_{n+1}$ has maximal image it is enough to find an element $t_{n+1} \in T(\mathcal O/\varpi^{n+1})$, a lift of $\rho_n(\sigma_\nu)$ such that $\alpha(t_{n+1}) \not \equiv N(\nu) \pmod{\varpi^{n+1}}$. This can be easily done as in the proof of Theorem \ref{3.4}. Therefore, the first property can be prescribed by a Chebotarev condition in $F_{n+1}^*/F$. 

    For $f \in H_{\mathcal L_{2r,S' \cup Q_0}}^1(\Gamma_{F,S' \cup Q_0},\rho_{2r}(\mathfrak g^{\der}))$ we define $F_{2r}^*(f)$ to be the fixed field of $\ker(\restr{f}{\Gamma_{F_{2r}^*}})$. As $F_{2r}^*$ trivializes the $\Gamma_F$-action on $\rho_{2r}(\mathfrak g^{\der})$ these are all Galois extensions of $F$. We then define $F_{2r}^*(H^1)$ to be the composite of all $F_{2r}^*(f)$, as $f$ ranges through $H_{\mathcal L_{2r,S' \cup Q_0}}^1(\Gamma_{F,S' \cup Q_0},\rho_{2r}(\mathfrak g^{\der}))$. To produce primes for which the second condition is satisfied we can use the Chebotarev Density Theorem and ask for their Frobenius to be trivial in $F_{2r}^*(H^1)/F$. To make sure that the first two properties can be arranged simultaneously we need to verify that these Chebotarev conditions are compatible. For this we need to show that the first Chebotarev condition is trivial on $L \coloneqq F_{n+1}^* \cap F_{2r}^*(H^1)$. It will be enough to show that this intersection lies inside $F_{4r}^*$.

    We first note that $\Gal(F_{2r}^*(H^1)/F_{2r}^*)$ is an abelian $p^{2r/e}$-torsion group. This is true because each $\restr{f}{\Gamma_{F_{2r}^*}}$ is a homomorphism valued in $\mathfrak g^{\der} \otimes_{\mathcal O} \mathcal O/\varpi^{2r}$. Then $\Gal(L/F_{2r}^*)$ is a quotient of $\Gal(F_{2r}^*(H^1)/F_{2r}^*)$ and hence it is an abelian $p^{2r/e}$-torsion group itself. On the other side, using the linear disjointness coming from Lemma \ref{3.19} and the fact that each reduction of $\rho$ has maximal image we have 

    $$\Gal(F_{n+1}^*/F_{2r}^*) \simeq \ker(\widehat{G^{\der}}(\mathcal O/\varpi^{n+1}) \to \widehat{G^{\der}}(\mathcal O/\varpi^{2r})) \oplus \mathbb Z/p^{\lceil{\frac{n+1-2r}e}\rceil}$$

    \vspace{2 mm}
 
     \noindent where the $\mathbb Z/p$ factor comes from the $p$-th power roots of unity adjoined to $F_{n+1}$ and $F_{2r}$. By Claim \ref{G^0Adjoint} we can assume and we actually have assumed that $G^0$ is adjoint and $\mathfrak g^{\der} = \mathfrak g$ consists of a single $\pi_0(G)$-orbit of simple factors. Therefore, $\mathfrak g^{\der} \otimes_\mathcal O k$ is a simple $\mathbb F_p[G(k)]$-module for $p \gg_G 0$. Thus, by Lemma \ref{A.2} the largest abelian $p^{2r/e}$-torsion quotient of $\Gal(F_{n+1}^*/F_{2r}^*)$ is isomorphic to $\ker(\widehat{G^{\der}}(\mathcal O/\varpi^{4r}) \to \widehat{G^{\der}}(\mathcal O/\varpi^{2r})) \oplus \mathbb Z/p^{2r/e}$ and it corresponds to $\Gal(F_{4r}^*/F_{2r}^*)$. Thus, $L \subseteq F_{4r}^*$, which is exactly what we wanted. Summarizing, we get that the first two conditions can be prescribed by a Chebotarev condition in $F_{n+1}^*(H^1)/F$. 

    For the final condition we are going to first show that $\restr{\overline{\psi}}{\Gamma_{F_{n+1}^*(H^1)}} \neq 0$. We have the inflation-restriction sequence:

    $$H^1(\Gal(F_{n+1}^*/F),\overline{\rho}(\mathfrak g^{\der})^*) \to H^1(\Gal(F_{n+1}^*(H^1)/F),\overline{\rho}(\mathfrak g^{\der})^*) \to H^1(\Gal(F_{n+1}^*(H^1)/F_{n+1}^*),\overline{\rho}(\mathfrak g^{\der})^*)^{\Gal(F_{n+1}^*/F)}$$

    \vspace{2 mm}

    By Lemma \ref{3.20} the first term is equal to $0$. As $F_{n+1}^*$ trivializes the action on $\overline{\rho}(\mathfrak g^{\der})^*$ the last term is $\Hom_{\Gamma_F}(\Gal(F_{n+1}^*(H^1)/F_{n+1}^*),\overline{\rho}(\mathfrak g^{\der})^*)$. Any non-zero element of it would give us an isomorphism between a quotient of $\Gal(F_{n+1}^*(H^1)/F_{n+1}^*)$ and a submodule of $\overline{\rho}(\mathfrak g^{\der})^*$. As the cocycles $f$ are valued in $\rho_{2r}(\mathfrak g^{\der})$ we have that $\Gal(F_{n+1}^*(H^1)/F_{n+1}^*)$ is an $\mathbb F_p[\Gamma_F]$-submodule of a direct sum of copies of $\rho_{2r}(\mathfrak g^{\der})$. Therefore, if $\Hom_{\Gamma_F}(\Gal(F_{n+1}^*(H^1)/F_{n+1}^*),\overline{\rho}(\mathfrak g^{\der})^*) \neq 0$ we have an isomorphism between a simple subquotient of $\rho_{2r}(\mathfrak g^{\der})$ and a simple submodule of $\overline{\rho}(\mathfrak g^{\der})^*$. Using the exact sequence

    $$0 \longrightarrow \rho_a(\mathfrak g^{\der}) \longrightarrow \rho_{a+b}(\mathfrak g^{\der}) \longrightarrow \rho_b(\mathfrak g^{\der}) \longrightarrow 0$$

    \vspace{2 mm}
    
    \noindent and the Jordan-Holder's Theorem $\rho_{2r}(\mathfrak g^{\der})$ and $\overline{\rho}(\mathfrak g^{\der})$ share the same simple subquotients (with different multiplicities). All this yields an isomorphism between a subquotient of $\overline{\rho}(\mathfrak g^{\der})$ and a submodule of $\overline{\rho}(\mathfrak g^{\der})^*$, which contradicts Assumptions \ref{1.1}. Therefore, $\Hom_{\Gamma_F}(\Gal(F_{n+1}^*(H^1)/F_{n+1}^*),\overline{\rho}(\mathfrak g^{\der})^*) = 0$ and hence the middle term in the exact sequence is also equal to $0$. This term vanishing tells us that the inflation-restricton sequence yields an inclusion $H^1(\Gamma_{F},\overline{\rho}(\mathfrak g^{\der})^*) \hookrightarrow H^1(\Gamma_{F_{n+1}^*(H^1)},\overline{\rho}(\mathfrak g^{\der})^*)$. As $\overline{\psi}$ is non-zero this means that $\restr{\overline{\psi}}{\Gamma_{F_{n+1}^*(H^1)}} \neq 0$. 

    We claim that there is a split maximal torus $T$ and a root $\alpha \in \Phi(G^0,T)$ such that $\overline{\psi}(\Gamma_{F_{n+1}^*(H^1)})$ is not contained in $\mathfrak g_\alpha^\perp$. We prove a more general statement in Proposition \ref{4.16}, so we omit the proof of this claim for the time being. Hence we can find $\gamma_2 \in \Gal(F_{n+1}^*(H^1,\overline{\psi})/F_{n+1}^*(H^1))$ such that $\overline{\psi}(\gamma_2) \notin \mathfrak g_\alpha^\perp$. Now, let $\gamma_1 \in \Gamma_F$ be an element whose restriction to $F_{n+1}^*(H^1)$ will provide us with a prime that satisfies the first two properties, as explained above. We remark that we are feeding the specific choice of $(T,\alpha)$ in Definition \ref{3.21} to produce the set $Q_{n,4r}$.

    Now, by Lemma \ref{A.4} there exists $a \in \{0,1\}$ such that $\overline{\psi}(\gamma_2^a\gamma_1) = a \cdot \overline{\psi}(\gamma_2) + \overline{\psi}(\gamma_1) \notin \mathfrak g_\alpha^\perp$. We set $\sigma = \gamma_2^a\gamma_1$. We now claim any prime $\nu$ whose Frobenius $\sigma_\nu$ agrees with $\sigma$ inside $\Gal(F_{n+1}^*(H^1,\overline{\psi})/F)$ will satisfy the wanted properties. First of all, we note that there is a positive density of such primes by the Chebotarev Density Theorem, so we can always find such prime. As the restriction of $\sigma$ to $F_{n+1}^*(H^1)$ is $\gamma_1$, by what we have arranged before $\nu$ will be an element of $Q_{n,4r}$ that satisfies the first two properties. We also have that $\overline{\psi}(\sigma_\nu) \notin \mathfrak g_\alpha^\perp$. As $\overline{\psi}$ isn't ramified at $\nu$ by Remark \ref{3.23} we have that the third property is satisfied, as well.
    
\end{proof}

We now fix a prime $\nu_1$ that is produced by this proposition.

\begin{prop}
    \label{4.11}

    $$H^1_{\mathcal L_{2r,S' \cup Q_0} \cup [L_{2r,\nu_1}^\alpha + L_{2r,\nu_1}^{\mathrm{unr}}]}(\Gamma_{F,S' \cup Q_0 \cup \nu_1},\rho_{2r}(\mathfrak g^{\der})) = H_{\mathcal L_{2r,S' \cup Q_0}}^1(\Gamma_{F,S' \cup Q_0},\rho_{2r}(\mathfrak g^{\der}))$$
    $$= H^1_{\mathcal L_{2r,S' \cup Q_0} \cup [L_{2r,\nu_1}^\alpha \cap L_{2r,\nu_1}^{\mathrm{unr}}]}(\Gamma_{F,S' \cup Q_0},\rho_{2r}(\mathfrak g^{\der})) = H_{\mathcal L_{2r,S' \cup Q_0 \cup \nu_1}}^1(\Gamma_{F,S' \cup Q_0 \cup \nu_1},\rho_{2r}(\mathfrak g^{\der}))$$

    \vspace{2 mm}
    
\end{prop}
\begin{proof}

    We start with the commuting diagram

    \begin{center}
        \begin{tikzcd}
            0 \arrow[r] & {H^1_{\mathcal L_{2r,S' \cup Q}^\perp \cup L_{2r,\nu_1}^{'\perp}}(\Gamma_{F,S' \cup Q_0},\rho_{2r}(\mathfrak g^{\der})^*)} \arrow[r] \arrow[d] & {H^1_{\mathcal L_{2r,S' \cup Q_0}^\perp}(\Gamma_{F,S' \cup Q_0},\rho_{2r}(\mathfrak g^{\der})^*)} \arrow[d] \arrow[r] & {L_{2r,\nu_1}^{\mathrm{unr}}/L_{2r,\nu_1}^{'\perp}} \arrow[d] \\
            0 \arrow[r] & {H^1_{\mathcal L_{1,S' \cup Q}^\perp \cup L_{1,\nu_1}^{'\perp}}(\Gamma_{F,S' \cup Q_0},\overline{\rho}(\mathfrak g^{\der})^*)} \arrow[r]                       & {H^1_{\mathcal L_{1,S' \cup Q_0}^\perp}(\Gamma_{F,S' \cup Q_0},\overline{\rho}(\mathfrak g^{\der})^*)} \arrow[r]           & {L_{1,\nu_1}^{\mathrm{unr}}/L_{1,\nu_1}^{'\perp}}          
        \end{tikzcd}
    \end{center}
    
    \vspace{2 mm}

    \noindent where the vertical maps are just reductions modulo $\varpi$, the last horizontal maps are restriction to $\Gamma_{F_{\nu_1}}$ and we used the notation $L_{2r,\nu_1}' = L_{2r,\nu_1}^{\alpha} + L_{2r,\nu_1}^{\mathrm{unr}}$ introduced in the proof of Lemma \ref{4.8}. Based on our choice of prime $\nu_1$ we get that $\restr{\overline{\psi}}{\Gamma_{F_{\nu_1}}}$ will have a non-zero image in $L_{1,\nu_1}^{\mathrm{unr}}/L_{1,\nu_1}^{'\perp}$, which by Lemma \ref{3.22} is isomorphic to $\mathcal O/\varpi$. This implies that $\restr{\psi}{\Gamma_{F_{\nu_1}}}$ will generate ${L_{2r,\nu_1}^{\mathrm{unr}}/L_{2r,\nu_1}^{'\perp}}$ which is isomorphic to $\mathcal O/\varpi^{2r}$, as the only elements that have a non-zero image under the map $\mathcal O/\varpi^{2r} \to \mathcal O/\varpi$ are the units, i.e. the generators of $\mathcal O/\varpi^{2r}$. Therefore the last maps in both rows are surjective.

    Applying the Greenberg-Wiles formula once with local conditions $\mathcal L_{2r,S' \cup Q_0} \cup L_{2r,\nu_1}'$ and once with $\mathcal L_{2r,S' \cup Q_0}$ and dividing the two equations we get the following equation:

    $$\frac{\left|\,H^1_{\mathcal L_{2r,S' \cup Q_0} \cup L_{2r,\nu_1}'}(\Gamma_{F,S' \cup Q_0 \cup \nu_1},\rho_{2r}(\mathfrak g^{\der}))\,\right|}{\left|\,H_{\mathcal L_{2r,S' \cup Q_0}}^1(\Gamma_{F,S' \cup Q_0},\rho_{2r}(\mathfrak g^{\der}))\,\right|} = \frac{\left|\,H^1_{\mathcal L_{2r,S' \cup Q_0}^\perp \cup L_{2r,\nu_1}^{'\perp}}(\Gamma_{F,S' \cup Q_0},\rho_{2r}(\mathfrak g^{\der})^*)\,\right|}{\left|\,H_{\mathcal L_{2r,S' \cup Q_0}^\perp}^1(\Gamma_{F,S' \cup Q_0},\rho_{2r}(\mathfrak g^{\der})^*)\,\right|} \cdot \frac{\left|\,L_{2r,\nu_1}'\,\right|}{\left|\,L_{2r,\nu_1}^{\mathrm{unr}}\,\right|}$$

    \vspace{2 mm}

    Based on the computations in Lemma \ref{3.22} the last fraction is equal to $|\mathcal O/\varpi^{2r}|$. On the other side, from the discussion above the first factor on the right is equal to $|\mathcal O/\varpi^{2r}|^{-1}$. Hence, the left side is equal to $1$. This coupled with the obvious inclusion $H_{\mathcal L_{2r,S' \cup Q_0}}^1(\Gamma_{F,S' \cup Q_0},\rho_{2r}(\mathfrak g^{\der})) \subseteq H^1_{\mathcal L_{2r,S' \cup Q_0} \cup L_{2r,\nu_1}'}(\Gamma_{F,S' \cup Q_0 \cup \nu_1},\rho_{2r}(\mathfrak g^{\der}))$ gives us the first equality of the proposition. The second equality follows immediately, since by the choice of the prime $\nu_1$ the restriction of every element of $H_{\mathcal L_{2r,S' \cup Q_0}}^1(\Gamma_{F,S' \cup Q_0},\rho_{2r}(\mathfrak g^{\der}))$ to $\Gamma_{F_{\nu_1}}$ will also lie in $L_{2r,\nu_1}^\alpha$. The final equality follows from the obvious inclusions and the equality between the first and the third Selmer groups in the statement of the proposition.
    
\end{proof}

\begin{rem}
    \label{4.12}

    Therefore, $\overline{H_{\mathcal L_{2r,S' \cup Q_0 \cup \nu_1}}^1(\Gamma_{F,S' \cup Q_0 \cup \nu_1},\rho_{2r}(\mathfrak g^{\der}))} = \overline{H_{\mathcal L_{2r,S' \cup Q_0}}^1(\Gamma_{F,S' \cup Q_0},\rho_{2r}(\mathfrak g^{\der}))} = \langle \overline{\phi} \rangle$ and by the balancedness of the local conditions $\overline{H_{\mathcal L^\perp_{2r,S' \cup Q_0 \cup \nu_1}}^1(\Gamma_{F,S' \cup Q_0 \cup \nu_1},\rho_{2r}(\mathfrak g^{\der})^*)} = \langle \overline{\psi'} \rangle$ for some cocycle $\psi'$.

\end{rem}

\begin{rem}

    \label{4.13}

    The first equality in the proposition relied only on the fact that $\restr{\overline{\psi}}{\Gamma_{F_{\nu_1}}} \notin L_{1,\nu}^{\alpha,\perp}$. As $\overline{\psi}$ generates the $r$-th relative dual Selmer group, as well, we get that:

    $$H_{\mathcal L_{r,S' \cup Q_0 \cup \nu_1}}^1(\Gamma_{F,S' \cup Q_0 \cup \nu_1},\rho_{r}(\mathfrak g^{\der})) \subseteq H^1_{\mathcal L_{r,S' \cup Q_0} \cup L_{r,\nu_1}'}(\Gamma_{F,S' \cup Q_0 \cup \nu_1},\rho_{r}(\mathfrak g^{\der})) = H_{\mathcal L_{r,S' \cup Q_0}}^1(\Gamma_{F,S' \cup Q_0},\rho_{r}(\mathfrak g^{\der}))$$

    \vspace{2 mm}

    Therefore:
    
    $$\overline{H_{\mathcal L_{2r,S' \cup Q_0 \cup \nu_1}}^1(\Gamma_{F,S' \cup Q_0 \cup \nu_1},\rho_{2r}(\mathfrak g^{\der}))} \subseteq \overline{H_{\mathcal L_{r,S' \cup Q_0 \cup \nu_1}}^1(\Gamma_{F,S' \cup Q_0 \cup \nu_1},\rho_{r}(\mathfrak g^{\der}))} \subseteq \overline{H_{\mathcal L_{r,S' \cup Q_0}}^1(\Gamma_{F,S' \cup Q_0},\rho_{r}(\mathfrak g^{\der}))}$$

    \vspace{2 mm}

    The first term is equal to $\langle \overline{\phi} \rangle$ by Remark \ref{4.12} and the same is true for the last one by Remark \ref{4.9}. Hence, the middle term is equal to $\langle \overline{\phi} \rangle$, too. As the local conditions are balanced $\overline{H_{\mathcal L^\perp_{r,S' \cup Q_0 \cup \nu_1}}^1(\Gamma_{F,S' \cup Q_0 \cup \nu_1},\rho_{r}(\mathfrak g^{\der})^*)}$ has dimension $1$, as well. It also contains the $2r$-th relative dual Selmer group $\langle \overline{\psi'} \rangle$, so it must be equal to it. 

\end{rem}

The next lemma constitutes one of the main technical innovations of this paper. Producing characteristic $0$ lifts relies heavily on our ability to annihilate certain relative dual Selmer groups. The results in \cite{FKP21}, as seen in Theorem \ref{3.27} tell us that we can always annihilate such groups by allowing extra ramification at a finite set of primes. However, for our purposes, in the dimension computations that follow it is crucial that allowing ramification at each prime reduces the dimension of the relative dual Selmer group. Unfortunately, as mentioned in \cite[Remark $6.10$]{FKP21} in order to annihilate the relative dual Selmer group we might need more primes than its dimension. The lemma will tell us that by allowing ramification at a carefully chosen prime although we might not be able to lower the dimension of that particular relative dual Selmer group, we can achieve that for relative dual Selmer group modulo some higher power of $\varpi$.

\begin{lem}

    \label{4.14}

    Let $1 \le s \le m$ be integers. Suppose that $\overline{H_{\mathcal L_{m,S}}^1(\Gamma_{F,S},\rho_m(\mathfrak g^{\der}))} = \overline{H_{\mathcal L_{s,S}}^1(\Gamma_{F,S},\rho_s(\mathfrak g^{\der}))}$ with $k$-dimension $d$. Then for any $t \le m-s+1$, $H^1_{(m,t)}(\Gamma_{F,S},\rho_m(\mathfrak g^{\der}))$ is a free $\mathcal O/\varpi^t$-module of rank $d$.
    
\end{lem}
\begin{proof}

    As the reduction maps factor through each other for $s \le l \le m$ we have that

    $$\overline{H_{\mathcal L_{m,S}}^1(\Gamma_{F,S},\rho_m(\mathfrak g^{\der}))} \,\subseteq\, \overline{H_{\mathcal L_{l,S}}^1(\Gamma_{F,S},\rho_m(\mathfrak g^{\der}))} \,\subseteq\, \overline{H_{\mathcal L_{s,S}}^1(\Gamma_{F,S},\rho_s(\mathfrak g^{\der}))}$$

    \vspace{2 mm}

    The hypothesis of the lemma ensures that we have an equality. From this we get that if $\overline{\phi_{m,1}},\overline{\phi_{m,2}},\dots,\overline{\phi_{m,d}}$ form a $k$-basis of $\overline{H_{L_{m,S}}^1(\Gamma_{F,S},\rho_m(\mathfrak g^{\der}))}$, then $\overline{\phi_{l,1}},\overline{\phi_{l,2}},\dots,\overline{\phi_{l,d}}$ form a $k$-basis of $\overline{H_{L_{l,S}}^1(\Gamma_{F,S},\rho_l(\mathfrak g^{\der}))}$ for any $s \le l \le m$, where $\phi_{l,i} \coloneqq \phi_{m,i} \pmod{\varpi^{l}}$.

    Now, let $f_m \in H_{\mathcal L_{m,S}}^1(\Gamma_{F,S},\rho_m(\mathfrak g^{\der}))$. Applying Lemma \ref{3.14} with $a=m-1$ and $b=1$ we get an exact sequence:

    $$H^1_{\mathcal L_{m-1,S}}(\Gamma_{F,S},\rho_{m-1}(\mathfrak g^{\der})) \stackrel{\cdot \varpi}{\longrightarrow} H^1_{\mathcal L_{m,S}}(\Gamma_{F,S},\rho_{m}(\mathfrak g^{\der})) \longrightarrow H^1_{\mathcal L_{1,S}}(\Gamma_{F,S},\overline{\rho}(\mathfrak g^{\der}))$$

    \vspace{2 mm}

    The image of the second map is the $m$-th relative Selmer group, which is spanned by $\overline{\phi_{m,1}},\overline{\phi_{m,2}},\dots,\overline{\phi_{m,d}}$. Thus, we can find $a_{m,i} \in \mathcal O$ such that $\overline{f_m - \sum_{i=1}^d a_{m,i} \cdot \phi_{m,i}} = 0$. Hence, from the exact sequence we can find $f_{m-1} \in H^1_{\mathcal L_{m-1,S}}(\Gamma_{F,S},\rho_{m-1}(\mathfrak g^{\der}))$ such that $f_m = \varpi f_{m-1} + \sum_{i=1}^d a_{m,i}\phi_{m,i}$. Repeating this procedure for $f_{m-1}$ and so on we get:

    \begin{align*}
        f_m &= \varpi f_{m-1} + \sum_{i=1}^d a_{m,i}\phi_{m,i} \\
        &= \varpi\left(\varpi f_{m-2} +\sum_{i=1}^d a_{m-1,i}\phi_{m-1,i}\right) + \sum_{i=1}^d a_{m,i}\phi_{m,i} \\
        &= \varpi^2 f_{m-2} + \sum_{i=1}^d \varpi a_{m-1,i}\phi_{m-1,i} + \sum_{i=1}^d a_{m,i}\phi_{m,i} \\
        &= \dots = \varpi^{m-s+1}f_{s-1} + \sum_{j=0}^{m-s}\sum_{i=0}^d \varpi^ja_{m-j,i}\phi_{m-j,i}
    \end{align*}

    \vspace{2 mm}

    Therefore, for $t \le m-s+1$ the first term will vanish modulo $\varpi^t$. Then, based on how $\phi_{l,i}$ were defined we get that $f_m \pmod{\varpi^t}$ can be written as an $\mathcal O/\varpi^t$-linear combination of $\phi_{t,1},\dots,\phi_{t,d}$. Thus, $H_{(m,t)}(\Gamma_{F,S},\rho_m(\mathfrak g^{\der}))$ will be generated by $d$ elements. It remains to show that the mentioned generators form a free basis. We write $b_1\phi_{t,1} + b_2\phi_{t,2} + \cdots + b_d\phi_{t,d} = 0$ for some $b_i \in \mathcal O/\varpi^t$ with not all $b_i$ being zero. Let $t_1 < t$ be the highest power of $\varpi$ that divides all $b_i$. Writing $b_i = \varpi^{t_1}b_i'$ we get that $b_1'\phi_{t,1} + b_2'\phi_{t,2} + \cdots + b_d'\phi_{t,d} = 0 \pmod{\varpi^{t-t_1}}$ with at least one $b_i'$ not being a multiple of $\varpi$. Reducing modulo $\varpi$ we obtain a non-trivial relation between $\overline{\phi_{t,i}}$, which contradicts the fact that those form a $k$-basis of $\overline{H_{\mathcal L_{m,S}}^1(\Gamma_{F,S},\rho_m(\mathfrak g^{\der}))}$.

\end{proof}

\begin{rem}

    \label{4.15}

    The analogous result when working with dual Selmer groups holds as well.

\end{rem}

Then we have the following result:

\begin{prop}

    \label{4.16}

    Recall $\overline{H^1_{\mathcal L_{2r,S' \cup Q}^\perp}(\Gamma_{F,S' \cup Q},\rho_{2r}(\mathfrak g^{\der})^*)} = \langle \overline{\psi} \rangle$, $\overline{H^1_{\mathcal L_{2r,S' \cup Q \cup \nu_1}^\perp}(\Gamma_{F,S' \cup Q \cup \nu_1},\rho_{2r}(\mathfrak g^{\der})^*)} = \langle \overline{\psi'} \rangle$ and $\overline{H^1_{\mathcal L_{2r,S' \cup Q}}(\Gamma_{F,S' \cup Q},\rho_{2r}(\mathfrak g^{\der}))} = \overline{H^1_{\mathcal L_{2r,S' \cup Q \cup \nu_1}}(\Gamma_{F,S' \cup Q \cup \nu_1},\rho_{2r}(\mathfrak g^{\der}))} = \langle \overline{\phi} \rangle$. There exists a prime \linebreak $\nu \in Q_{n,2r}$ such that

    \begin{itemize}
        \item $\restr{\overline{\psi}}{\Gamma_{F_\nu}} \notin L_{1,\nu}^{\alpha,\perp}$,
        \item $\restr{\overline{\psi'}}{\Gamma_{F_\nu}} \notin L_{1,\nu}^{\alpha,\perp}$; and
        \item $\restr{\phi_r}{\Gamma_{F_\nu}} \notin L_{r,\nu}^{\alpha}$.
    \end{itemize}

    \noindent where we write $\phi_r$ for the reduction of the cocycle $\phi$ modulo $\varpi^r$, $\alpha$ is the root associated to $\nu$ in the definition of the set $Q_{n,r}$ and the local conditions at $\nu$ are defined by applying Lemma \textup{\ref{3.16}} with $M = n-2r$ and $s=2r$.
    
\end{prop}
\begin{proof}

    We first show that the restrictions of $\overline{\psi}$, $\overline{\psi'}$ and $\phi_r$ to $\Gamma_{F_n^*}$ are non-zero. For the first two, by the first part of Lemma \ref{3.20} we have that $H^1(\Gal(F_n^*/F),\overline{\rho}(\mathfrak g^{\der})^*) = 0$. Therefore, the inflation-restriction sequence gives us an inclusion $H^1(\Gamma_F,\overline{\rho}(\mathfrak g^{\der})^*) \hookrightarrow H^1(\Gamma_{F_n^*},\overline{\rho}(\mathfrak g^{\der})^*)$. Thus, as $\overline{\psi}$ and $\overline{\psi'}$ are non-zero, so are their restrictions to $\Gamma_{F_n^*}$. For $\phi_r$ we consider the diagram:

    \begin{center}
        \begin{tikzcd}
            0 \arrow[rr] &  & {H^1(\Gal(F_n^*/F),\rho_r(\mathfrak g^{\der}))} \arrow[rr] \arrow[d, "0"] &  & {H^1(\Gamma_F,\rho_r(\mathfrak g^{\der}))} \arrow[rr] \arrow[d] &  & {H^1(\Gamma_{F_n^*},\rho_r(\mathfrak g^{\der}))} \arrow[d] \\
            0 \arrow[rr] &  & {H^1(\Gal(F_n^*/F),\overline{\rho}(\mathfrak g^{\der}))} \arrow[rr]       &  & {H^1(\Gamma_F,\overline{\rho}(\mathfrak g^{\der}))} \arrow[rr]           &  & {H^1(\Gamma_{F_n^*},\overline{\rho}(\mathfrak g^{\der}))}          
            \end{tikzcd}
    \end{center}    

    \vspace{2 mm}

    \noindent where the exact rows come from the inflation-restriction sequence, the vertical maps are reductions modulo $\varpi$ and the first vertical map is $0$ by Lemma \ref{3.20} as $r \ge M$. If $\restr{\phi_r}{\Gamma_{F_n^*}}$ were zero, then $\phi_r$ would come from an element in $H^1(\Gal(F_n^*/F),\rho_r(\mathfrak g^{\der}))$. Hence its restriction modulo $\varpi$ would be $0$, which is a contradiction. Thus, $\restr{\phi_r}{\Gamma_{F_n^*}} \neq 0$. By Claim \ref{G^0Adjoint} we can assume that $\mathfrak g^{\der}$ is a single $\pi_0(G)$-orbit of simple factors. Hence, we can assume that the image of $\Gamma_{F_n^*}$ under each of these cocycles has non-trivial projection on each of the simple factors of $\mathfrak g^{\der}$. Indeed, we know that $\Gamma_F$ surjects onto $\pi_0(G)$ via $\overline{\rho}$, so the fact that $\mathfrak g^{\der}$ is a single $\pi_0(G)$-orbit of simple factors and the all the images are $\Gamma_F$-equivariant allows us to make the wanted conclusion. We can now focus on those projections for one simple factor, which allows us to assume that $G$ is connected and simple in the rest of the argument.

    Now, let $F_n^*(\phi_r),F_n^*(\overline{\psi})$ and $F_n^*(\overline{\psi'})$ be the fixed field of these restricted cocycles. By above, each of these is a non-trivial extension of $F_n^*$. We claim that we can find $\gamma_2 \in \Gal(F_n^*(\phi_r,\overline{\psi},\overline{\psi'})/F_n^*)$ such that $\phi_r(\gamma_2),\overline{\psi}(\gamma_2),\overline{\psi'}(\gamma_2) \neq 0$. This corresponds to finding an element in $\Gamma_{F_n^*}$, which acts non-trivially on each of the extensions. For $2$ non-trivial extension this can be easily arranged by looking on a case-by-case basis whether one field is contained into another or not. Thus, we can find an element in $\Gamma_{F_n^*}$ which acts non-trivially on $F_n^*(\overline{\psi})$ and $F_n^*(\overline{\psi'})$. To make sure that we can choose it to act non-trivally on $F_n^*(\phi_r)$, as well it is enough to show that $F_n^*(\phi_r)$ is not contained in $F_n^*(\overline{\psi},\overline{\psi'})$. If this is the case we get that $\overline{\rho}(\mathfrak g^{\der})^*$ and $\rho_r(\mathfrak g^{\der})$ share a non-trivial $\mathbb F_p[\Gamma_F]$-subquotient, which moreover gives us a common non-trivial $\mathbb F_p[\Gamma_F]$-subquotient of $\overline{\rho}(\mathfrak g^{\der})$ and $\overline{\rho}(\mathfrak g^{\der})^*$, which contradicts Assumptions \ref{1.1}. Hence, such $\gamma_2$ exists.  

    $\phi_r(\gamma_2)$ might be trivial modulo $\varpi$, however, as it is non-zero we can write it as $\varpi^s X$, for some $s < r$ and $X$ such that $X \pmod{\varpi} \neq 0$. We now fix a maximal torus $T_1$ and a root $\alpha_1 \in \Phi(G^0,T_1)$. Applying Lemma \ref{A.3} with $\overline{X}$, $\overline{\psi}(\gamma_2)$ and $\overline{\psi'}(\gamma_2)$ we can find $\overline{g} \in G(k)$ such that $\Ad(\overline{g}^{-1}) \overline{X} \notin \ker(\restr{\alpha_1}{\mathfrak t_1}) \oplus \bigoplus g_\beta$ and $\Ad(\overline{g}^{-1}) \overline{\psi}(\gamma_2), \Ad(\overline{g}^{-1}) \overline{\psi'}(\gamma_2) \notin \mathfrak g_{\alpha_1}^\perp$. We can lift $\overline{g}$ to $g \in G(\mathcal O/\varpi^r)$ using the smoothness of $G$. Then, writing $(T_g,\alpha_g)$ for $\Ad(g)(T_1,\alpha_1)$ we get that $\phi_r(\gamma_2) \notin \ker(\restr{\alpha_g}{\mathfrak t_g}) \oplus \bigoplus g_\beta$ and $\overline{\psi}(\gamma_2),\overline{\psi'}(\gamma_2) \notin \mathfrak g_{\alpha_g}^\perp$. 
    
    As noted many times before the primes in $Q_{n,r}$ are given by a Chebotarev condition in $F_n^*/F$, where we use $(T_g,\alpha_g)$ in Definition \ref{3.21}. Let $\gamma_1 \in \Gamma_F$, whose restriction to $F_n^*$ will produce primes in $Q_{n,r}$ with respect to $(T_g,\alpha_g)$ using the Chebotarev Density Theorem. By Lemma \ref{A.4} there exists $a \in \{0,1,2,3\}$ such that 

    $$\phi_r(\gamma_2^a \gamma_1) = a\phi_r(\gamma_2) + \phi_r(\gamma_1) \notin \ker(\restr{\alpha_g}{\mathfrak t_g}) \oplus \bigoplus g_\beta,$$

    \vspace{2 mm}
    
    $$\overline{\psi}(\gamma_2^a\gamma_1) = a\overline{\psi}(\gamma_2) + \overline{\psi}(\gamma_1) \notin \mathfrak g_{\alpha_g}^\perp$$

    \vspace{2 mm}
    $$\overline{\psi'}(\gamma_2^a\gamma_1) = a\overline{\psi'}(\gamma_2) + \overline{\psi'}(\gamma_1) \notin \mathfrak g_{\alpha_g}^\perp$$

    \vspace{2 mm}

    \noindent As noted in Remark \ref{3.23} any prime whose Frobenius is equal to $\gamma_2^a\gamma_1$ in $\Gal(F_n^*(\phi_r,\overline{\psi},\overline{\psi'})/F)$ will satisfy the desired conditions.
    
\end{proof}

We now fix a prime $\nu_2$ that will satisfy these conditions. Based on how we chose this prime we have the following result:

\begin{prop}
    \label{4.17}

    $$H_{\mathcal L_{2r,S' \cup Q_0} \cup [L_{2r,\nu_2}^\alpha + L_{2r,\nu_2}^{\mathrm{unr}}]}^1(\Gamma_{F,S' \cup Q_0 \cup \nu_2},\rho_{2r}(\mathfrak g^{\der})) = H^1_{\mathcal L_{2r,S' \cup Q_0} \cup [L_{2r,\nu_1}^\alpha \cap L_{2r,\nu_1}^{\mathrm{unr}}]}(\Gamma_{F,S' \cup Q_0},\rho_{2r}(\mathfrak g^{\der}))$$

    \vspace{2 mm}
    
\end{prop}
\begin{proof}

    We recall that $\overline{H_{\mathcal L_{2r,S' \cup Q_0}^\perp}^1(\Gamma_{F,S' \cup Q_0},\rho_{2r}(\mathfrak g^{\der}))} = \langle \overline{\psi} \rangle$. The prime $\nu_2$ was chosen so that $\restr{\overline{\psi}}{\Gamma_{F_{\nu_2}}} \notin L_{1,\nu_2}^{\alpha,\perp}$. Based on this, a similar computation as in Proposition \ref{4.11} gives us the equality:

    $$H^1_{\mathcal L_{2r,S' \cup Q_0} \cup [L_{2r,\nu_2}^\alpha + L_{2r,\nu_2}^{\mathrm{unr}}]}(\Gamma_{F,S' \cup Q_0 \cup \nu_2},\rho_{2r}(\mathfrak g^{\der})) = H_{\mathcal L_{2r,S' \cup Q_0}}^1(\Gamma_{F,S' \cup Q_0},\rho_{2r}(\mathfrak g^{\der}))$$

    \vspace{2 mm}

    By Proposition \ref{4.11} the Selmer group on the right is equal to $H^1_{\mathcal L_{2r,S' \cup Q_0} \cup [L_{2r,\nu_1}^\alpha \cap L_{2r,\nu_1}^{\mathrm{unr}}]}(\Gamma_{F,S' \cup Q_0},\rho_{2r}(\mathfrak g^{\der}))$, which gives us what we want. 
    
\end{proof}

\begin{prop}

    \label{4.18}

    $$\overline{H_{\mathcal L^\perp_{2r,S' \cup Q_0 \cup \{\nu_1,\nu_2\}}}^1(\Gamma_{F,S' \cup Q_0 \cup \{\nu_1,\nu_2\}},\rho_{2r}(\mathfrak g^{\der})^*)} = 0$$

    \vspace{2 mm}
    
\end{prop}
\begin{proof}

    By Remark \ref{4.13} we have that $\overline{H_{\mathcal L_{2r,S' \cup Q_0 \cup \nu_1}}^1(\Gamma_{F,S' \cup Q_0 \cup \nu_1},\rho_{2r}(\mathfrak g^{\der}))}$ and $\overline{H_{\mathcal L_{r,S' \cup Q_0 \cup \nu_1}}^1(\Gamma_{F,S' \cup Q_0 \cup \nu_1},\rho_{r}(\mathfrak g^{\der}))}$ are both equal to $\langle \overline{\phi} \rangle$. Then, by Lemma \ref{4.14} we have that $H^1_{(2r,r)}(\Gamma_{F,S' \cup Q_0 \cup \nu_1},\rho_{2r}(\mathfrak g^{\der}))$ is a free $\mathcal O/\varpi^r$-module of rank $1$. Moreover, this image is generated by $\phi_r$, the reduction of $\phi$ modulo $\varpi^r$. By Remark $\ref{4.12}$ we have $\overline{H_{\mathcal L^\perp_{2r,S' \cup Q_0 \cup \nu_1}}^1(\Gamma_{F,S' \cup Q_0 \cup \nu_1},\rho_{2r}(\mathfrak g^{\der})^*)} = \langle \overline{\psi'} \rangle$ and since $\restr{\overline{\psi'}}{\Gamma_{F_{\nu_2}}} \notin L_{1,\nu_2}^{\alpha,\perp}$ by the analogous computations to the ones in Proposition \ref{4.11} using the Greenberg-Wiles formula we get the inclusion:

    $$H_{\mathcal L_{2r,S' \cup Q_0 \cup \{\nu_1,\nu_2\}}}^1(\Gamma_{F,S' \cup Q_0 \cup \{\nu_1,\nu_2\}},\rho_{2r}(\mathfrak g^{\der})) \subseteq H_{\mathcal L_{2r,S' \cup Q_0 \cup \nu_1}}^1(\Gamma_{F,S' \cup Q_0 \cup \nu_1},\rho_{2r}(\mathfrak g^{\der}))$$

    \vspace{2 mm}

    This gives us that $\overline{H_{\mathcal L_{2r,S' \cup Q_0 \cup \{\nu_1,\nu_2\}}}^1(\Gamma_{F,S' \cup Q_0 \cup \{\nu_1,\nu_2\}},\rho_{2r}(\mathfrak g^{\der}))}$ has dimension $\le 1$. It remains to show that the dimension can't be equal to $1$. Suppose that it equals $1$. Since, $\overline{H_{\mathcal L^\perp_{r,S' \cup Q_0 \cup \nu_1}}^1(\Gamma_{F,S' \cup Q_0 \cup \nu_1},\rho_{r}(\mathfrak g^{\der})^*)} = \langle \overline{\psi'} \rangle$, as noted in Remark \ref{4.13}, by the same computations we get the inclusion:

    $$H_{\mathcal L_{r,S' \cup Q_0 \cup \{\nu_1,\nu_2\}}}^1(\Gamma_{F,S' \cup Q_0 \cup \{\nu_1,\nu_2\}},\rho_{r}(\mathfrak g^{\der})) \subseteq H_{\mathcal L_{r,S' \cup Q_0 \cup \nu_1}}^1(\Gamma_{F,S' \cup Q_0 \cup \nu_1},\rho_{r}(\mathfrak g^{\der}))$$

    \vspace{2 mm}

    By the choice of $\nu_1$ we have that $\overline{H_{\mathcal L_{r,S' \cup Q_0 \cup \nu_1}}^1(\Gamma_{F,S' \cup Q_0 \cup \nu_1},\rho_{r}(\mathfrak g^{\der}))} = \langle \overline{\phi} \rangle$ and it has dimension $1$. On the other side, $\overline{H_{\mathcal L_{r,S' \cup Q_0 \cup \{\nu_1,\nu_2\}}}^1(\Gamma_{F,S' \cup Q_0 \cup \{\nu_1,\nu_2\}},\rho_{r}(\mathfrak g^{\der}))}$ contains the $2r$-th relative Selmer group, which we assumed is also one dimensional. Therefore, this $r$-th relative Selmer group has dimension $1$ as well. Thus, by Lemma \ref{4.14} we have that $H^1_{(2r,r)}(\Gamma_{F,S' \cup Q_0 \cup \{\nu_1,\nu_2\}},\rho_{2r}(\mathfrak g^{\der}))$ is a free $\mathcal O/\varpi^r$-module of rank $1$. The inclusions above tell us that this is a subgroup of $H^1_{(2r,r)}(\Gamma_{F,S' \cup Q_0 \cup \nu_1},\rho_{2r}(\mathfrak g^{\der}))$. But, both of these are free $\mathcal O/\varpi^r$-modules of rank $1$, so they must be equal. This will imply that $\phi_r \in H^1_{(2r,r)}(\Gamma_{F,S' \cup Q_0 \cup \{\nu_1,\nu_2\}},\rho_{2r}(\mathfrak g^{\der}))$, which contradicts the fact that $\restr{\phi_r}{\Gamma_{F_{\nu_2}}} \notin L_{r,\nu_2}^\alpha$. 
    
\end{proof}

\begin{rem}
    \label{4.19}

    We can easily generalize this argument and show that if an $r$-th relative Selmer group has dimension $m \ge 1$, then we can make sure that the $2r$-th relative Selmer group has dimension $\le m-1$ by allowing ramification at at most one prime. Indeed, if these relative Selmer groups are not equal we are done without the need for any extra ramification. Otherwise, running the same argument as in the proposition, relying on Lemma \ref{4.14} and choosing a prime as in Proposition \ref{4.16} we can guarantee that the dimension of the $2r$-th relative Selmer group decreases by allowing ramification at the new prime.
    
\end{rem}

Having annihilated a $2r$-th relative dual Selmer group, we can run the argument of Claim \ref{3.31} starting with the pair $(\rho_{n-2r},\rho_{n})$ to produce a characteristic $0$ lift $\rho':\Gamma_{F,S' \cup Q_0 \cup \{\nu_1,\nu_2\}} \to G(\mathcal O)$ of $\rho_{n-2r}$. This is possible because of the bounds we imposed on $n$ in (\ref{Nbound2}). We also imposed additional bounds on $n$ that as explained above will guarantee that $\rho'$ ramifies at all the primes at which $\rho$ ramified. We now show that $\rho'$ is ramified at $\nu_1$ and $\nu_2$. 

\begin{prop}

    \label{4.20}

    $\rho_{n+1}'$ is ramified at $\nu_1$ and $\nu_2$.
    
\end{prop}
\begin{proof}

    From the way we produced $\rho'$ we get that both $\rho$ and $\rho'$ are lifts of $\rho_{n-2r}$. In fact, looking at the proof of Claim \ref{3.31}, in particular by part $(2)$ of it, we see that $\rho'_{n-2r+1} = \rho_{n-2r+1}$. Therefore, we can write:

    $$\rho_{n+1} = \exp(\varpi^{n-2r+1}f)\rho'_{n+1}$$

    \vspace{2 mm}

    \noindent for some $f \in H^1(\Gamma_{F,S' \cup Q \cup \{\nu_1,\nu_2\}},\rho_{2r}(\mathfrak g^{\der}))$. Suppose that $\rho_{n+1}'$ isn't ramified at $\nu_i$ for $i \in \{1,2\}$ and let $j$ be the other index. We then have
    
    $$f \in H^1_{\mathcal L_{2r,S' \cup Q_0} \cup [L_{2r,\nu_j}^\alpha + L_{2r,\nu_j}^{\mathrm{unr}}]}(\Gamma_{F,S' \cup Q_0 \cup \nu_j},\rho_{2r}(\mathfrak g^{\der}))$$

    \vspace{2 mm}
    
    This follows because by our assumption neither $\rho_{n+1}$ nor $\rho'_{n+1}$ is ramified at $\nu_0$ or $\nu_i$, they satisfy the same local conditions at each prime in $S' \cup Q_0$, and $\restr{\rho_{n+1}}{\Gamma_{F_{\nu_j}}}$ is unramified, while $\restr{\rho_{n+1}'}{\Gamma_{F_{\nu_j}}} \in \Lift_{\restr{\rho_{n-2r+1}}{\Gamma_{F_{\nu_j}}}}^{\mu,\alpha}(\mathcal O/\varpi^{n+1})$. However, the Selmer group in question is equal to $ H^1_{\mathcal L_{2r,S' \cup Q_0} \cup [L_{2r,\nu_1}^\alpha \cap L_{2r,\nu_1}^{\mathrm{unr}}]}(\Gamma_{F,S' \cup Q_0},\rho_{2r}(\mathfrak g^{\der}))$ by Proposition \ref{4.17} if $j = 2$ or Proposition \ref{4.11} if $j = 1$. As $\restr{\rho_{n+1}'}{\Gamma_{F_{\nu_1}}} \in \Lift_{\restr{\rho_{n-2r+1}}{\Gamma_{F_{\nu_1}}}}^{\mu,\alpha}(\mathcal O/\varpi^{n+1})$ and the fibers of this class of lifts are preserved under the action of $L_{2r,\nu_1}^\alpha$ we have that $\restr{\rho_{n+1}}{\Gamma_{F_{\nu_1}}}  \in \Lift_{\restr{\rho_{n-2r+1}}{\Gamma_{F_{\nu_1}}}}^{\alpha,\mu}(\mathcal O/\varpi^{n+1})$, as well. But this contradicts the first bullet point of Proposition \ref{4.10}, which was used to produce the prime $\nu_1$.
    
\end{proof}

As $\rho'_{n+1}$ is ramified at $\nu_1$ and $\nu_2$ the same is true for its lift $\rho'$. Summarizing, we get that we can decrease the size of $Q_{\mathrm{unr}}$ at the cost of only guaranteeing the vanishing of a relative dual Selmer group modulo a higher power of $\varpi$. We then repeat this argument recursively, with $\rho'$ in place of $\rho$, $S' \cup Q_0 \cup \{\nu_1,\nu_2\}$ in place of $S' \cup Q$ and $2r$ in place of $r$, until we discard all the primes in $Q_{\mathrm{unr}}$. Eventually, we end up with a lift $\rho$ and a set of primes $Q$ such that $\rho$ is ramified at each of those primes, which by Lemma \ref{4.4} is enough to guarantee the local smoothness at each prime in $Q$. Hence, we have proven the following result:

\begin{thm}
\label{4.21}

    Let $p \gg_G 0$. Suppose $\overline{\rho}:\Gamma_{F,S} \to G(k)$ satisfies Assumptions \ref{1.1}. Then there exists a finite set $T \supseteq S$, a finite extension $\mathcal O'$ of $\mathcal O$ and a geometric lift $\rho:\Gamma_{F,T} \to G(\mathcal O')$ of $\overline{\rho}$ with multiplier type $\mu$ such that for each prime $\nu \in T$ the local restriction $\restr{\rho}{\Gamma_{F_\nu}}$ corresponds to a formally smooth point of
    
    \begin{itemize}
        \item the generic fiber of the local lifting ring $R_{\restr{\overline{\rho}}{\Gamma_{F_\nu}}}^{\sqr,\mu}[1/\varpi]$ for $\nu \nmid p$.
        \item the inertial type $\tau$ and Hodge type $\textup{\textbf{v}}$ lifting ring $R^{\sqr,\mu,\tau,\textup{\textbf{v}}}_{\restr{\overline{\rho}}{\Gamma_{F_\nu}}}[1/\varpi]$ (see \cite[Proposition $3.0.12$]{Bal12} for the construction of this ring) for $\nu \mid p$.
    \end{itemize}
    
\end{thm}

\subsection{Raising the level}

In the context when we know the representations produced by these lifting methods to be modular we can interpret Theorem \ref{4.21} as some sort of a level raising result. We will be able to produce lifts which correspond to automorphic representations of higher level and establish congruence relations between them. From this point of view, the trivial case, i.e. the case where we don't have to force any ramification requires a bit more attention. More precisely, when the initial set $Q$ is the empty set and the relative Selmer group is trivial after the doubling method, we don't have any extra ramification and subsequently we do not increase the level. Following the idea of \cite[Proposition $3.17$]{Pat17} we can always overcome this issue. 

\begin{prop}
\label{4.22}

The set $Q_{\mathrm{ram}}$ can be arbitrarily enlarged, while simultaneously keeping $Q_{\mathrm{unr}}$ empty. In particular, the set $T$ in Theorem \textup{\ref{4.21}} can be chosen to properly contain the set $S'$, hence the initial set $S$.
    
\end{prop}
\begin{proof}

    If $Q_{\mathrm{unr}} \neq \emptyset$ by running the forcing ramification argument we can remove one prime of it and replace it with two new primes where the new characteristic $0$ lift will be ramified, hence increasing the set $Q_{\mathrm{ram}}$. Thus, we only need to consider the case $Q_{\mathrm{unr}} = \emptyset$, i.e. $Q = Q_{\mathrm{ram}}$.

    Suppose that $\overline{H^1_{\mathcal{L}_{r,S'\cup Q}^\perp}(\Gamma_{F,S'\cup Q},\rho_r(\mathfrak g^{\der})^*)} = 0$ for some integer $r$ and we used this to produce the lift $\rho$. First, we extend the local conditions modulo $\varpi^{2r}$ as in Lemma \ref{4.5}. Let $n$ be a large enough integer, so that it satisfies the inequality (\ref{Nbound2}) and $\rho_n$ is ramified at each primes in $Q$. As in Lemma \ref{4.8} for any $\nu \in Q_{n,4r}$ disjoint from $S' \cup Q$ the Selmer group $\overline{H^1_{\mathcal{L}_{2r,S'\cup Q \cup \nu}^\perp}(\Gamma_{F,S'\cup Q \cup \nu},\rho_{2r}(\mathfrak g^{\der})^*)}$ has dimension at most $1$. If the dimension is $1$ by the same computations we have that $\overline{H^1_{\mathcal{L}_{r,S'\cup Q \cup \nu}^\perp}(\Gamma_{F,S'\cup Q \cup \nu},\rho_{r}(\mathfrak g^{\der})^*)}$ has dimension $1$, as well. Then, running the forcing ramification argument, skipping the removal of a prime, we can produce primes $\nu_1,\nu_2$ such that $\overline{H^1_{\mathcal{L}_{2r,S'\cup Q \cup \{\nu,\nu_1,\nu_2\}}^\perp}(\Gamma_{F,S'\cup Q \cup \{\nu,\nu_1,\nu_2\}},\rho_{2r}(\mathfrak g^{\der})^*)} = 0$ and we have a lift $\rho':\Gamma_{F,S' \cup Q \cup \{\nu,\nu_1,\nu_2\}} \to G(\mathcal O)$ of $\rho_{n-2r}$ which is ramified at all primes in $S' \cup Q \cup \{\nu_1,\nu_2\}$. If $\rho'$ is also ramified at $\nu$ then we have enlarged $Q_{\mathrm{ram}}$ with $\nu,\nu_1$ and $\nu_2$ and we are done. If $\rho'$ isn't ramified at $\nu$, the new set $Q_{\mathrm{unr}}$ is non-empty and we can run the forcing ramification argument again. During it we can either discard $\nu$ as in Lemma \ref{4.7} and just add $\nu_1,\nu_2$ to $Q_{\mathrm{ram}}$ or replace $\nu$ with two new primes $\nu_3,\nu_4$ and a new characteristic $0$ lift which will be ramified at them and therefore enlarge $Q_{\mathrm{ram}}$ by $\nu_1,\nu_2,\nu_3$ and $\nu_4$.

    Thus, we can assume that $\overline{H^1_{\mathcal{L}_{2r,S'\cup Q \cup \nu}^\perp}(\Gamma_{F,S'\cup Q \cup \nu},\rho_{2r}(\mathfrak g^{\der})^*)} = 0$ for all primes $\nu \in Q_{n,4r}$. Now, we choose a prime $\nu \in Q_{n,4r}$ such that $\restr{\rho_n}{\Gamma_{F_\nu}} \in \Lift_{\restr{\overline{\rho}}{\Gamma_{F_\nu}}}^{\alpha,\mu}(\mathcal O/\varpi^n)$, but $\restr{\rho_{n+1}}{\Gamma_{F_\nu}} \notin \Lift_{\restr{\overline{\rho}}{\Gamma_{F_\nu}}}^{\alpha,\mu}(\mathcal O/\varpi^{n+1})$ and $\restr{f}{\Gamma_{F_\nu}} \in L_{2r,\nu}^\alpha$ for each $f \in H_{\mathcal L_{2r,S' \cup Q}}^1(\Gamma_{F,S' \cup Q},\rho_{2r}(\mathfrak g^{\der}))$. This can be done exactly as in Proposition \ref{4.10}. By our assumption the $2r$-th relative dual Selmer group will still vanish after adding $\nu$ to the list of primes. Hence, we can produce a lift $\rho':\Gamma_{F,S'\cup Q \cup \nu} \to G(\mathcal O)$ of $\rho_{n-2r}$. We claim that $\rho'$ is ramified at $\nu$. For the sake of contradiction, suppose that $\rho'$ isn't ramified at $\nu$. As both $\rho_{n+1}$ and $\rho'_{n+1}$ are lifts of $\rho_{n-2r+1}$ we have that $\rho_{n+1} = \exp(\varpi^{n-2r+1}f)\rho'_{n+1}$ for some $f \in H_{\mathcal L_{2r,S' \cup Q}}^1(\Gamma_{F,S' \cup Q},\rho_{2r}(\mathfrak g^{\der}))$. By our choice of the prime $\nu$ we have that $\restr{f}{\Gamma_{F_\nu}} \in L_{2r,\nu}^\alpha$. Thus, as $\restr{\rho_{n+1}'}{\Gamma_{F_\nu}} \in \Lift_{\restr{\overline{\rho}}{\Gamma_{F_\nu}}}^{\alpha,\mu}(\mathcal O/\varpi^{n+1})$ the same will be true for $\restr{\rho_{n+1}}{\Gamma_{F_\nu}}$, which is a contradiction. Therefore, we have enlarged $Q_{\mathrm{ram}}$ by $\nu$
    
\end{proof}
    
\section{Finiteness of Selmer groups}

Knowing that we can produce a lift $\rho: \Gamma_{F,S' \cup Q} \to G(\mathcal O)$ of $\overline{\rho}$ as in Theorem \ref{4.21}, we can show that the analogous results to the ones in \cite[\S $4$]{KR03} hold. In particular, we can show the vanishing of certain Selmer groups associated to the adjoint representation of the lift $\rho$, which as explained in \S 1 is expected to follow from the Bloch-Kato conjecture. We can additionally get some information about the structure of the universal deformation ring of $\overline{\rho}$.

For notational convenience from now on all deformation and lifting rings of the local representation $\overline{\rho_\nu}$ will be labeled with a subscript $\nu$, while the deformation and lifting rings without a subscript will correspond to the global representation $\overline{\rho}$. Additionally, we will be dealing with lifts and deformation rings which are unramified away from $S' \cup Q$. So, to ease off the notation we will drop that label. Let $R^{\sqr,\mu}$ be the unramified away from $S' \cup Q$, with fixed multiplier $\mu$ universal lifting ring of $\overline{\rho}$. Moreover, by our assumptions $H^0(\Gamma_{F,S' \cup Q},\overline{\rho}(\mathfrak g)) = \mathfrak z_{\mathfrak g}$, so we can talk about $R^{\mu}$, the unramified away from $S' \cup Q$, with fixed multiplier $\mu$ universal deformation ring of $\overline{\rho}$. We will now define a suitable quotient of $R^{\mu}$ that corresponds to lifts lying in a certain irreducible component of the generic fiber of the local lifting rings. 

For each prime $\nu \mid p$ we choose an irreducible component of the generic fiber of the fixed inertial type $\tau$ and $p$-adic Hodge type \textbf{v} lifting ring $R_{\nu}^{\sqr,\mu,\tau,\textbf{v}}[1/\varpi]$ (see \cite[Proposition $3.0.12$]{Bal12} for the construction of this ring) that contains the points coming from the given local lifts $\rho_\nu$. We then let $\overline{R^{\sqr}_\nu}$ be the quotient ring that corresponds to the scheme-theoretic closure of this component in $\Spec(R_{\nu}^{\sqr,\mu})$. We note that these rings are non-zero since by construction the subschemes contain the local lifts $\rho_\nu$. Now, the restriction to each $\Gamma_{F_\nu}$ gives rise to a natural transformation of the lifting functors $\Lift_{\overline{\rho}}^\mu(-) \to \Lift_{\overline{\rho_\nu}}^\mu(-)$, so the choice of a universal lift of $\overline{\rho}$ gives us a morphism $R_\nu^{\sqr,\mu} \to R^{\sqr,\mu}$ for each $\nu \mid p$. Thus, if we set $R_{p}^{\sqr,\mu} \coloneqq \widehat{\otimes}_{\nu \mid p, \mathcal O} R_\nu^{\sqr,\mu}$ we get a morphism $R_{p}^{\sqr,\mu} \to R^{\sqr,\mu}$. We can then define $\overline{R^{\sqr}} \coloneqq R^{\sqr,\mu} \widehat{\otimes}_{R_{p}^{\sqr,\mu}} (\widehat{\otimes}_{\nu \mid p, \mathcal O} \overline{R_\nu^{\sqr}})$, which is non-zero as each $\overline{R^{\sqr}_\nu}$ is non-zero. 

Now, for any complete local Noetherian $\mathcal O$-algebra $R$ with residue field $k$ we define $D_p(R)$ to be the subset of $\Lift_{\overline{\rho}}^\mu(R)$ consisting of lifts such that the induced map $R^{\sqr,\mu} \to R$ factors through $\overline{R^{\sqr}}$. These are all the lifts of $\overline{\rho}$ whose restrictions to the local Galois groups lie in $\Spec(\overline{R^{\sqr}_\nu})$ for each $\nu \mid p$. By \cite[Lemma $3.4.1$]{BG19}, conjugation by elements in $\widehat{G}(O)$ fixes the irreducible components of the generic fiber. Thus, by \cite[Lemma $3.2$]{BLGHT11}, $D_p$ define a deformation problem in the sense of \cite[Definition $2.2$]{CHT08} (see discussion preceding \cite[Lemma $1.3.3$]{BLGGT14}). Therefore, we can talk about the quotient $\overline{R}$ of $R^\mu$ that corresponds to $\overline{R^{\sqr}}$. In particular, for an element in $\Def_{\overline{\rho}}^\mu(R)$ the induced map $R^\mu \to R$ will factor through $\overline{R}$ if and only if any choice of a representative is an element of $D_p(R)$.

Based on the construction of the lift $\rho$ and the choice of the irreducible component of the generic fibers we get that $\rho$ provides us with a point in $\Spec(\overline{R})$, and so we get an $\mathcal O$-algebra local homomorphism $\varphi:\overline{R} \to \mathcal O$, which is necessarily surjective. We label the prime ideal $\ker \varphi$ by $\mathfrak p$. We're interested in the cotangent space $\mathfrak p/\mathfrak p^2$ and in particular our goal is to prove the analogue to \cite[Theorem $15$]{KR03}, i.e. that this is a finite abelian group. To achieve this we will identify its dual $\Hom_\mathcal O(\mathfrak p/\mathfrak p^2,E/\mathcal O)$ with a direct limit of certain Selmer groups. In \cite[\S $4$]{KR03} the authors use the explicit description of the local conditions appearing in this direct limit of Selmer groups to show that it must be finite. Instead of this, in what follows we will use the formal smoothness of the lift $\rho$ at each prime in $S' \cup Q$ and the formal smoothness criterion from Proposition \ref{3.1} to show the finiteness of the direct limit. 

One of our main improvements to the result of \cite{KR03} is generalizing the forcing ramification argument to representations valued in reductive groups $G$ and not just in $\GL_2$. Another one is the ability to produce a lift $\rho$ which will be locally smooth at all primes in $S' \cup Q$. In comparison, Khare and Ramakrishna produce a lift which will be ramified at all the new primes, which translates to the lift being formally smooth only at the primes where we allow extra ramification. Because of this we can establish the finiteness of the cotangent space at $\rho$ of a larger quotient of the universal deformation ring $R^\mu$. For us this quotient will be the ring $\overline{R}$ which we have constructed above. From its definition we see that to lifts corresponding to points of $\overline{R}$ we do not impose any condition at primes in $S$ not lying over $p$. In \cite{KR03} the relevant quotient is a deformation ring that parametrizes minimal lifts at all primes in the original ramifying set $S$.

As mentioned above in order to show the finiteness of $\mathfrak p/\mathfrak p^2$ we will identify its dual with a direct limit of Selmer groups. We first define the local conditions at each prime $\nu \mid p$ that appear in these Selmer groups. These new local conditions will be slightly larger than the local conditions $L_{n,\nu}$ that appeared in the relative deformation method and were constructed in \cite[Proposition $4.7$]{FKP21}. The connection between these two sets of local conditions will be key in showing the finiteness of the cotangent space. 

For each $\nu \mid p$ let $\Omega_{\overline{R^{\sqr}_\nu}/\mathcal O}$ be the module of continuous differentials. This is the representing object of the functor $\mathrm{Der}_{\mathcal O}(\overline{R^{\sqr}_\nu},-)$ of continuous $\mathcal O$-linear derivations into finite $\overline{R^{\sqr}_\nu}$-modules. For any positive integer $n$ we have that $\mathrm{Der}_{\mathcal O}(\overline{R^{\sqr}_\nu},\mathcal O/\varpi^n)$ is naturally isomorphic to an $\mathcal O$-submodule of $Z^1(\Gamma_{F_{\nu}},\rho_n(\mathfrak g^{\der}))$. This identification is readily used in \cite[\S $4$]{FKP21} during the construction of the local conditions $L_{n,\nu}$. We work through the details of this identification, as we need to know it commutes with the multiplication by $\varpi$ maps. We have the following lemma (cf. \cite[\S 2]{KP25}):

\begin{lem}
\label{5.1}

    For each positive integer $n$, $\Hom_{\mathcal O}(\Omega_{\overline{R^{\sqr}_\nu}/\mathcal O} \otimes_{\overline{R^{\sqr}_\nu}} \mathcal O,\mathcal O/\varpi^n)$ can be identified with an $\mathcal O$-submodule of $Z^1(\Gamma_{F_\nu},\rho_n(\mathfrak g^{\der}))$ containing the coboundaries, where we view $\mathcal O$ as an $\overline{R^{\sqr}_\nu}$-module using the associated map $\overline{R^{\sqr}_\nu} \to \mathcal O$ coming from the local lift $\restr{\rho}{\Gamma_{F_\nu}}$. Moreover, this identification depends only on the reduction of $\restr{\rho}{\Gamma_{F_\nu}}$ modulo $\varpi^n$.
    
\end{lem}
\begin{proof}

    For any $m \ge n$ from the universal property of the ring $R^{\sqr,\mu}_\nu$ we have the isomorphism

    $$\Lift_{\restr{\rho_m}{\Gamma_{F_\nu}}}^\mu(\mathcal O/\varpi^{m+n}) \simeq \Hom_{\mathcal O}(R^{\sqr,\mu}_\nu,\mathcal O/\varpi^{m+n})_{\varphi_{m,\nu}}$$

    \vspace{2 mm}

    \noindent where the right-hand side stands for local $\mathcal O$-algebra homomorphisms $R_{\nu}^{\sqr,\mu} \to \mathcal O/\varpi^{m+n}$ that modulo $\varpi^m$ reduce to the morphism $\varphi_{m,\nu}:R_{\nu}^{\sqr,\mu} \to \mathcal O/\varpi^{m}$ coming from the lift $\restr{\rho_m}{\Gamma_{F_\nu}}$. The left-hand side is a $Z^1(\Gamma_\nu,\rho_n(\mathfrak g^{\der}))$-torsor under the action $\alpha \cdot \rho = \exp(\varpi^m\alpha)\rho$, while the right-hand side is a $\mathrm{Der}_{\mathcal O}(R_\nu^{\sqr,\mu},\varpi^m\mathcal O/\varpi^{m+n}\mathcal O)$-torsor via $\delta \cdot f = f + \delta$. Here, we remark that $\varpi^{m}\mathcal O/\varpi^{m+n}\mathcal O$ obtains its $R^{\sqr,\mu}_\nu$-module structure from the previously mentioned map coming from the local lift $\restr{\rho}{\Gamma_{F_\nu}}$. This structure will depend only on $\restr{\rho_n}{\Gamma_{F_\nu}}$. Thus, we can identify the $1$-cocycles and the derivations. We make this identification explicit. Let $\varphi_{m+n,\nu}:R^{\sqr,\mu}_\nu \to \mathcal O/\varpi^{m+n}$ be the map coming from $\restr{\rho_{m+n}}{\Gamma_{F_\nu}}$ and for each cocycle $\alpha$, let $\varphi_{m+n,\nu}^\alpha :R^{\sqr,\mu}_\nu \to \mathcal O/\varpi^{m+n}$ be the map corresponding to $\alpha \cdot \restr{\rho_{m+n}}{\Gamma_{F_\nu}}$. By the identification, there exists a unique derivation $\delta_\alpha$ such that $\varphi_{m+n,\nu}^\alpha = \delta_\alpha + \varphi_{m+n,\nu}$. Let $\rho_\nu^{\sqr,\mu}:\Gamma_{F_\nu} \to G(R_{\nu}^{\sqr,\mu})$ be the universal lift of $\restr{\overline{\rho}}{\Gamma_{F_\nu}}$. We then get:

    $$\delta_\alpha(\rho_\nu^{\sqr,\mu}) = (\exp(\varpi^m\alpha) - 1)\varphi_{m+n,\nu}(\rho_\nu^{\sqr,\mu})$$

    \vspace{2 mm}

    \noindent where we use a faithful finite-dimensional representation $G \hookrightarrow \GL_n \hookrightarrow \mathrm{M}_{n \times n}$, allowing us to work in a matrix ambient space. In particular, this allows us to apply the derivation $\delta_\alpha$ to the universal lift $\rho_\nu^{\sqr,\mu}$ and also to make sense of addition.

    Since $\exp(\varpi^m\alpha) - 1$ is a multiple of $\varpi^m$, $\delta_\alpha(\rho_\nu^{\sqr,\mu})$ will depend only on the reduction of $\varphi_{m+n,\nu}(\rho_{\nu}^{\sqr,\mu})$ modulo $\varpi^n$, which is exactly $\varphi_{n,\nu}(\rho_{\nu}^{\sqr,\mu}) = \restr{\rho_n}{\Gamma_{F_\nu}}$. Now, from the universality of $\rho_\nu^{\sqr,\mu}$ we get that $\delta_\alpha(\rho_\nu^{\sqr,\mu})$ uniquely determines $\delta_\alpha$. Indeed, if $\delta(\rho_\nu^{\sqr,\mu}) = \delta'(\rho_\nu^{\sqr,\mu})$ we then have that $(\phi_{m+n}+\delta)(\rho_\nu^{\sqr,\mu}) = (\phi_{m+n}+\delta')(\rho_\nu^{\sqr,\mu})$, which implies that $\phi_{m+n} + \delta = \phi_{m+n} + \delta'$ and therefore $\delta = \delta'$. Hence the identification depends on $\restr{\rho_n}{\Gamma_{F_\nu}}$ only.

    As $\overline{R^{\sqr}_\nu}$ is a quotient of $R_\nu^{\sqr,\mu}$ by a closed ideal $I_\nu$ we have an inclusion $\mathrm{Der}_{\mathcal O}(\overline{R^{\sqr}_\nu},\mathcal O/\varpi^n) \hookrightarrow \mathrm{Der}_{\mathcal O}(R_\nu^{\sqr,\mu},\mathcal O/\varpi^n)$ and so we can identify $\mathrm{Der}_{\mathcal O}(\overline{R^{\sqr}_\nu},\mathcal O/\varpi^n)$ with an $\mathcal O$-submodule of $Z^1(\Gamma_{F_\nu},\rho_n(\mathfrak g^{\der}))$, where we used the isomorphism $\varpi^{m}\mathcal O/\varpi^{m+n}\mathcal O \simeq \mathcal O/\varpi^n$, coming from our fixed choice of a uniformizer $\varpi$. Using the representability of the functor we have:

    $$\mathrm{Der}_{\mathcal O}(\overline{R^{\sqr}_\nu},\mathcal O/\varpi^n) = \Hom_{\overline{R_\nu^{\sqr}}}(\Omega_{\overline{R_\nu^{\sqr}}/\mathcal O},\mathcal O/\varpi^n) = \Hom_{\mathcal O}( \Omega_{\overline{R_\nu^{\sqr}}/\mathcal O} \otimes_{\overline{R_\nu^{\sqr}}} \mathcal O,\mathcal O/\varpi^n)$$

    \vspace{2 mm}

    \noindent which gives us the desired identification. It remains to show that under this identification this subspace contains the coboundaries. More precisely, we want to show that coboundaries correspond to $\mathcal O$-linear derivations on $R_\nu^{\sqr,\mu}$ that vanish on the ideal $I_\nu$. Let $\alpha \in Z^1(\Gamma_{F_\nu},\rho_n(\mathfrak g^{\der}))$ be a coboundary. Then $\varphi_{m+n,\nu}^\alpha(\rho_\nu^{\sqr,\mu})$ and $\restr{\rho_{m+n}}{\Gamma_{F_\nu}}$ are conjugates by an element in $\ker(G(\mathcal O/\varpi^{m+n}) \to G(\mathcal O/\varpi^m))$. We know that $\restr{\rho_{m+n}}{\Gamma_{F_\nu}}$ corresponds to a point on $\overline{R_\nu^{\sqr}}$ and since the irreducible components of the generic fiber are preserved under conjugation by element in $\widehat{G}(\mathcal O)$ the same is true for $\varphi_{m+n,\nu}^\alpha(\rho_\nu^{\sqr,\mu})$. In particular, this means that both $\varphi_{m+n,\nu}$ and $\varphi_{m+n,\nu}^\alpha$ factor through $\overline{R_\nu^{\sqr}}$ and are trivial on $I_\nu$. Hence, the same is true for $\delta_\alpha$. Thus, $\alpha$ corresponds to an element of $\mathrm{Der}_{\mathcal O}(\overline{R^{\sqr}_\nu},\mathcal O/\varpi^n)$.

\end{proof}

\begin{rem}
\label{5.2}

The importance of this identification depending only on $\restr{\rho_n}{\Gamma_{F_\nu}}$ is two-fold. On one hand, as we see below this will tell us that the identification commutes with the multiplication by $\varpi$ maps $\mathcal O/\varpi^n \to \mathcal O/\varpi^{n+1}$. On the other hand, this identification plays a key part in the definition of local conditions using \cite[Proposition $4.7$]{FKP21}. The local conditions at primes in $S'$ of the first type and in particular primes over $p$ will correspond to an $\mathcal O$-submodule of $\Hom_{\mathcal O}(\Omega_{\overline{R^{\sqr}_\nu}/\mathcal O} \otimes_{\overline{R^{\sqr}_\nu}} \mathcal O,\mathcal O/\varpi^n)$, where for $\nu \nmid p$ the ring $\overline{R^{\sqr}}$ is defined similarly as above by a choice of an irreducible component of $R_\nu^{\sqr,\mu}[1/\varpi]$. The lemma will be still valid for $\nu \nmid p$, hence even though we need a characteristic $0$ lift to define the $\varpi^n$ conditions at $\nu$, they will depend only on the reduction of that lift modulo $\varpi^n$.

\end{rem}

\begin{defn}

\label{5.3}
    
For each positive integer $n$ we define local conditions $\mathcal P_{n,p} = \{P_{n,\nu}\}_{\nu \mid p}$, where $P_{n,\nu}$ consists of the classes that land in the image of $\Hom_{\mathcal O}( \Omega_{\overline{R_\nu^{\sqr}}/\mathcal O} \otimes_{\overline{R_\nu^{\sqr}}} \mathcal O,\mathcal O/\varpi^n)$ inside of $H^1(\Gamma_{F_\nu},\rho_n(\mathfrak g^{\der}))$ under the identification of Lemma \ref{5.1}.
    
\end{defn}

Now, using these local conditions we define the Selmer groups $H^1_{\mathcal P_{n,p}}(\Gamma_{F,S'\cup Q},\rho_n(\mathfrak g^{\der}))$. We then have:

\begin{prop}
\label{5.4}

$\Hom(\mathfrak p/\mathfrak p^2,E/\mathcal O) \simeq \varinjlim_n H^1_{\mathcal P_{n,p}}(\Gamma_{F,S' \cup Q},\rho_n(\mathfrak g^{\der}))$, where all the maps in the direct limit are induced by the multiplication by $\varpi$ maps.

\end{prop}
\begin{proof}

    We first note that the direct limit on the right is well-defined. Since the local conditions $\mathcal P_{n,p}$ are the defined in terms of the K\"ahler differentials they interact nicely with the multiplication by $\varpi$ maps. In particular, the multiplication by $\varpi$ map sends $H^1_{\mathcal P_{n,p}}(\Gamma_{F,S' \cup Q},\rho_n(\mathfrak g^{\der}))$ to $H^1_{\mathcal P_{n+1,p}}(\Gamma_{F,S' \cup Q},\rho_{n+1}(\mathfrak g^{\der}))$, so the direct limit indeed makes sense. We will first show that for every positive integer $n$, we have
    
    $$\Hom(\mathfrak p/\mathfrak p^2,\mathcal O/\varpi^n) \simeq H^1_{\mathcal P_{n,p}}(\Gamma_{F,S' \cup Q},\rho_n(\mathfrak g^{\der}))$$

    \vspace{2 mm}

    As in Lemma \ref{5.1} we can identify $H^1_{\mathcal P_{n,p}}(\Gamma_{F,S' \cup Q},\rho_n(\mathfrak g^{\der}))$ with $\mathrm{Der}_{\mathcal O}(\overline{R},\mathcal O/\varpi^n)$. We explain this in more details. Let $\varphi_m:\overline{R} \to \mathcal O/\varpi^m$ be the map induced by the lift $\rho_m$. For $m \ge n$,  $\Hom_{\mathcal O}(\overline{R},\mathcal O/\varpi^{m+n})_{\varphi_m}$ is a $\mathrm{Der}_{\mathcal O}(\overline{R},\varpi^m\mathcal O/\varpi^{m+n}\mathcal O) = \mathrm{Der}_{\mathcal O}(\overline{R},\mathcal O/\varpi^n)$-torsor, where again the $\overline{R}$-module structure comes from the map $\varphi$ and in fact depends only on the reduction $\varphi_n$. On the other hand, $\Hom_{\mathcal O}(\overline{R},\mathcal O/\varpi^{m+n})_{\varphi_m}$ is naturally isomorphic to the subset of $\Def_{\rho_m}^\mu(\mathcal O/\varpi^{m+n})$ that corresponds to points of $\Spec(\overline{R})$. The latter is an $H^1_{\mathcal P_{n,p}}(\Gamma_{F,S' \cup Q},\rho_n(\mathfrak g^{\der}))$-torsor. Indeed, in general $\Def_{\rho_m}^\mu(\mathcal O/\varpi^{m+n})$ is a $H^1(\Gamma_{F,S' \cup Q},\rho_n(\mathfrak g^{\der}))$-torsor. By the definition of $\overline{R}$ its points will correspond to deformations whose representatives when restricted to $\Gamma_{F_\nu}$ are points in $\Spec(\overline{R_\nu^{\sqr}})$, for each $\nu \mid p$. This property is preserved by exactly the cocycles whose local restrictions land in $\mathcal P_{n,p}$ for each $\nu \mid p$.

    The restriction to $\mathfrak p$ gives us a bijection between $\mathrm{Der}_{\mathcal O}(\overline{R},\mathcal O/\varpi^n)$ and $\Hom_{\mathcal O}(\mathfrak p/\mathfrak p^2,\mathcal O/\varpi^n)$. Since the $\overline{R}$-module structure on $\mathcal O/\varpi^n$ is given using the map $\varphi_n$, which is trivial on $\mathfrak p$, the Leibnitz Rule tells us that any such derivation is trivial on $\mathfrak p^2$. The $\mathcal O$-linearity of the derivations guarantees that we indeed end up with a homomorphism $\mathfrak p/\mathfrak p^2 \to \mathcal O/\varpi^n$. As $\varphi$ is an $\mathcal O$-algebra homomorphism every element in $\overline{R}$ is of the form $a + x$, where $a \in \mathcal O$ and $x \in \mathfrak p$. Then, since $\mathcal O$-linear derivations are trivial on $\mathcal O$, they are completely determined by the values on $\mathfrak p$. Hence, the restriction to $\mathfrak p$ is injective. We can easily see that it is also surjective by simply extending every homomorphism $\mathfrak p/\mathfrak p^2 \to \mathcal O/\varpi^n$ to an $\mathcal O$-derivation on $\overline{R}$ by letting it vanish on elements of $\overline{R} - \mathfrak p$.

    Unraveling all these identifications we can make the isomorphism more explicit. Starting with $m \ge n$ and a cocycle $\alpha \in H^1_{\mathcal P_{n,p}}(\Gamma_{F,S' \cup Q},\rho_n(\mathfrak g^{\der}))$, we consider $\exp(\varpi^m\alpha)\rho_{m+n}$. By the universality of $\overline{R}$ we get a map $\varphi_{m+n}^\alpha:\overline{R} \to \mathcal O/\varpi^{m+n}$. As this is a lift of $\rho_m$ its reduction modulo $\varpi^m$ will be $\varphi_m$ and therefore under $\varphi_{m+n}^\alpha$, the ideal $\mathfrak p$ will have an image lying inside of $\varpi^{m}\mathcal O/\varpi^{m+n}\mathcal O$. As $m \ge n$, this map will vanish on $\mathfrak p^2$ and so restricting to $\mathfrak p$ we get an element of $\Hom_{\mathcal O}(\mathfrak p/\mathfrak p^2,\varpi^m\mathcal O/\varpi^{m+n}\mathcal O)$, which we can identify with $\Hom_{\mathcal O}(\mathfrak p/\mathfrak p^2,\mathcal O/\varpi^n)$. As remarked in Lemma \ref{5.1} this isomorphism will depend only on $\rho_n$ and not the integer $m$. This will allow us to conclude that the isomorphism is compatible with the multiplication by $\varpi$ maps. More precisely we have the following commuting diagram:

    \begin{center}
        \begin{tikzcd}
            {H^1_{\mathcal P_{n+1,p}}(\Gamma_{F,S' \cup Q},\rho_{n+1}(\mathfrak g^{\der}))} \arrow[rr, "\simeq"]       &  & {\Hom(\mathfrak p/\mathfrak p^2,\mathcal O/\varpi^{n+1})}       \\
            {H^1_{\mathcal P_{n,p}}(\Gamma_{F,S' \cup Q},\rho_{n}(\mathfrak g^{\der}))} \arrow[rr, "\simeq"] \arrow[u] &  & {\Hom(\mathfrak p/\mathfrak p^2,\mathcal O/\varpi^n)} \arrow[u]
        \end{tikzcd}
    \end{center}

    \vspace{2 mm}

    \noindent where the two vertical maps are induced by the multiplication by $\varpi$ map. Let $m \ge n+1$. The two horizontal isomorphism are independent of $m$, which allows us to use the same $m$ for both of them. Let $\alpha \in H^1_{\mathcal P_{n,p}}(\Gamma_{F,S' \cup Q},\rho_{n}(\mathfrak g^{\der}))$. The first vertical map will send it to $\varpi\alpha \in H^1_{\mathcal P_{n+1,p}}(\Gamma_{F,S' \cup Q},\rho_{n+1}(\mathfrak g^{\der}))$. Now, $\exp(\varpi^m(\varpi\alpha))\rho_{m+n+1}$ is a lift of $\rho_m$, so as above we get a map $\varphi_{m+n+1}^{\varpi\alpha}:\overline{R} \to \mathcal O/\varpi^{m+n+1}$, which gives us a homomorphism $\mathfrak p/\mathfrak p^2 \to \varpi^m\mathcal O/\varpi^{m+n+1}\mathcal O$, determining the image of $\varpi\alpha$ under the top isomorphism. On the other hand, $\exp(\varpi^{m+1}\alpha)\rho_{m+n+1}$ is a lift of $\rho_{m+1}$, so we get a map $\varphi_{m+n+1}^\alpha:\overline{R} \to \mathcal O/\varpi^{m+n+1}$, which gives us a homomorphism $\mathfrak p/\mathfrak p^2 \to \varpi^{m+1}\mathcal O/\varpi^{m+n+1}\mathcal O$, determining the image of $\alpha$ under the bottom isomorphism. Since they come from the same deformation $\varphi_{m+n+1}^{\varpi\alpha} = \varphi_{m+n+1}^\alpha$, so they determine the same homomorphism in $\Hom(\mathfrak p/\mathfrak p^2,\varpi^{m+1}\mathcal O/\varpi^{m+n+1}\mathcal O) \subseteq \Hom(\mathfrak p/\mathfrak p^2,\varpi^{m}\mathcal O/\varpi^{m+n+1}\mathcal O)$. After the identification of $\varpi^{m+1}\mathcal O/\varpi^{m+n+1}\mathcal O$ with $\mathcal O/\varpi^n$ and $\varpi^{m}\mathcal O/\varpi^{m+n+1}\mathcal O$ with $\mathcal O/\varpi^{n+1}$, this inclusion becomes exactly the second vertical map in the diagram. Therefore, the diagram commutes.

    Finally, we use this diagram to get an isomorphism on level of direct limits and we end with

    $$\Hom_{\mathcal O}(\mathfrak p/\mathfrak p^2,E/\mathcal O) = \Hom_{\mathcal O}(\mathfrak p/\mathfrak p^2,\varinjlim_n \mathcal O/\varpi^n) = \varinjlim_n \Hom_{\mathcal O}(\mathfrak p/\mathfrak p^2,\mathcal O/\varpi^n) \simeq \varinjlim_n H^1_{\mathcal P_{n,p}}(\Gamma_{F,S' \cup Q},\rho_{n}(\mathfrak g^{\der}))$$
     
\end{proof}

Having identified the dual of $\mathfrak p/\mathfrak p^2$ with the direct limit of Selmer groups we can use our explicit description of $\rho$ to show its finiteness. However, it turns out that working with the local conditions $\mathcal P_{n,p}$ is not quite the correct thing to do, so we will define new local conditions at primes over $p$. We recall that the local conditions 
$\mathcal P_{n,p}$ were defined using the cocycles in $\Hom_{\mathcal O}( \Omega_{\overline{R_\nu^{\sqr}}/\mathcal O} \otimes_{\overline{R_\nu^{\sqr}}} \mathcal O,\mathcal O/\varpi^n)$ under the identification of Lemma \ref{5.1}. We now define local conditions $\mathcal L_{n,p} = \{L_{n,\nu}\}_{\nu \mid p}$ using the cocycles which correspond to homomorphisms in $\Hom_{\mathcal O}( \Omega_{\overline{R_\nu^{\sqr}}/\mathcal O} \otimes_{\overline{R_\nu^{\sqr}}} \mathcal O,\mathcal O/\varpi^n)$ that are trivial on the torsion of $\Omega_{\overline{R_\nu^{\sqr}}/\mathcal O} \otimes_{\overline{R_\nu^{\sqr}}} \mathcal O$. By \cite[Lemma $4.5$]{FKP21} these cocycles will contain all the coboundaries, which means that the local conditions $\mathcal L_{n,p}$ are well-defined.

These local conditions will extend our previous local conditions $L_{n,\nu}$ which were defined earlier only for $n \le 2^dM$, hence using the same notation is justified. We extended the local conditions in Lemma \ref{4.5} in this manner and hence the same properties will be satisfied. Moreover, this is exactly how the local conditions $L_{n,\nu}$ for primes in $S'$ of the first two types, which we have left implicit in the lifting method are defined. (see \cite[Lemma $4.4$]{FKP21} for this explicit description). The only difference is the used charactersitic $0$ lift and hence the $\overline{R^{\sqr}_\nu}$-module structure of $\mathcal O$. As argued in Lemma \ref{4.5} these lifts will be equal modulo $\varpi^{2^dM}$, which by Remark \ref{5.2} is enough to conclude we can extend the old local conditions at $\nu \mid p$ to modulo any power of $\varpi$.

The next lemma tells us that the finiteness of $\varinjlim_n H^1_{\mathcal P_{n,p}}(\Gamma_{F,S' \cup Q},\rho_n(\mathfrak g^{\der}))$ will follow from the finiteness of $\varinjlim_n H^1_{\mathcal L_{n,p}}(\Gamma_{F,S' \cup Q},\rho_n(\mathfrak g^{\der}))$

\begin{lem}
\label{5.5}

    $\varinjlim_n \left(H^1_{\mathcal P_{n,p}}(\Gamma_{F,S' \cup Q},\rho_n(\mathfrak g^{\der}))\right)/\varinjlim_n \left(H^1_{\mathcal L_{n,p}}(\Gamma_{F,S' \cup Q},\rho_n(\mathfrak g^{\der}))\right)$ is a finite abelian group.

\end{lem}
\begin{proof}

    We first note that by definition $\mathcal L_{n,p} \subseteq \mathcal P_{n,p}$, so all the quotients make sense. Moreover, as both local conditions are compatible with the multiplication by $\varpi$ map so does the direct limit. 
    
    We consider the composition:

    $$H^1(\Gamma_{F,S'\cup Q},\rho_n(\mathfrak g^{\der})) \longrightarrow \prod_{\nu \mid p} H^1(\Gamma_{F_\nu},\rho_n(\mathfrak g^{\der}))/L_{n,\nu} \longrightarrow \prod_{\nu \mid p} H^1(\Gamma_{F_\nu},\rho_n(\mathfrak g^{\der}))/P_{n,\nu}$$ 

    \vspace{2 mm}

    Using the kernel-cokernel sequence we get an exact sequence:

    $$0 \longrightarrow H^1_{\mathcal L_{n,p}}(\Gamma_{F,S' \cup Q},\rho_n(\mathfrak g^{\der})) \longrightarrow H^1_{\mathcal P_{n,p}}(\Gamma_{F,S' \cup Q},\rho_n(\mathfrak g^{\der})) \longrightarrow \prod_{\nu \mid p} P_{n,\nu}/L_{n,\nu}$$

    \vspace{2 mm}

    As direct limits are exact in the category of modules, and they commute with finite products we get the exact sequence:

    $$0 \longrightarrow \varinjlim_n H^1_{\mathcal L_{n,p}}(\Gamma_{F,S' \cup Q},\rho_n(\mathfrak g^{\der})) \longrightarrow \varinjlim_n H^1_{\mathcal P_{n,p}}(\Gamma_{F,S' \cup Q},\rho_n(\mathfrak g^{\der})) \longrightarrow \prod_{\nu \mid p} \varinjlim_n \left(P_{n,\nu}/L_{n,\nu}\right)$$

    \vspace{2 mm}

    Therefore, to prove the lemma it suffices to show that $\varinjlim_n P_{n,\nu}/L_{n,\nu}$ is finite for each $\nu \mid p$. This will follow directly from the next lemma.

\end{proof}

\begin{lem}

    \label{5.6}

    For $\nu \mid p$ the quotients $P_{n,\nu}/L_{n,\nu}$ eventually stabilize as $n \to \infty$. In particular, $\varinjlim_n P_{n,\nu}/L_{n,\nu}$ is finite.
    
\end{lem}
\begin{proof}
    
    First we note that as both $P_{n,\nu}$ and $L_{n,\nu}$ contain the coboundaries it is enough to prove the claim on the level of cocycles. From the definition of $\mathcal L_{n,p}$ in order to understand these quotients we will need to study the torsion of $\Omega_{\overline{R_\nu^{\sqr}}/\mathcal O} \otimes_{\overline{R_\nu^{\sqr}}} \mathcal O$. 
    
    As $\overline{R^{\sqr}_\nu}$ is a Noetherian $\mathcal O$-algebra, $\Omega_{\overline{R_\nu^{\sqr}}/\mathcal O} \otimes_{\overline{R_\nu^{\sqr}}} \mathcal O$ is a finite $\mathcal O$-module. This is also true for its torsion so there exists an integer $r_\nu \ge 0$ such that the torsion of $\Omega_{\overline{R_\nu^{\sqr}}/\mathcal O} \otimes_{\overline{R_\nu^{\sqr}}} \mathcal O$ is annihilated by $\varpi^{r_\nu}$. In fact, by \cite[Lemma $4.1$]{FKP21} this integer $r_\nu$ doesn't depend on the lift $\rho$ we obtain via the lifting method, hence it doesn't depend on the $\overline{R^{\sqr}_\nu}$-module structure of $\mathcal O$. This means that for $n \ge r_\nu$ the image of the torsion of $\Omega_{\overline{R_\nu^{\sqr}}/\mathcal O} \otimes_{\overline{R_\nu^{\sqr}}} \mathcal O$ under any element in $\Hom_{\mathcal O}( \Omega_{\overline{R_\nu^{\sqr}}/\mathcal O} \otimes_{\overline{R_\nu^{\sqr}}} \mathcal O,\mathcal O/\varpi^n)$ will land in $\varpi^{n-r_\nu}\mathcal O/\varpi^n\mathcal O$. Therefore, for $n \ge r_\nu$, the map $\Hom_{\mathcal O}(\Omega_{\overline{R_\nu^{\sqr}}/\mathcal O} \otimes_{\overline{R_\nu^{\sqr}}} \mathcal O,\mathcal O/\varpi^n) \to \Hom_{\mathcal O}(\Omega_{\overline{R_\nu^{\sqr}}/\mathcal O} \otimes_{\overline{R_\nu^{\sqr}}} \mathcal O,\mathcal O/\varpi^{n+1})$ induced by the multiplication by $\varpi$ map yields an isomorphism modulo the homomorphisms which are trivial on the torsion of $\Omega_{\overline{R_\nu^{\sqr}}/\mathcal O} \otimes_{\overline{R_\nu^{\sqr}}} \mathcal O$. The injectivity follows directly from the injectivity of the multiplication by $\varpi$ map $\mathcal O/\varpi^n \to \mathcal O/\varpi^{n+1}$. For the surjectivity, let $f \in \Hom_{\mathcal O}(\Omega_{\overline{R_\nu^{\sqr}}/\mathcal O} \otimes_{\overline{R_\nu^{\sqr}}} \mathcal O,\mathcal O/\varpi^{n+1})$. As remarked above, the image of the torsion of $\Omega_{\overline{R_\nu^{\sqr}}/\mathcal O} \otimes_{\overline{R_\nu^{\sqr}}} \mathcal O$, under $f$ will land in $\varpi^{n+1-r_\nu}\mathcal O/\varpi^{n+1}\mathcal O$. Identifying this with $\varpi^{n-r_\nu}\mathcal O/\varpi^{n}\mathcal O$ we can define $f' \in \Hom_{\mathcal O}(\Omega_{\overline{R_\nu^{\sqr}}/\mathcal O} \otimes_{\overline{R_\nu^{\sqr}}} \mathcal O,\mathcal O/\varpi^n)$, which is trivial on the free part of $\Omega_{\overline{R_\nu^{\sqr}}/\mathcal O} \otimes_{\overline{R_\nu^{\sqr}}} \mathcal O$ and $\varpi f'$ acts the same as $f$ on its torsion. Thus, $\varpi f'$ and $f$ differ by an element that acts trivially on the torsion of $\Omega_{\overline{R_\nu^{\sqr}}/\mathcal O} \otimes_{\overline{R_\nu^{\sqr}}} \mathcal O$, which gives us the surjectivity. 

    From this we conclude that the limit $\varinjlim_n P_{n,\nu}/L_{n,\nu}$ eventually stabilizes and in particular is isomorphic to $P_{r_\nu,\nu}/L_{r_\nu,\nu}$. The representatives of this quotient when viewed as elements of $\Hom_{\mathcal O}(\Omega_{\overline{R_\nu^{\sqr}}/\mathcal O} \otimes_{\overline{R_\nu^{\sqr}}} \mathcal O,\mathcal O/\varpi^{r_\nu})$ are determined by the image on the torsion. As its torsion is annihilated by $\varpi^{r_\nu}$ the torsion is finite and so there are finitely many possibilities for its image in $\mathcal O/\varpi^{r_\nu}$. Thus, $\varinjlim_n P_{n,\nu}/L_{n,\nu} \simeq P_{r_\nu,\nu}/L_{r_\nu,\nu}$ is finite for each $\nu \mid p$.
    
\end{proof}

To show the finiteness of $\varinjlim_n H^1_{\mathcal L_{n,p}}(\Gamma_{F,S' \cup Q},\rho_n(\mathfrak g^{\der}))$ we will leverage the vanishing of the $2^dM$-th relative Selmer group we arranged during the forcing ramification argument. For this we will need to define local conditions $\mathcal L_{n,S' \cup Q}$ for any $n \ge 1$. Up to now we only defined them for $n \le 2^dM$, but they can be extended in an analogous way to Lemma \ref{4.5}. In fact, we've already did this for primes $\nu$ lying over $p$. In a similar way, using Remark \ref{5.2} we can extend the local conditions at primes in $S'$ of the first two types. At primes in $S'$ of the third type and primes in $Q$ the local conditions were defined using Lemma \ref{3.16}. At them we extend the local conditions by using reductions of $\rho$ modulo higher and higher powers of $\varpi$. Finally, the local conditions at primes in $S'$ of the fourth type were unramified and these can be easily defined for any $n$.

As $\overline{H^1_{\mathcal L_{2^dM,S' \cup Q}}(\Gamma_{F,S' \cup Q},\rho_{2^dM}(\mathfrak g^{\der}))} = 0$ we get that $\overline{H^1_{\mathcal L_{n,S' \cup Q}}(\Gamma_{F,S' \cup Q},\rho_{n}(\mathfrak g^{\der}))} = 0$ for any $n \ge 2^dM$. Therefore, by applying Lemma \ref{4.14} we get that for any $n \ge 2^dM$ the $(n,n-2^dM+1)$-relative Selmer group $H^1_{(n,n-2^dM+1)}(\Gamma_{F,S' \cup Q},\rho_n(\mathfrak g^{\der}))$ with local conditions $\mathcal L_{n,S \cup Q}$ is trivial. By the exact sequence from Lemma \ref{3.14} for any $n \ge 2^dM$ the map $H^1_{\mathcal L_{2^dM-1,S' \cup Q}}(\Gamma_{F,S' \cup Q},\rho_{2^dM-1}(\mathfrak g^{\der})) \to H^1_{\mathcal L_{n,S' \cup Q}}(\Gamma_{F,S' \cup Q},\rho_n(\mathfrak g^{\der}))$, which is induced by multiplication by $\varpi^{n-2^dM+1}$ is surjective. Thus, we can conclude that the direct limit $\varinjlim_n H^1_{\mathcal L_{n,S' \cup Q}}(\Gamma_{F,S' \cup Q},\rho_n(\mathfrak g^{\der}))$ can be identified with a quotient of $H^1_{\mathcal L_{2^dM-1,S' \cup Q}}(\Gamma_{F,S' \cup Q},\rho_{2^dM-1}(\mathfrak g^{\der}))$ and in particular is finite. 

Next, we consider the composition:

$$H^1(\Gamma_{F,S'\cup Q},\rho_n(\mathfrak g^{\der})) \longrightarrow \prod_{\nu \in S' \cup Q} H^1(\Gamma_{F_\nu},\rho_n(\mathfrak g^{\der}))/L_{n,\nu} \longrightarrow \prod_{\nu \mid p} H^1(\Gamma_{F_\nu},\rho_n(\mathfrak g^{\der}))/L_{n,\nu}$$

\vspace{2 mm}

As in Lemma \ref{5.5} using the kernel-cokernel sequence and taking direct limits we get the exact sequence:

$$0 \longrightarrow \varinjlim_n H^1_{\mathcal L_{n,S'\cup Q}}(\Gamma_{F,S'\cup Q},\rho_n(\mathfrak g^{\der})) \longrightarrow \varinjlim_n H^1_{\mathcal L_{n,p}}(\Gamma_{F,S'\cup Q},\rho_n(\mathfrak g^{\der})) \longrightarrow \prod_{\substack{v \in S' \cup Q \\ v \nmid p}} \varinjlim_n H^1(\Gamma_{F_\nu},\rho_n(\mathfrak g^{\der}))/L_{n,\nu}$$

\vspace{2 mm}

By the above discussion the first direct limit is finite, hence to show the finiteness of the the limit in the middle it is enough to show that $\varinjlim_n H^1(\Gamma_{F_\nu},\rho_n(\mathfrak g^{\der}))/L_{n,\nu}$ is finite for any $\nu \in S' \cup Q$ that doesn't divide $p$. This will follow from the formal smoothness of $\rho$ we have arranged at each such prime in \S $4$.

\begin{prop}
\label{5.7}

    For a prime $\nu \in S' \cup Q$ not dividing $p$ the size of the quotients $H^1(\Gamma_{F_\nu},\rho_n(\mathfrak g^{\der}))/L_{n,\nu}$ is uniformly bounded. In particular, $\varinjlim_n H^1(\Gamma_{F_\nu},\rho_n(\mathfrak g^{\der}))/L_{n,\nu}$ is a finite abelian group.
    
\end{prop}
\begin{proof}

    Since $\restr{\rho}{\Gamma_{F_\nu}}$ is a formally smooth point of the generic fiber of the local lifting ring $R_{\nu}^{\sqr,\mu}[1/\varpi]$ by Proposition \ref{3.1} we have that $\varprojlim_n H^0(\Gamma_{F_\nu},\rho_n(\mathfrak g^{\der})^*) = 0$. This means that for every $X \in H^0(\Gamma_{F_\nu},\rho_n(\mathfrak g^{\der})^*)$ we can find $m \ge n$ such that that $X$ is not in the image of the reduction  map $H^0(\Gamma_{F_\nu},\rho_m(\mathfrak g^{\der})^*) \to H^0(\Gamma_{F_\nu},\rho_n(\mathfrak g^{\der})^*)$. Doing this for every element in $H^0(\Gamma_{F_\nu},\overline{\rho}(\mathfrak g^{\der})^*)$ and taking the maximal of those integers we can find a positive integer $m$ such that the map $H^0(\Gamma_{F_\nu},\rho_m(\mathfrak g^{\der})^*) \to H^0(\Gamma_{F_\nu},\overline{\rho}(\mathfrak g^{\der})^*)$ is zero. This in fact implies that for any $n$, the map $H^0(\Gamma_{F_\nu},\rho_{m+n}(\mathfrak g^{\der})^*) \to H^0(\Gamma_{F_\nu},\rho_{n+1}(\mathfrak g^{\der})^*)$ is zero. This can be proven using the same idea as in the proof of Lemma \ref{4.14}, working with the $\Gamma_{F_\nu}$-invariants instead of Selmer groups. Indeed, the proof of the lemma relies on the exact sequences coming from Lemma \ref{3.14}, and in this case the analogous exact sequence comes from the short exact sequence $ 0 \to \rho_{n-1}(\mathfrak g^{\der})^* \to \rho_{n}(\mathfrak g^{\der})^* \to \overline{\rho}(\mathfrak g^{\der})^* \to 0$ and using the fact that taking $\Gamma_{F_\nu}$-invariants is a left-exact functor. This means that for $n \ge m$, $H^0(\Gamma_{F_\nu},\rho_n(\mathfrak g^{\der})^*)$ is contained in the $\varpi^{n-m}$ multiples inside of $\rho_n(\mathfrak g^{\der})^*$. In particular, the size of $H^0(\Gamma_{F_\nu},\rho_n(\mathfrak g^{\der})^*)$ can be uniformly bounded, namely by $|\rho_m(\mathfrak g^{\der})|$. Using Tate's local duality we can also uniformly bound the size of $H^2(\Gamma_{F_\nu},\rho_n(\mathfrak g^{\der}))$ for $n \ge m$. Now, by \cite[Chapter $1$, Lemma $2.9$]{Mil06} we have the following equality:

    $$|\,H^1(\Gamma_{F_\nu},\rho_n(\mathfrak g^{\der}))\,| = |\,H^0(\Gamma_{F_\nu},\rho_n(\mathfrak g^{\der}))\,| \cdot |\,H^2(\Gamma_{F_\nu},\rho_n(\mathfrak g^{\der}))\,|$$

    \vspace{2 mm}

    As shown in Lemma \ref{4.5} the local conditions $L_{n,\nu}$ have the same size as $H^0(\Gamma_{F_\nu},\rho_n(\mathfrak g^{\der}))$ and therefore we have that the quotient $H^1(\Gamma_{F_\nu},\rho_n(\mathfrak g^{\der}))/L_{n,\nu}$ and $H^2(\Gamma_{F_\nu},\rho_n(\mathfrak g^{\der}))$ have the same size. Since, the size of the latter can be uniformly bounded for each $n$ we get that the direct limit $\varinjlim_n H^1(\Gamma_{F_\nu},\rho_n(\mathfrak g^{\der}))/L_{n,\nu}$ is finite. Indeed, it's an easy algebra exercise that the direct limit of groups of bounded size will also be bounded by the same bound, and in particular is finite.
    
\end{proof}

\begin{rem}
\label{5.8}

For primes $\nu$ for which the local conditions $L_{n,\nu}$ were defined using Lemma \ref{3.16} we can explicitly compute the bound on the size of the direct limit. This is due to the fact that we know how the lift $\restr{\rho}{\Gamma_{F_\nu}}$ will look. More precisely, the image of the Frobenius under it will lie in a torus and the representation will be tamely ramified in a subgroup associated to a root space. Then, we can take $m$ that appeared in the proof above to be the largest integer at which $\restr{\rho}{\Gamma_{F_\nu}} \pmod{\varpi^m}$ is not ramified. We can then compute the uniform bound in terms of this $m$ and the root datum of $\mathfrak g^{\der}$ using the analogous computations to \cite[Lemma $16$]{KR03}.
    
\end{rem}

Summarizing, we have the following result:

\begin{thm}
\label{5.9}

    $\varinjlim_n H^1_{\mathcal P_{n,p}}(\Gamma_{F,S' \cup Q},\rho_n(\mathfrak g^{\der}))$ is a finite abelian $p$-power group.
    
\end{thm}

We now have the immediate corollary

\begin{cor}

\label{5.10}

The cotangent space $\mathfrak p/\mathfrak p^2$ is a finite $p$-abelian group.
    
\end{cor}
\begin{proof}

    The claim follows directly by combining the isomorphism from Proposition \ref{5.4} and the finiteness from Theorem \ref{5.9}.
    
\end{proof}

We can now prove some more corollaries which follow rather formally and whose proofs in \cite[\S $4$]{KR03} will work more or less verbatim. For the sake of completeness we list some of the results and their proofs here.

\begin{cor}
\label{5.11}

The prime ideal $\mathfrak p$ is minimal in $\overline{R}$ and therefore $\Spec(\overline{R})$ has an irreducible component that is isomorphic to $\Spec(\mathcal O)$.

\end{cor}
\begin{proof}

    Let $\overline{R}_{\mathfrak p}$ be the localization of $\overline{R}$ at $\mathfrak p$. We know that $\overline{R}/\mathfrak p \simeq \mathcal O$ and hence it has characteristic $0$. This means that $p \notin \mathfrak p$ and hence it is invertible in $\overline{R}_{\mathfrak p}$. On the other hand, by Corollary \ref{5.10} we have that $\mathfrak p/\mathfrak p^2$ is a finite $\overline{R}_{\mathfrak p}$-module and in fact it is $p$-power torsion. As $p$ is invertible in $\overline{R}_{\mathfrak p}$ this means that $\mathfrak p/\mathfrak p^2 = 0$, i.e. $\mathfrak p\overline{R}_{\mathfrak p} = \mathfrak p^2\overline{R}_{\mathfrak p}$. Now, as $\overline{R}_{\mathfrak p}$ is Noetherian, $\mathfrak p\overline{R}_{\mathfrak p}$ is a finitely generated $\overline{R}_{\mathfrak p}$-module and hence by Nakayama's Lemma we get $\mathfrak p\overline{R}_{\mathfrak p} =0$. Therefore, $\mathfrak p$ doesn't contain any proper prime ideal of $\overline{R}$, i.e. it is a minimal prime ideal in $\overline{R}$. Therefore, $\Spec(\overline{R})$ has an irreducible component isomorphic to $\Spec(\overline{R}/\mathfrak p) \simeq \Spec(\mathcal O)$.
    
\end{proof}

\begin{cor}
\label{5.12}

The $\mathcal O$-algebra homomorphism $\varphi:\overline{R} \to \mathcal O$ can not be a non-trivial limit of $\mathcal O$-algebra homomorphisms  $f_i:\overline{R} \to \mathcal O$.    
    
\end{cor}
\begin{proof}

    Suppose that $\varphi = \lim f_i$. We want to show that this limit will eventually be constant, i.e. $f_i = \varphi$ for $i$ large enough. By Corollary \ref{5.10} we have that $\mathfrak p/\mathfrak p^2$ is a finite abelian $p$-power group. Let $p^n$ be its exponent. As $\varphi = \lim f_i$, for all large enough $i$ we have that $\varphi_{en+1} = f_{i,en+1}$, where as before we use the subscript notation to indicate the reductions modulo powers of $\varpi$. We claim this implies that $\varphi = f_i$, which will finish the proof. Fix such large enough $i$, so that $\varphi_{en+1} = f_{i,en+1}$. If $\varphi \neq f_i$, we can find $m \ge en+1$ such that $\varphi_m = f_{i,m}$, but $\varphi_{m+1} \neq f_{i,m+1}$. Then as $\varphi_{2m}$ and $f_{i,2m}$ are both lifts of $\varphi_m$, as seen in the proof of Proposition \ref{5.4} the difference $\varphi_{2m} - f_{i,2m}$ gives us an $\mathcal O$-module homomophism $\mathfrak p/\mathfrak p^2 \to \varpi^m\mathcal O/\varpi^{2m}\mathcal O$. Since $\varphi$ and $f_i$ are $\mathcal O$-algebra homomorphisms they agree on elements of $\mathcal O$. Thus, as $\overline{R}/\mathfrak p \simeq \mathcal O$ and $\varphi_{m+1} \neq f_{i,m+1}$ we can find $a \in \mathfrak p$ such that $(\varphi_{2m} - f_{i,2m})(a)$ is a multiple of $\varpi^m$, but not of $\varpi^{m+1}$. In particular, this means that this image generates $\varpi^m\mathcal O/\varpi^{2m}\mathcal O$ as an $\mathcal O$-module. Therefore, $\varphi_{2m} - f_{i,2m}:\mathfrak p/\mathfrak p^2 \to \varpi^m\mathcal O/\varpi^{2m}\mathcal O$ is surjective. This is an obvious contradiction, since the target has exponent greater than or equal to $p^{n+1}$, which is larger than the exponent of the domain. 
    
\end{proof}

Using Theorem \ref{5.9} we will show the vanishing of the geometric Bloch-Kato Selmer group $H^1_g(\Gamma_{F,S' \cup Q},\rho(\mathfrak g^{\der}))$ attached to the adjoint representation of $\rho:\Gamma_{F,S' \cup Q} \to G(E)$, which is a result predicted by the Bloch-Kato conjecture. We first state the definition of the geometric Bloch-Kato Selmer group:

$$H_g^1(\Gamma_{F,S'\cup Q},\rho(\mathfrak g^{\der})) \coloneqq \ker\bigg(H^1(\Gamma_{F,S'\cup Q},\rho(\mathfrak g^{\der})) \to \prod_{\nu \mid p} H^1(\Gamma_{F_\nu},\rho(\mathfrak g^{\der}) \otimes_{\mathbb Q_p} B_{\dR})\bigg)$$

\vspace{2 mm}

We then have the following result:

\begin{thm}
\label{5.13}

    The geometric Bloch-Kato Selmer group $H^1_{g}(\Gamma_{F,S' \cup Q},\rho(\mathfrak g^{\der}))$ is trivial.
    
\end{thm}
\begin{proof}

    We will first reinterpret the result of Theorem \ref{5.9}. We recall that in Definition \ref{5.3} for primes $\nu \mid p$ the local conditions $P_{n,\nu}$ were defined as the image of the cocycles $Z_{n,\nu}$ in $H^1(\Gamma_{F_\nu},\rho_n(\mathfrak g^{\der}))$. These cocycles are compatible with the reduction maps $\mathcal O/\varpi^{n+1} \to \mathcal O/\varpi^n$, so we define the cocycles $Z_\nu \coloneqq (\varprojlim_n Z_{n,\nu}) \otimes_{\mathcal O} E$ inside $Z^1(\Gamma_{F_\nu},\rho(\mathfrak g^{\der}))$. We then define local conditions $P_\nu$ inside $H^1(\Gamma_{F_\nu},\rho(\mathfrak g^{\der})) = \varprojlim_n H^1(\Gamma_{F_\nu},\rho_n(\mathfrak g^{\der})) \otimes_{\mathcal O} E$. Setting $\mathcal P_p \coloneqq \{P_\nu\}_{\nu \mid p}$, we immediately conclude from the finiteness in Theorem \ref{5.9} that the Selmer group $H^1_{\mathcal P_p}(\Gamma_{F,S'\cup Q},\rho(\mathfrak g^{\der}))$ vanishes. Hence, we will be done if we can show that $\mathcal P_p$ and the local conditions in the geometric Bloch-Kato Selmer group are the same. We will do this by showing that for any $\nu \mid p$ the cocycles used to define each of these local conditions are canonically isomorphic to the tangent space of $R_{\nu}^{\sqr,\mu,\tau,\textbf{v}}[1/\varpi]$ at the point $x$, which is the point corresponding to $\restr{\rho}{\Gamma_{F_\nu}}$ under the universal property of the lifting ring.
    
    Fix $\nu \mid p$. By Lemma \ref{5.1} the set of cocycles $Z_{n,\nu}$ are isomorphic to $\Hom_{\mathcal O}( \Omega_{\overline{R_\nu^{\sqr}}/\mathcal O} \otimes_{\overline{R_\nu^{\sqr}}} \mathcal O,\mathcal O/\varpi^n)$. We then have the following series of isomorphisms:

    \begin{align*}
        Z_{\nu} &= (\varprojlim_n Z_{n,\nu}) \otimes_{\mathcal O} E \simeq \varprojlim_n \Hom_{\mathcal O}( \Omega_{\overline{R_\nu^{\sqr}}/\mathcal O} \otimes_{\overline{R_\nu^{\sqr}}} \mathcal O,\mathcal O/\varpi^n) \otimes_{\mathcal O} E \simeq \Hom_{\mathcal O}( \Omega_{\overline{R_\nu^{\sqr}}/\mathcal O} \otimes_{\overline{R_\nu^{\sqr}}} \mathcal O,\mathcal O) \otimes_{\mathcal O} E \\
        &\simeq \Hom_{\mathcal O}( \Omega_{\overline{R_\nu^{\sqr}}/\mathcal O} \otimes_{\overline{R_\nu^{\sqr}}} \mathcal O,E) \simeq \Hom_{\overline{R_\nu^{\sqr}}}(\Omega_{\overline{R_\nu^{\sqr}}/\mathcal O},E) \simeq \mathrm{Der}_\mathcal O(\overline{R_\nu^{\sqr}},E) \simeq \Hom_{\mathcal O-\mathrm{alg}}(\overline{R_\nu^{\sqr}},E[\varepsilon])_x \\[1.4ex]
        &\simeq \Hom_{E-\mathrm{alg}}(\overline{R_\nu^{\sqr}}[1/\varpi],E[\varepsilon])_x \simeq \Hom_{\mathrm{local} \, E-\mathrm{alg}}((\overline{R_\nu^{\sqr}}[1/\varpi])_{\mathfrak m_x},E[\varepsilon])
    \end{align*}

    \vspace{2 mm}

    All these isomorphisms follow directly from the universal properties of inverse limits, K\"ahler differentials and localizations, but we take our time to explain some of the notation. $E[\varepsilon]$ are the dual numbers over $E$, i.e. the quotient $E[\varepsilon]/(\varepsilon^2)$. On the other hand, the point $x$ will give us an $E$-algebra homomorphism $x:R_{\nu}^{\sqr,\mu,\tau,\textbf{v}}[1/\varpi] \to E$ with maximal ideal $\mathfrak m_x$. This homomorphism will factor through $\overline{R_\nu^{\sqr}}[1/\varpi]$, which corresponds to the irreducible component of $R_{\nu}^{\sqr,\mu,\tau,\textbf{v}}[1/\varpi]$ we chose at the beginning. We use this to give $E$ an $\overline{R_\nu^{\sqr}}$-module structure. Then, the subscript $x$ stands for homomorphisms which agree with $x$ modulo $\varepsilon$.

    The last object in the series of isomorphisms is canonically isomorphic to the tangent space of $\overline{R_\nu^{\sqr}}[1/\varpi]$ at $x$. As mentioned $\overline{R_\nu^{\sqr}}[1/\varpi]$ is an irreducible component of $R_{\nu}^{\sqr,\mu,\tau,\textbf{v}}[1/\varpi]$. By Theorem \ref{4.21} the point $x$ is a formally smooth point of $R_{\nu}^{\sqr,\mu,\tau,\textbf{v}}[1/\varpi]$. In particular $x$ is a regular point of $R_{\nu}^{\sqr,\mu,\tau,\textbf{v}}[1/\varpi]$, so by Lemma \ref{A.5} the cocycles $Z_{\nu}$ are canonically isomorphic to the tangent spaces of $R_{\nu}^{\sqr,\mu,\tau,\textbf{v}}[1/\varpi]$ at $x$.

    On the other hand, the local conditions in the geometric Bloch-Kato Selmer group are defined by the cocycles

    $$Z_g^1(\Gamma_{F_\nu},\rho(\mathfrak g^{\der})) \coloneqq \ker\bigg(Z^1(\Gamma_{F_\nu},\rho(\mathfrak g^{\der})) \to H^1(\Gamma_{F_\nu},\rho(\mathfrak g^{\der}) \otimes_{\mathbb Q_p} B_{\dR})\bigg)$$

    \vspace{2 mm}

    To show that this set of cocycles is canonically isomorphic to the tangent space of $R_{\nu}^{\sqr,\mu,\tau,\textbf{v}}[1/\varpi]$ at $x$ we first note that $\rho$ is a de Rham representation of inertial type $\tau$ and Hodge type $\textbf{v}$, where we write $\rho$ instead of $\restr{\rho}{\Gamma_{F_\nu}}$, leaving the restriction implicit in order to ease off the notation. We now fix a faithful finite-dimensional representation $\sigma: G \hookrightarrow \GL_n$. We then have that $\sigma \circ \rho: \Gamma_{F_\nu} \to \GL_n(E)$ is a de Rham representation of inertial type $\sigma \circ \tau$ and Hodge type $\sigma \circ \textbf{v}$ (see \cite[\S 6]{Bel16} for the exact interpretation of these). For $f \in Z^1(\Gamma_{F_\nu},\rho(\mathfrak g^{\der}))$ we consider the lift $\rho_f \coloneqq \exp(\varepsilon f)\rho : \Gamma_{F_\nu} \to G(E[\varepsilon])$ of $\rho$. We get a representation $\sigma \circ \rho_f: \Gamma_{F_\nu} \to \GL_n(E[\varepsilon])$, which will be a lift of $\sigma \circ \rho$ and will be equal to $(1+\varepsilon(\sigma f))(\sigma \circ \rho)$. Here $\sigma f$ is the image of $f$ under the map $Z^1(\Gamma_{F_\nu},\rho(\mathfrak g^{\der})) \to Z^1(\Gamma_{F_\nu},\ad^0(\sigma \circ \rho))$, which is induced by the map on Lie algebras coming from $\sigma$. In fact, from the definition of the exponential map given in Remark $\ref{3.6}$ the map on Lie algebras is exactly the map $\sigma$ on $E[\varepsilon]$-points which are trivial modulo $\varepsilon$. Then, by \cite[Lemma $1.2.5$]{All16} $\sigma f$ lies in $Z_g^1(\Gamma_{F_\nu},\ad^0(\sigma \circ \rho))$ if and only if $\sigma \circ \rho_f$ is a de Rham representation with inertial type $\sigma \circ \tau$ and Hodge type $\sigma \circ \textbf{v}$. This in turn translates to $\rho_f$ being a de Rham representation, which is necessarily with inertial type $\tau$ and Hodge type $\textbf{v}$, as the inertial type and the Hodge type are locally constant and $\rho_f$ is a deformation of $\rho$. Now, the commuting diagram below tells us that $\sigma f \in Z_g^1(\Gamma_{F_\nu},\ad^0(\sigma \circ \rho))$ if and only if $f \in Z_g^1(\Gamma_{F_\nu},\rho(\mathfrak g^{\der}))$.

    \begin{center}
        \begin{tikzcd}
            {Z^1(\Gamma_{F_\nu},\rho(\mathfrak g^{\der}))} \arrow[rr] \arrow[d, hook] &  & {H^1(\Gamma_{F_\nu},\rho(\mathfrak g^{\der}) \otimes_{\mathbb Q_p} B_{\dR})} \arrow[d, hook] \\
            {Z^1(\Gamma_{F_\nu},\ad^0(\sigma \circ \rho))} \arrow[rr]                 &  & {H^1(\Gamma_{F_\nu},\ad^0(\sigma \circ \rho) \otimes_{\mathbb Q_p} B_{\dR})}                
         \end{tikzcd}
    \end{center}

    \vspace{2 mm}

    The claim follows immediately as the cocycles $Z_g^1$ are just the kernels of the horizontal maps, but we still need to justify the injectivity of the vertical maps. As the map on Lie algebras comes from the map $\sigma$ on $E[\varepsilon]$-points it is injective and hence the map on $1$-cocycles is injective, as well. The injectivity of the second vertical maps follows from the fact that $\rho$ is de Rham. Indeed, this implies that $\ad^0(\sigma \circ \rho)$ is de Rham, too. As the functor $D_{\dR}$ is exact on the category of de Rham representations starting with an exact sequence $0 \to \rho(\mathfrak g^{\der}) \to \ad^0(\sigma \circ \rho) \to V \to 0$ of $\mathbb Q_p[\Gamma_{F_\nu}]$-modules we get an exact sequence 
    
    $$0 \to (\rho(\mathfrak g^{\der}) \otimes_{\mathbb Q_p} B_{\dR})^{\Gamma_{F_\nu}} \to (\ad^0(\sigma \circ \rho) \otimes_{\mathbb Q_p} B_{\dR})^{\Gamma_{F_\nu}} \to (V \otimes_{\mathbb Q_p} B_{\dR})^{\Gamma_{F_\nu}} \to 0$$

    \vspace{2 mm}

    \noindent which from the long exact sequence on cohomology implies that the second vertical map in the commuting diagram is injective. Finally, by \cite[Proposition $2.3.5$]{Kis09b} the ring $(R_{\nu}^{\sqr,\mu,\tau,\textbf{v}}[1/\varpi])\widehat{{}_x}$ represents lifts of $\rho$ which are with multiplier $\mu$ and de Rham with inertial type $\tau$ and Hodge type $\textbf{v}$. Therefore, from the equivalence we proved above $Z_g^1(\Gamma_{F,S'\cup Q},\rho(\mathfrak g^{\der}))$ is canonically isomorphic to the tangent space of $R_{\nu}^{\sqr,\mu,\tau,\textbf{v}}[1/\varpi]$ at $x$. From all this we conclude that 

    $$H^1_{g}(\Gamma_{F,S' \cup Q},\rho(\mathfrak g^{\der})) = H^1_{\mathcal P_p}(\Gamma_{F,S'\cup Q},\rho(\mathfrak g^{\der})) = 0$$

\end{proof}

\newpage

\appendix

\section{Lemmas}

\begin{lem}[{\cite[Lemma $2.1$]{FKP22}}]

    \label{A.1}

    Let $L/F$ be a finite Galois extension of fields, and let $M/F$ be an abelian extension. Then $\Gamma_F$ acts trivially on $\Gal(LM/L)$ via conjugation.
    
\end{lem}
\begin{proof}

    Let $\sigma \in \Gamma_F$ and $h \in \Gal(LM/L)$. We first recall that the action is given by $\sigma \cdot h = (\restr{\sigma}{LM}) h (\restr{\sigma^{-1}}{LM})$, where the conjugation happens inside $\Gal(LM/F)$. As $L/F$ is a Galois extension, $\sigma \cdot h$ is trivial on $L$ and this is a well-defined action. Now, we consider the injection $\Gal(LM/L) \hookrightarrow \Gal(M/F)$ given by restriction to $M$. Indeed this map is injective, since if $\restr{h}{M} = \restr{h'}{M}$, then as $\restr{h}{L} = \restr{h'}{L} = \mathrm{id}$ they must be the same on the composite $LM$. However, $\restr{(\sigma \cdot h)}{M} = (\restr{\sigma}{M})(\restr{h}{M})(\restr{\sigma^{-1}}{M}) = \restr{h}{M}$, since these are all elements of the abelian group $\Gal(M/F)$. Using the injectivity of the map we deduce that $\sigma \cdot h = h$, i.e. the $\Gamma_F$-action on $\Gal(LM/L)$ is trivial.
    
\end{proof}

We suppose that $G$ satisfies all the conditions mentioned in \S 2.

\begin{lem}

    \label{A.2}

    Suppose that $\mathfrak g \otimes_{\mathcal O} k$ is a simple $\mathbb F_p[G(k)]$-module under the adjoint action. For $m \ge n \ge 1$, let $G_{m,n}$ denote the kernel of the modulo $\varpi^n$ reduction $G(\mathcal O/\varpi^m) \to G(\mathcal O/\varpi^n)$. We then have:

    $$G_{m,n}^{\mathrm{ab}} \simeq \begin{cases} G_{m,n} &; \text{if } m \le 2n \\ G_{2n,n} &; \text{if } m > 2n \end{cases}$$
    
\end{lem}
\begin{proof}

    Let $m \le 2n$. Using the exponential map from Remark \ref{3.6} we have that $\mathfrak g \otimes_\mathcal O \varpi^n\mathcal O/\varpi^m\mathcal O \simeq G_{m,n}$. As this is an isomorphism of group we have that $G_{m,n}$ is abelian and the claim follows for $m \le 2n$.

    For $m \ge 2n$, as $G$ is smooth from the kernel-cokernel sequence we have that $G_{m,n}/G_{m,2n} \simeq G_{2n,n}$, which we know is abelian. Thus, $[G_{m,n},G_{m,n}] \subseteq G_{m,2n}$. It remains to prove the reverse inclusion. For a fixed $n$ we will prove this by induction on $m$. The base case $m = 2n$ is already proven. Now, let $m > 2n$ and suppose that $[G_{m-1,n},G_{m-1,n}] = G_{m-1,2n}$. For $g \in G_{m,2n}$ by the inductive hypothesis $g \pmod{\varpi^{m-1}}$ is a commutator in $G_{m-1,n}$. Hence, using the smoothness of $G$ we can find $g_1,g_2 \in G_{m,n}$ such that $g[g_1,g_2]\equiv 1 \pmod{\varpi^{m-1}}$. This means that $g[g_1,g_2] \in G_{m,m-1}$ and therefore it is enough to show that $G_{m,m-1} \subseteq [G_{m,n},G_{m,n}]$. We now consider $G_{m,m-1} \cap [G_{m,n},G_{m,n}]$. We claim this is an $\mathbb F_p[G(k)]$-submodule of $G_{m,m-1}$. By the exponential map $G_{m,m-1} \simeq \mathfrak g \otimes_{\mathcal O} \varpi^{m-1}\mathcal O/\varpi^m\mathcal O \simeq \mathfrak g \otimes_{\mathcal O} k$, which is an abelian group of exponent $p$. Thus, by being a subgroup of it $G_{m,m-1} \cap [G_{m,n},G_{m,n}]$ is an $\mathbb F_p$-vector space. The $G(k)$ action on $G_{m,m-1}$ is given by lifting to $G(\mathcal O/\varpi^m)$ and conjugating. Since $z[x,y]z^{-1} = [zxz^{-1},zyz^{-1}]$ we get that $G_{m,m-1} \cap [G_{m,n},G_{m,n}]$ is closed under conjugation by $G(k)$ and hence it is an $\mathbb F_p[G(k)]$-module.

    As mentioned in Remark \ref{3.6} the exponential isomorphism will be $G(k)$-equivariant. Therefore, the only $\mathbb F_p[G(k)]$-submodule of $G_{m,m-1}$ are the trivial one and $G_{m,m-1}$, itself. To show that $G_{m,m-1} \cap [G_{m,n},G_{m,n}]$ is equal to the latter it suffices to find a non-trivial element of it. Let $T$ be a split maximal torus in $G$ and $\alpha \in \Phi(G^0,T)$. Then, as in the proof of Theorem \ref{3.4} we can find $t \in T(\mathcal O/\varpi^m)$, which is trivial modulo $\varpi^{m-n-1}$, but $\alpha(t) \not\equiv 1 \pmod{\varpi^{m-n}}$. Now, let $Y_\alpha$ be an $\mathcal O$-basis element of the root space $\mathfrak g_\alpha$. We then have

    $$[t,u_\alpha(\varpi^nY_\alpha)] = tu_\alpha(\varpi^nY_\alpha)t^{-1}u_\alpha(-\varpi^nY_\alpha) = u_\alpha(\mathrm{Ad}_t(\varpi^nY_\alpha) - \varpi^nY_\alpha) = u_\alpha((\alpha(t)-1)\varpi^nY_\alpha)$$

    \vspace{2 mm}

    As $\alpha(t)-1$ is divisible by $\varpi^{m-n-1}$, but not by $\varpi^{m-n}$ this is a non-trivial element of $G_{m,m-1}$. $u_\alpha(\varpi^nY_\alpha)$ is trivial modulo $\varpi^n$ and as $m-n-1 \ge n$ the same is true about $t$. Therefore, $[t,u_\alpha(\varpi^nY_\alpha)]$ is a non-trivial element of $G_{m,m-1} \cap [G_{m,n},G_{m,n}]$. Therefore, as explained above this group is equal to $G_{m,m-1}$, which implies that $G_{m,m-1} \subseteq [G_{m,n},G_{m,n}]$.
     
\end{proof}

\begin{lem}
    \label{A.3}

    Suppose that $G$ is simple and connected. Let $T$ be a split maximal torus in $G$ and $\alpha \in \Phi(G,T)$. Given non-zero $A_1,\dots,A_r \in \mathfrak g^{\der} \otimes_ \mathcal O k$ and $B_1,\dots,B_s \in (\mathfrak g^{\der} \otimes_ \mathcal O k)^*$, for $p \gg_G 0$ (depending on $r$ and $s$, too) there exists $g \in G(k)$ such that $\Ad(g^{-1})A_i \notin \ker(\restr{\alpha}{\mathfrak t}) \oplus \bigoplus_\beta \mathfrak g_\beta$ and $\Ad(g^{-1})B_j \notin \mathfrak g_\alpha^\perp$ for all $i \le r$ and $j \le s$.
    
\end{lem}
\begin{proof}

    Let $\Phi_{A_i}$ be the proper closed subscheme of $G$ defined by the condition $\Ad(g)^{-1}A_i \in \ker(\restr{\alpha}{\mathfrak t}) \oplus \bigoplus_\beta \mathfrak g_\beta$. Similarly, we define $\Phi_{B_j}$. Let $Y$ be the union of all these subschemes. We want to show that the set of $k$-points of the complement $G \setminus Y$ is non-empty. By the Bruhat's decomposition (\cite[Corollary $14.14$]{Bor91}) $G$ contains an open subset $X$ that is isomorphic to the open subset $\mathbb G_m^{\dim(T)} \times \mathbb A^{|\Phi(G,T)|}$ of an affine space $\mathbb A^{d_G}$, where $d_G = \dim(G)$. Thus:
    
    $$|X(k)| \ge (q-1)^{\dim(T)} \cdot q^{|\Phi(G,T)|}$$

    \vspace{2 mm}

    \noindent where $q = |k|$. From the explicit description of $\Phi_{A_i}$ and $\Phi_{B_j}$ we get that $\Phi_{A_i} \cap X$ and $\Phi_{B_j} \cap X$ are cut out in $\mathbb A^{d_G}$ by equations of degree at most $d$, depending only on the root datum of $G$. We remark that the coordinates that are coming from the $\mathbb G_m$ factors might have negative powers, but after clearing the denominators we can make sense of them over $\mathbb A^{d_G}$. Now, any polynomial of degree $d$ in $a$ variables over $k$ will have at most $dq^{a-1}$ solutions. Indeed, we are free to choose any values for $a-1$ of the coordinates, reducing it to a polynomial in $1$ variable, which we know has no more solutions than its degree. From this we deduce that:

    $$|(X \cap Y)(k)| \le (r+s)dq^{d_G-1}$$

    \vspace{2 mm}

    Combining these two bounds we get:

    $$|(G \setminus Y)(k)| \ge (q-1)^{\dim(T)} \cdot q^{|\Phi(G,T)|} - (r+s)dq^{d_G-1} \ge \frac{1}{2^{\dim(T)}}q^{d_G} - (r+s)dq^{d_G-1}$$ 

    \vspace{2 mm}
    
    As the first term has the highest degree for $p \gg_{G} 0$ the right hand side will be positive. Hence the $k$-points of the complement are non-empty. 
    
\end{proof}

\begin{lem}
    \label{A.4}

    Let $R_1,R_2,\dots,R_n$ be prime characteristic rings such that $\mathrm{char}(R_i) > n$ for all $i$. For each $i$ let $M_i$ be an $R_i$-module and $M_i'$ an $R_i$-submodule of $M_i$. Suppose that for every $i$ there exists $x_i \in M_i$ such that $x_i \notin M_i'$. Then for any given $y_i \in M_i$ there exists an integer $a \le n$ such that $a \cdot x_i + y_i \notin M_i'$ for all $i$.
    
\end{lem}
\begin{proof} 

    For a fixed $i$ we consider the sum $a \cdot x_i + y_i$ for $a \in \{0,1,\dots,n\}$. We note that at most one of them can be an element of $M_i'$. Indeed, if $a \cdot x_i + y_i$, $a' \cdot y_i + x_i \in M_i'$ for distinct $a$ and $a'$ in $\{0,1,\dots,n\}$ we have that $(a-a') \cdot x_i \in M_i'$. But, $0 < |a-a'| \le n < \mathrm{char}(R_i)$. As the characteristic of $R_i$ is prime $a-a'$ is invertible in $R_i$, meaning that $x_i \in M_i'$, which is a contradiction. Therefore, as we have $n+1$ options for $a$ and $n$ rings $R_i$ for at least one value of $a \in \{0,1,\dots,n\}$ we have that $a \cdot x_i + y_i \notin M_i'$ for all $i$.
    
\end{proof}

\begin{lem}
\label{A.5}

Let $R$ be a Noetherian ring and the closed point $x \in \Spec(R)$ be regular. Let $I$ a minimal ideal in $R$ so that the irreducible component $\Spec(R/I)$ contains the point $x$. Then, the tangent space to $\Spec(R)$ at $x$ and the tangent space to $\Spec(R/I)$ at $x$ are canonically isomorphic.

\end{lem}
\begin{proof}

    Let $\mathfrak m$ be the maximal ideal of $R$ corresponding to the closed point $x$ and write $k = \kappa(x)$ for the residue field at $x$. Since $x \in \Spec(R)$ is a regular point we have that $R_\mathfrak m$ is a regular local ring and therefore $\dim_k \mathfrak m/\mathfrak m = \dim(R_{\mathfrak m})$. On the other hand, as $x \in \Spec(R/I)$ we have that $I \subseteq \mathfrak m$. Thus, $IR_{\mathfrak m}$ is also a minimal ideal in $R_\mathfrak m$. However, regular local rings are integral domains, so $IR_\mathfrak m = (0)$. Therefore, $R_{\mathfrak m} = R_{\mathfrak m}/IR_{\mathfrak m} \simeq (R/I)_\mathfrak m$. From all this we get:
    
    $$\dim_k \mathfrak m/\mathfrak m^2 = \dim(R_\mathfrak m) = \dim((R/I)_\mathfrak m) \le \dim_k (\mathfrak m/I)/((\mathfrak m^2 + I)/I)$$

    \vspace{2 mm}

    However, quotienting with $I$ we have a surjection $\mathfrak m/\mathfrak m^2 \to (\mathfrak m/I)/((\mathfrak m^2 + I)/I)$, which by the inequality above has to be an isomorphism of $k$-vector spaces. We then conclude that the tangent spaces to $\Spec(R)$ and $\Spec(R/I)$ at $x$ are canonically isomorphic.

\end{proof}

\bibliographystyle{alpha}
\bibliography{biblio}

\end{document}